\documentclass{amsart}
\usepackage[colorlinks=true, allcolors=blue]{hyperref}
\usepackage{amsmath}
\usepackage{paralist}
\usepackage{amssymb}
\usepackage{amsthm}
\usepackage{amscd}
\usepackage{graphicx,mathrsfs}
\usepackage{fontenc}
\usepackage{enumitem}

\usepackage[paperwidth=7in, paperheight=10in, margin=.75in]{geometry}
\numberwithin{equation}{section}

\newtheorem{Thm}{Theorem}[section]
\newtheorem{Lem}{Lemma}[section]
\newtheorem{Cor}{Corollary}[section]
\newtheorem{Prop}{Proposition}[section]

\newtheorem{Def}{Definition}[section]
\newtheorem{Rem}{Remark}[section]

\begin{document}
\sloppy
\allowdisplaybreaks[4]
\title[Well-posedness of the degenerate master equation]{Well-posedness of the mean field game master equation on Carnot tori}

\author{Yiming Jiang, Yawei Wei, Yiyun Yang}
\address{School of Mathematical Sciences and LPMC\\ Nankai University\\ Tianjin 300071 China}
\email{ymjiangnk@nankai.edu.cn}
\address{School of Mathematical Sciences and LPMC\\ Nankai University\\ Tianjin 300071 China}
\email{weiyawei@nankai.edu.cn}
\address{School of Mathematical Sciences\\ Nankai University\\ Tianjin 300071 China}
\email{1120210035@mail.nankai.edu.cn}

\begin{abstract}
We study the master equation for a second-order mean field game on Carnot tori, which means the generic player can move periodically only along admissible trajectories given by the family of vector fields generating the Carnot group. As examples of sub-Riemannian manifolds, Carnot groups represent a type of non-commutative groups characterized by stratified Lie algebra structures. In order to obtain the well-posedness of the master equation, we analyze the properties of its solution by investigating a degenerate mean field game system for which there exists an equivalent characterization with the master equation. The main part of this paper lies in leveraging the regularity properties of solutions to two classes of linear degenerate parabolic equations and a class of linear degenerate coupled systems to derive the existence of solutions to the master equation. The research in this paper is motivated by \cite{19CDLL,24MMM}.
\end{abstract}

\maketitle
\section{Introduction}
In this paper, we study a kind of non-linear infinite-dimensional partial differential equation (PDE in short) defined on the space of probability measures, called master equation as follows:
\begin{equation}\label{ME}
\begin{cases}
-\partial_t U(t, x, m)-\Delta_{\mathcal{X}} U(t, x, m)+H\left(x, D_{\mathcal{X}} U(t, x, m)\right)\\
\quad-\int_{\mathbb{T}_{\mathbb{G}}} \Delta_{\mathcal{X}}^y\frac{\delta U}{\delta m}(t, x, m, y) d m(y)\\
\quad+\int_{\mathbb{T}_{\mathbb{G}}} D_{\mathcal{X}}^y \frac{\delta U}{\delta m} (t, x, m, y) \cdot D_p H\left(y, D_{\mathcal{X}} U(t, y, m)\right) d m(y)\\
=F(x, m),\quad \text{in } [0, T] \times \mathbb{T}_{\mathbb{G}} \times \mathcal{P}\left(\mathbb{T}_{\mathbb{G}}\right), \\
U(T, x, m)=G(x, m),\quad \text{in } \mathbb{T}_{\mathbb{G}} \times \mathcal{P}\left(\mathbb{T}_{\mathbb{G}}\right).
\end{cases}
\end{equation}
Here, $\mathbb{T}_{\mathbb{G}}$ denotes the torus in the Carnot group $\mathbb{G}=\left(\mathbb{R}^n,\circ\right)$, namely $\mathbb{T}_{\mathbb{G}}:=\mathbb{G}/\mathbb{Z}^{n}$, and $\mathcal{X}=\{X_1,X_2,\ldots,X_{n_1}\},n_1< n$ denotes the Jacobian generators of $\mathbb{G}$, which is a typical family of vector fields with an anisotropic structure satisfying the H\"{o}rmander condition, i.e.
\begin{equation*}
\operatorname{rank}(\operatorname{Lie}\{X_1,\ldots,X_{n_1}\}(x))=n,\,\text{ for any }x\in\mathbb{R}^n,
\end{equation*}
where $\operatorname{Lie}\{X_1,\ldots,X_{n_1}\}(x)$ denotes the Lie algebra induced by the given vector fields. Note that for a smooth vector field $X$ defined on some $n$-dimensional domain $\Omega$, it can be expressed in two equivalent forms:
$$
X=\sum\limits_{j=1}^n c_{j}(x)\partial_{x_j} \text{ or } X=\left(c_{1}(x),c_{2}(x),\ldots,c_{n}(x)\right)^{\top},\,x\in\Omega,
$$
for some smooth functions $\{c_{j}(x)\}_{j=1,\ldots,n}$. For further details on H\"{o}rmander vector fields, Carnot groups and Carnot tori, see Section \ref{Sec_2}. In addition, $\mathcal{P}\left(\mathbb{T}_{\mathbb{G}}\right)$ is the set of Borel probability measures on $\mathbb{T}_{\mathbb{G}}$, $\frac{\delta U}{\delta m}$ is the first order derivative of $U$ with respect to the measure $m$ (see Definition \ref{derivative w.r.t. m}), and $H: \mathbb{T}_{\mathbb{G}} \times \mathbb{R}^{n_1} \to \mathbb{R}$ is given by $H(x,p):=\frac{1}{2}|p|^2$. For any function $f:\mathbb{T}_{\mathbb{G}} \to \mathbb{R}$, we define the subgradient and the hypoelliptic operator associated with $\mathcal{X}$ respectively as
$$
D_{\mathcal{X}}f:=(X_1f,\ldots,X_{n_1}f)^{\top},\quad\Delta_{\mathcal{X}}f:=\sum_{i=1}^{n_1}X_i^2f.
$$
As for any vector-valued function $g:\mathbb{T}_{\mathbb{G}} \to \mathbb{R}^{n_1}$, the corresponding divergence is defined as
$$
\operatorname{div}_{\mathcal{X}}g=\sum_{i=1}^{n_1}X_ig.
$$
For the sake of distinction, the differential operators $D_{\mathcal{X}}^y$ and $\Delta_{\mathcal{X}}^y$ with superscript $y$ denote the operators acting on the variable $y$.

Heuristically, equation \eqref{ME} is a Hamilton-Jacobi-Bellman (HJB in short) equation on the space of measures, arising from the limit problem of a differential game with finitely many indistinguishable players, in which the dynamic of player $i$, with $1\leq i \leq N$, is driven by the stochastic diffusion process defined on the Carnot torus $\mathbb{T}_{\mathbb{G}}$, which means each player must periodically follow the horizontal curves with respect to the family of vector fields $\mathcal{X}$ generating the Carnot group $\mathbb{G}$, i.e.
\begin{equation*}
\begin{cases}
d Z_{s}^i=\sum_{k=1}^{n_1}\alpha^{i,k}(s,Z_s^i)X_k(Z_s^i) ds+\sqrt{2}\sum_{k=1}^{n_1}X_k(Z_s^i)d B_{s}^{i,k}, \\
Z_{t}^i=x^{i}=(x^{i,1},\ldots,x^{i,n})^{\top}, \text{ in } \mathbb{T}_{\mathbb{G}},
\end{cases}
\end{equation*}
where $X_k(Z_s^i)d B_{s}^{i,k}$ is the Stratonovich-type, $Z_{s}^i=(Z_{s}^{i,1},\ldots,Z_{s}^{i,n})^{\top}$ represents the state of the $i$-th player, $\alpha^i=(\alpha^{i,1},\ldots,\alpha^{i,n_1})^{\top}$ is the control chosen from some admissible domain $A_i\subseteq\mathbb{R}^{n_1}$, $\{B_{s}^{i}=(B_s^{i,1},\ldots,B_s^{i,n_1})^{\top}\}_{i=1,\ldots,N}$ are independent $n_1$-dimensional Brownian motions with coordinate independence, and $x^{i}$ is the given initial condition. Let $\boldsymbol{Z}_s=(Z_{s}^{1},\ldots,Z_{s}^{N})$ denote the collection of all players' states and the notation $\boldsymbol{x}$ indicates a vector of $\left(\mathbb{T}_{\mathbb{G}}\right)^N$ defined by $\boldsymbol{x}:=(x^1,\ldots,x^N)$. The player $i$ choose his own strategy in order to minimize the cost function
\begin{equation*}
J^{N,i}(t, \boldsymbol{x}, (\alpha^i)_{i=1,\ldots,N})=\mathbb{E}\left[\int_{t}^{T} \left(L\left(Z_{s}^i, \alpha_{s}^i\right)+F^{N,i}\left(\boldsymbol{Z}_{s}\right)\right) d s+G^{N,i}\left(\boldsymbol{Z}_{T}\right)\right],
\end{equation*}
where the Lagrangian $L\left(Z_{s}^i, \alpha_{s}^i\right):=\frac{1}{2}|\alpha_{s}^i|^2 $, while $F^{N,i}$ and $G^{N,i}$ are respectively the running cost and final cost function associated with the $i$-th player. Due to the symmetry of the game, we can suppose that $F^{N,i}$ and $G^{N,i}$ take the form
\begin{equation*}
F^{N,i}(\boldsymbol{x}) = F(x^i,m_{\boldsymbol{x}}^{N,i}),\quad G^{N,i}(\boldsymbol{x}) = G(x^i,m_{\boldsymbol{x}}^{N,i}),
\end{equation*}
where $F,G: \mathbb{T}_{\mathbb{G}} \times \mathcal{P}(\mathbb{T}_{\mathbb{G}}) \to \mathbb{R}$, and $m_{\boldsymbol{x}}^{N,i}:=\frac{1}{N-1}\sum\limits_{j \neq i}\delta_{x^j}$ is the empirical measure of all players except player $i$.

Now, denoting $\left(v^{N,i}\right)_{i=1,\ldots,N}$ as the value functions of the players, we say that the controls $\left(\hat{\alpha}^i\right)_{i=1,\ldots,N}$ provide a Nash equilibrium if the following inequality holds for all controls $\left(\alpha^i\right)_{i=1,\ldots,N}$ and for all $i$,
\begin{equation*}
v^{N,i}(t,\boldsymbol{x}):=J^{N,i}\left(t,\boldsymbol{x},\left(\hat{\alpha}^i\right)_{i=1,\ldots,N}\right) \leq J^{N,i}\left(t,\boldsymbol{x},\alpha^i,\left(\hat{\alpha}^j\right)_{j \neq i}\right).
\end{equation*}
By using the It\^{o}'s formula and the dynamic programming principle, we can deduce the HJB equations that $\left(v^{N,i}\right)_{i=1,\ldots,N}$ satisfy, namely the Nash system as follows:
\begin{equation}\label{Nash system}
\begin{cases}
-\partial_{t} v^{N,i}(t, \boldsymbol{x})-\sum\limits_{j=1}^{N} \Delta_{\mathcal{X}}^{x^j}v^{N,i}(t, \boldsymbol{x})+H\left(x^i,D_{\mathcal{X}}^{x^i} v^{N,i}(t,\boldsymbol{x})\right)\\
\quad-\sum\limits_{j \neq i}D_pH\left(x^j,D_{\mathcal{X}}^{x^j} v^{N,j}(t,\boldsymbol{x})\right) \cdot D_{\mathcal{X}}^{x^j} v^{N,i}(t,\boldsymbol{x})=F^{N,i}(\boldsymbol{x}),\\
\quad \text {in }[0,T] \times \left(\mathbb{T}_{\mathbb{G}}\right)^N, i \in\{1, \ldots, N\},\\
v^{N,i}(T,\boldsymbol{x})=G^{N,i}(\boldsymbol{x}), \quad \text {in } \left(\mathbb{T}_{\mathbb{G}}\right)^N.
\end{cases}
\end{equation}
Note that the Hamiltonian of the system is the Fenchel conjugate of the Lagrangian $L$, namely
\begin{equation*}
H(x,p):=\sup_{a \in A_i} \left(-a \cdot p-\frac{1}{2}|a|^2\right)=\frac{1}{2}|p|^2.
\end{equation*}
This leads to the optimal feedback strategies as
\begin{equation*}
\left(\hat{\alpha}^i(t,\boldsymbol{x})=-D_pH\left(x^i,D_{\mathcal{X}}^{x^i} v^{N,i}(t,\boldsymbol{x})\right)\right)_{i=1,\ldots,N}.
\end{equation*}
Assume that \eqref{Nash system} has a unique solution $\left(v^{N,i}\right)_{i=1,\ldots,N}$. Then for each $i\in \{1,\ldots,N\}$, $v^{N,i}$ is symmetric with respect to all the permutations on $\{1,\ldots,N\}\setminus\{i\}$ and for $i \neq j$, the $x^i$ in $v^{N,i}$ plays the same role as $x^j$ plays in $v^{N,j}$. Therefore, it is reasonable to expect that, as $N \to +\infty$,
\begin{equation*}
v^{N,i}(t,\boldsymbol{x}) \simeq U(t,x^i,m_{\boldsymbol{x}}^{N,i}),
\end{equation*}
where $U$ maps from $[0,T]\times \mathbb{T}_{\mathbb{G}} \times \mathcal{P}(\mathbb{T}_{\mathbb{G}})$ to $\mathbb{R}$. Analogous to the method employed in \cite[Proposition 6.1.1]{19CDLL} for computing the relevant derivatives, we derive
\begin{equation*}
D_{\mathcal{X}}^{x^j} v^{N,i}(t,\boldsymbol{x}) \simeq
\begin{cases}
D_{\mathcal{X}}U\left(t,x^i,m_{\boldsymbol{x}}^{N,i}\right), & \mbox{if } j=i, \\
\frac{1}{N-1}D_{\mathcal{X}}^{x^j}\frac{\delta U}{\delta m}\left(t,x^i,m_{\boldsymbol{x}}^{N,i},x^j\right), & \mbox{otherwise},
\end{cases}
\end{equation*}
and
\begin{equation*}
\Delta_{\mathcal{X}}^{x^j} v^{N,i}(t,\boldsymbol{x}) \simeq
\begin{cases}
\Delta_{\mathcal{X}}U\left(t,x^i,m_{\boldsymbol{x}}^{N,i}\right), & \mbox{if } j=i,\\
\frac{1}{N-1}\Delta_{\mathcal{X}}^{x^j}\frac{\delta U}{\delta m}\left(t,x^i,m_{\boldsymbol{x}}^{N,i},x^j\right)\\
\quad+\left(\frac{1}{N-1}\right)^2 \operatorname{Tr}\left[D_{\mathcal{X}}^{x^j,x^j}\frac{\delta^2 U}{\delta m^2}\right]\left(t,x^i,m_{\boldsymbol{x}}^{N,i},x^j,x^j\right), & \mbox{otherwise}.
\end{cases}
\end{equation*}
Substituting the above relations into the Nash system \eqref{Nash system} and then letting $N \to +\infty$, we thus obtain the master equation of the form in \eqref{ME}.

The mean field game (MFG in short) theory is typically employed to solve the differential game problems involving infinitely many small and indistinguishable players. This theory was first introduced in 2006 by Lasry and Lions \cite{06LL_I,06LL_II,07LL,08LLG}. Around the same time, Huang, Caines and Malham\'{e} \cite{06HCM} also established similar definitions. The limit problem in MFG theory essentially reduces to the analysis of a coupled system of PDEs, which consists of a backward HJB equation satisfied by the value function $u$ of individual players and a forward Fokker-Planck-Kolmogorov (FPK in short) equation satisfied by the distribution law $m$ of the population. It was established by Lions in his lectures at Coll\`{e}ge de France \cite{Li} that there exists an equivalent characterization between the solutions to the MFG system and the master equation. Consequently, the study of Nash equilibria in a MFG can be reduced to the analysis of one unique equation, namely the master equation. It is worth mentioning that the form of the MFG system considered in this paper is as follows:
\begin{equation}\label{MFG system}
\begin{cases}
-\partial_t u-\Delta_{\mathcal{X}}u+H\left(x,D_{\mathcal{X}}u(t,x)\right)=F(x, m(t)), & \mbox{in } [t_0,T] \times \mathbb{T}_{\mathbb{G}}, \\
\partial_t m-\Delta_{\mathcal{X}} m-\operatorname{div}_{\mathcal{X}}\left(m D_pH\left(x,D_{\mathcal{X}}u\right)\right)=0, & \mbox{in } [t_0,T] \times \mathbb{T}_{\mathbb{G}}, \\
u(T,x)=G(x, m(T)),\quad m(t_0)=m_0, & \mbox{in } \mathbb{T}_{\mathbb{G}},
\end{cases}
\end{equation}
which is composed of two degenerate parabolic equations.

In the last decade, different papers have investigated the master equation and presented the most critical topics such as existence and uniqueness; regularity results and convergence problem. For instance, from a probabilistic perspective, Carmona and Delarue \cite{14CD} interpreted the master equation as a decoupled field of infinite-dimensional forward-backward stochastic differential equations. Chassagneux et al. \cite{14CCD} proved for the first time the well-posedness of the first-order master equation using a probabilistic approach. Moreover, Gangbo et al. \cite{22GMMZ} provided the first global in time well-posedness result in the case of non-separable displacement monotone Hamiltonians and non-degenerate idiosyncratic noise. From an analytical perspective, Bensoussan, Frehse and Yam \cite{15BFY,17BFY} recharacterized the master eqution as a system of PDEs on $L^{2}$ spaces. Furthermore, Gangbo and \'{S}wi\c{e}ch \cite{15GS} proved a small time existence for the master equation in the first-order MFG (i.e. MFG without idiosyncratic noise). Cardaliaguet et al. \cite{19CDLL} provided some general and well-known results on the well-posedness of classical solutions to both first and second-order master equations, corresponding to the nonlocal MFGs without and with common noise, respectively. Additionally, we can refer to \cite{23JR} for some new findings and clarifications concerning the results in \cite{19CDLL}. In recent years, Ricciardi \cite{21Ri} studied the well-posedness of the MFG master equation in a framework of Neumann boundary condition. In addition to the above studies on classical solutions, researches on weak solutions are likewise important. Bertucci \cite{21Be_F,21Be_C} presented the notion of monotone solution to master equations in the case of finite or continuous state space. Meanwhile, Cardaliaguet and Souganidis \cite{21CS} introduced a notion of weak solution to the MFG master equation without idiosyncratic noise. In this paper, we investigate the first-order master equation for a class of degenerate MFGs. Using the analytical approach, we establish an equivalent characterization between the master equation \eqref{ME} and the degenerate MFG system \eqref{MFG system}, and finally prove the well-posedness and regularity of the solution in a space-non-isotropic framework.

Compared with the case where PDEs satisfy the uniform ellipticity condition (i.e., non-degenerate), the researches on the degenerate MFG systems have emerged more recently. In the context of hypoelliptic MFGs, Dragoni and Feleqi \cite{18DF} studied the ergodic MFG systems with H\"{o}rmander diffusion, which is a class of systems of degenerate elliptic PDEs satisfying H\"{o}rmander condition. Later, Feleqi et al. \cite{20FGT} considered hypoelliptic MFG systems with quadratic Hamiltonians and proved the existence and uniqueness of the solution using the technique of Hopf-Cole transform. Furthermore, Mimikos-Stamatopoulos \cite{21Mi} considered the hypoelliptic MFG system with local coupling, driven by the FPK diffusion; Jiang et al. \cite{23JRWX} proved the global well-posedness of the hypoelliptic MFG systems driven by the Grushin diffusion. While Mannucci et al. \cite{24MMM} obtained the short-time existence of classical solutions to the MFG system defined on Carnot groups. As for more general degenerate MFGs, Cardaliaguet et al. \cite{15CGPT} proved existence and uniqueness of a suitably defined weak solution to the degenerate parabolic MFG system with local coupling. Further, Ferreira et al. \cite{21FGT} extended the existence of weak solutions to a wide class of time-dependent degenerate MFG systems. Moreover, Cardaliaguet et al. \cite{22CSS} built a new notion of probabilistically weak solutions for the MFG systems with common noise and degenerate idiosyncratic noise. However, researches on degenerate MFG master equations are fewer, and is limited to the totally degenerate case. For instance, 
Bansil et al. \cite{23BMM} constructed global in time classical solutions to degenerate MFG master equations without idiosyncratic noise. 
More recently, Bansil and M\'{e}sz\'{a}ros \cite{24BM} 
established new global well-posedness results for the associated master equations in the case of potentially degenerate idiosyncratic noise. In this paper, we investigate the non-totally degenerate MFG master equation, namely by addressing the periodic MFG problem on Carnot groups. Such groups are non-commutative groups with specific structures and also typical examples of sub-Riemannian manifolds.

\noindent
\textbf{Notations.} Throughout this paper, let $T > 0$ be a fixed finite time. Note that the positive constant, typically denoted by $C$, may take distinct values across different lines. By convention, we let $c>0$ denote a constant dependent on $\mathbb{G}$; when $c$ further depends on $f_1,\ldots,f_k$, we write it as $c_{f_1,\ldots,f_k}$.

Let $\alpha_{i}$ be the homogeneous degree of $X_{i}$ (in fact, $\alpha_{i}=1$ if $i\in \{1,\ldots,n_1\}$). For any $k\in\mathbb{N}$ and multi-index $I=\left(i_1,\ldots,i_k\right), i_j \in \{1,\ldots,n_1\},j\in\{1,\ldots,k\}$ with the length $|I|:=\sum_{j=1}^{n}\alpha_{i_j}$, we define
\begin{equation*}
X_I:=
\begin{cases}
X_{i_1} \cdots X_{i_k}, & \mbox{if } \dim(I)\geq1, \\
\mathrm{Id}, & \mbox{if } \dim(I)=0,
\end{cases}
\end{equation*}
where $\dim(I)$ is the dimension of vector $I$, i.e. $\dim(I)=k$.

For any interval $D \subseteq\mathbb{R}$ and $n$-dimensional open domain $\Omega$, we define the space
\begin{equation*}
B\left(D;C_{\mathcal{X}}^{k+\alpha}\left(\Omega\right)\right):=\left\{\phi:D\to C_{\mathcal{X}}^{k+\alpha}\left(\Omega\right)~\bigg|~ \sup_{t\in D}\|\phi(t,\cdot)\|_{C_{\mathcal{X}}^{k+\alpha}\left(\Omega\right)}<+\infty\right\}
\end{equation*}  
for any $k\in\mathbb{N}$, $\alpha\in(0,1]$, where $C_{\mathcal{X}}^{k+\alpha}$ is the non-isotropic H\"{o}rder space defined in \eqref{non-isotropic Holder spaces_Def} below, with the H\"{o}lder norm defined in \eqref{Holder norm}. Also, we define the function space
$$
C_{\mathcal{X}}^{1,2}(D\times\Omega):=\left\{u:D\times\Omega\to\mathbb{R}~|~\partial_{t}u, X_Iu\in C(D\times\Omega),\,\forall\,|I|\leq 2 \right\}.
$$

Denote $d_{cc}$ as the Carnot-Carath\'{e}odory distance induced by the vector fields $\mathcal{X}$ (see Definition \ref{Def_CC distance} below). For $x \in \Omega$, the $d_{cc}$-ball is then defined as
$$
B_\delta(x)=\{y \in \Omega:d_{cc}(x,y)<\delta\}.
$$

Let $\mathcal{P}\left(\Omega\right)$ be the set of Borel probability measures on $\Omega$, endowed with the Kantorovich-Rubinstein distance
$$
d_1\left(m, m'\right):=\sup_{[\phi]_{C_{\mathcal{X}}^{0+1}\left(\Omega\right)}\leq 1} \int_{\Omega} \phi(y) d\left(m-m'\right)(y),
$$
where the supremum is taken over all $d_{cc}$-Lipschitz continuous maps $\phi: \Omega \to \mathbb{R}$ with the Lipschitz constant bounded by $1$. It can be known from \cite[Section 5.1]{18CD} that $d_1$ is well-defined and this distance metricizes the weak convergence of measures. Also, we have the following equivalent definition:
$$
d_1\left(m, m'\right):=\inf_{\gamma \in \Pi\left(m, m'\right)} \left[\int_{\Omega \times \Omega} d_{cc}(x,y) d\gamma(x,y)\right],
$$
where $\Pi\left(m, m'\right)$ is the set of Borel joint probability measures on $\Omega \times \Omega$ such that the marginal probability measures are respectively $m$ and $m'$.

When the probability measure $m$ is absolutely continuous with respect to the Lebesgue measure, we reuse the notation $m$ to denote its density, namely, we write $m: \Omega \ni x \mapsto m(x) \in \mathbb{R}_{+}$. In addition, consider $(m(t))_{t \in[0, T]}$ as the flows of time dependent measures, with $m(t) \in \mathcal{P}\left(\Omega\right)$ for any $t \in[0, T]$. For each $t \in[0, T]$, if $m(t)$ is absolutely continuous with respect to the Lebesgue measure, we can identify $m(t)$ with its density and denote $m:[0, T] \times \Omega \ni(t, x) \mapsto m(t, x) \in \mathbb{R}_{+}$ as the collection of the densities.

Let us introduce the definition of the first order derivative of a function with respect to the measure.
\begin{Def}[see {\cite[Definition 2.2.1]{19CDLL}}]\label{derivative w.r.t. m}
Let $U:\mathcal{P}\left(\Omega\right)\to\mathbb{R}$. We say that $U$ is $C^1$ if there exists a continuous map $\frac{\delta U}{\delta m}: \mathcal{P}\left(\Omega\right) \times \Omega
\to \mathbb{R}$ such that, for any $m, m^{\prime} \in \mathcal{P}\left(\Omega\right)$,
$$
\lim _{s \to 0^{+}} \frac{U\left((1-s) m+s m^{\prime}\right)-U(m)}{s}=\int_{\Omega} \frac{\delta U}{\delta m}(m, y) d\left(m^{\prime}-m\right)(y) .
$$
\end{Def}
It is noteworthy that $\frac{\delta U}{\delta m}$ is defined up to an additive constant. To ensure uniqueness, we impose the following normalization condition:
$$
\int_{\Omega} \frac{\delta U}{\delta m}(m, y) d m(y)=0 .
$$
The relevant integral form is that, for any $m, m^{\prime} \in \mathcal{P}\left(\Omega\right)$,
\begin{equation*}
U \left(m^{\prime}\right)-U(m) =\int_0^1 \int_{\Omega} \frac{\delta U}{\delta m}\left((1-s) m+s m^{\prime}, y\right) d\left(m^{\prime}-m\right)(y) d s .
\end{equation*}

\noindent
\textbf{Assumptions.} Let $k\in\{2,3,\ldots\}$ and $\alpha\in(0,1)$. We make the following assumptions:
\begin{enumerate}[label=\bfseries{A\arabic*)}, leftmargin=*, itemindent=2.3em]
\item $F:\mathbb{T}_{\mathbb{G}} \times \mathcal{P}(\mathbb{T}_{\mathbb{G}}) \to \mathbb{R}$ satisfies, there exists a constant $c_F>0$ such that for any $m,m' \in \mathcal{P}\left(\mathbb{T}_{\mathbb{G}}\right)$ and any $\rho \in C_{\mathcal{X}}^{-(k-1+\alpha)}\left(\mathbb{T}_{\mathbb{G}}\right)$,
\begin{equation}\label{monotonicity_F}
\int_{\mathbb{T}_{\mathbb{G}}}\left(F(x, m)-F\left(x, m^{\prime}\right)\right) d\left(m-m^{\prime}\right)(x) \geq 0,
\end{equation}
\begin{equation}\label{property_F}
\left\langle \rho,\frac{\delta F}{\delta m}(\cdot,m)(\rho) \right\rangle \geq 0,
\end{equation}
where $\frac{\delta F}{\delta m}(x,m)(\rho):=\left\langle \rho,\frac{\delta F}{\delta m}(x,m,\cdot) \right\rangle$, and
\begin{equation*}
\sup _{m \in \mathcal{P}(\mathbb{T}_{\mathbb{G}})}\left(\|F(\cdot, m)\|_{C_{\mathcal{X}}^{k-2+\alpha}\left(\mathbb{T}_{\mathbb{G}}\right)} +\left\|\frac{\delta F}{\delta m}(\cdot,m,\cdot)\right\|_{C_{\mathcal{X},y}^{k-1+\alpha}\left(\mathbb{T}_{\mathbb{G}};C_{\mathcal{X},x}^{k-1+\alpha}\left(\mathbb{T}_{\mathbb{G}}\right)\right)}\right) +\operatorname{Lip}\left(\frac{\delta F}{\delta m}\right) \leq c_F,
\end{equation*}
where
\begin{equation*}
\operatorname{Lip}\left(\frac{\delta F}{\delta m}\right):=\sup _{\substack{m_1,m_2 \in \mathcal{P}(\mathbb{T}_{\mathbb{G}})\\ m_1 \neq m_2}}\frac{\left\|\frac{\delta F}{\delta m}\left(\cdot,m_1,\cdot\right)-\frac{\delta F}{\delta m}\left(\cdot,m_2,\cdot\right)\right\|_{C_{\mathcal{X},y}^{k-1+\alpha}\left(\mathbb{T}_{\mathbb{G}};C_{\mathcal{X},x}^{k-1+\alpha}\left(\mathbb{T}_{\mathbb{G}}\right)\right)}}{d_1\left(m_1, m_2\right)};
\end{equation*}
\label{assum1}
\item $G:\mathbb{T}_{\mathbb{G}} \times \mathcal{P}(\mathbb{T}_{\mathbb{G}}) \to \mathbb{R}$ satisfies, there exists a constant $c_G>0$ such that for any $m,m' \in \mathcal{P}\left(\mathbb{T}_{\mathbb{G}}\right)$ and any $\rho \in C_{\mathcal{X}}^{-(k-1+\alpha)}\left(\mathbb{T}_{\mathbb{G}}\right)$,
\begin{equation}\label{monotonicity_G}
\int_{\mathbb{T}_{\mathbb{G}}}\left(G(x, m)-G\left(x, m^{\prime}\right)\right) d\left(m-m^{\prime}\right)(x) \geq 0,
\end{equation}
\begin{equation}\label{property_G}
\left\langle \rho,\frac{\delta G}{\delta m}(\cdot,m)(\rho) \right\rangle \geq 0,
\end{equation}
where $\frac{\delta G}{\delta m}(x,m)(\rho):=\left\langle \rho,\frac{\delta G}{\delta m}(x,m,\cdot) \right\rangle$, and
$$
\sup_{m \in \mathcal{P}(\mathbb{T}_{\mathbb{G}})}\left(\|G(\cdot, m)\|_{C_{\mathcal{X}}^{k+\alpha}\left(\mathbb{T}_{\mathbb{G}}\right)} +\left\|\frac{\delta G}{\delta m}(\cdot,m,\cdot)\right\|_{C_{\mathcal{X},y}^{k-1+\alpha}\left(\mathbb{T}_{\mathbb{G}};C_{\mathcal{X},x}^{k+\alpha}\left(\mathbb{T}_{\mathbb{G}}\right)\right)}\right) +\operatorname{Lip}\left(\frac{\delta G}{\delta m}\right) \leq c_G,
$$
where
$$
\operatorname{Lip}\left(\frac{\delta G}{\delta m}\right):=\sup _{\substack{m_1,m_2 \in \mathcal{P}(\mathbb{T}_{\mathbb{G}})\\ m_1 \neq m_2}}\frac{\left\|\frac{\delta G}{\delta m}\left(\cdot,m_1,\cdot\right)-\frac{\delta G}{\delta m}\left(\cdot,m_2,\cdot\right)\right\|_{C_{\mathcal{X},y}^{k-1+\alpha}\left(\mathbb{T}_{\mathbb{G}};C_{\mathcal{X},x}^{k+\alpha}\left(\mathbb{T}_{\mathbb{G}}\right)\right)}}{d_1\left(m_1, m_2\right)}.
$$
Here, the definition of norm $\|\cdot\|_{C_{\mathcal{X},y}^{k_1+\alpha_1}\left(\mathbb{T}_{\mathbb{G}};C_{\mathcal{X},x}^{k_2+\alpha_2}\left(\mathbb{T}_{\mathbb{G}}\right)\right)}$ for any $k_1,k_2 \in \mathbb{N}$ and $\alpha_1,\alpha_2 \in(0,1]$ can be found in \eqref{Holder norm_two space variables} below.
\label{assum2}
\item $H:\mathbb{T}_{\mathbb{G}}\times\mathbb{R}^{n_1}\to\mathbb{R}$ is smooth, globally Lipschitz continuous, and satisfies the coercivity condition:
\begin{equation*}
c_{H}^{-1}I_{n_1 \times n_1} \leq D_{pp}^2 H(x,p)\leq c_{H}I_{n_1 \times n_1},\,(x,p)\in\mathbb{T}_{\mathbb{G}}\times\mathbb{R}^{n_1}
\end{equation*}
for some constant $c_{H}>0$.\label{assum3}
\end{enumerate}
\begin{Rem}
The assumptions \eqref{property_F} and \eqref{property_G} are respectively stronger than the monotonicity conditions \eqref{monotonicity_F} and \eqref{monotonicity_G}. In fact, the latter cannot imply the former, as we can find a counter-example in \cite[Remark 2.23]{23JR}.
\end{Rem}

Before giving the main results, let us state the concept of the solution to the master equation.
\begin{Def}\label{Def_ME sol.}
Let $\alpha\in(0,1)$. We say that a map $U:[0, T] \times \mathbb{T}_{\mathbb{G}} \times \mathcal{P}\left(\mathbb{T}_{\mathbb{G}}\right) \to \mathbb{R}$ is a solution to the master equation \eqref{ME} if
\begin{enumerate}[label=(\arabic*), leftmargin=*, itemindent=1.7em]
\item $U\in C\left([0,T]\times\mathbb{T}_{\mathbb{G}}\times \mathcal{P}(\mathbb{T}_{\mathbb{G}})\right)$ and $U(t,x,m)$ is of class $C^{1}$ in $t$ uniformly with respect to $(x,m)$ and class $C_{\mathcal{X}}^{2+\alpha}$ in $x$ uniformly with respect to $(t,m)$;
\item $U$ is of class $C^1$ in $m$, where the first-order derivative
    $$
    [0, T] \times \mathbb{T}_{\mathbb{G}} \times \mathcal{P}\left(\mathbb{T}_{\mathbb{G}}\right) \times \mathbb{T}_{\mathbb{G}} \ni (t, x, m, y)\mapsto\frac{\delta U}{\delta m}(t,x,m,y)
    $$
    is continuous in all its variables, and $\frac{\delta U}{\delta m}$ is of class $C_{\mathcal{X}}^1$ in $x$ and class $C_{\mathcal{X}}^2$ in $y$ with the derivatives being continuous in all the variables;
\item $U$ satisfies the master equation \eqref{ME}.
\end{enumerate}
\end{Def}
Under the assumptions outlined above, we intend to investigate the existence and uniqueness of the solution to the master equation \eqref{ME}. The following is the main result.
\begin{Thm}[Well-posedness of the master equation]\label{Thm_ME wellposed regularity}
Let assumptions \ref{assum1}-\ref{assum2} hold with $k\in\{3,4,\ldots\}$ and $H(x,p)=\frac{1}{2}\left|p\right|^2$. Then there exists a unique solution U to the master equation \eqref{ME} in the sense of Definition \ref{Def_ME sol.}.

Moreover, the derivative $\frac{\delta U}{\delta m}$ satisfies
\begin{equation*}
\sup_{(t,m) \in [0,T] \times \mathcal{P}(\mathbb{T}_{\mathbb{G}})}\left\|\frac{\delta U}{\delta m}\left(t, \cdot, m, \cdot\right)\right\|_{C_{\mathcal{X},y}^{k-1+\alpha}\left(\mathbb{T}_{\mathbb{G}};C_{\mathcal{X},x}^{k+\alpha}\left(\mathbb{T}_{\mathbb{G}}\right)\right)} \leq C,
\end{equation*}
and $\frac{\delta U}{\delta m}$ and its Lie derivatives with respect to $(x,y)$ are continuous on $[0,T] \times \mathbb{T}_{\mathbb{G}} \times \mathcal{P}\left(\mathbb{T}_{\mathbb{G}}\right) \times \mathbb{T}_{\mathbb{G}}$. Furthermore,
\begin{align*}
& \left\|U\left(t, \cdot, m_1\right)-U\left(t, \cdot, m_2\right)-\int_{\mathbb{T}_{\mathbb{G}}} \frac{\delta U}{\delta m}\left(t, \cdot, m_2, y\right) d\left(m_1-m_2\right)(y)\right\|_{C_\mathcal{X}^{k+\alpha}(\mathbb{T}_{\mathbb{G}})} \\
\leq & C d_{1}^{2}\left(m_1,m_2\right).
\end{align*}
Here, the constants $C>0$ depend on $\mathbb{G}$, $\alpha$, $k$, $T$, $c_{F}$ and $c_{G}$ only.
\end{Thm}
Let us outline the idea of the proof for the theorem. To prove uniqueness, we construct a solution to the MFG system \eqref{MFG system} from a solution to the master equation \eqref{ME}. The uniqueness of the solution to the master equation \eqref{ME} can then be implied by the uniqueness of the solution to the MFG system \eqref{MFG system} (see Proposition \ref {Prop_MFG system wellposed regularity} below). To prove existence, we conversely consider the unique solution $(u,m) \in C_{\mathcal{X}}^{1,2}\left([t_0,T]\times\mathbb{T}_{\mathbb{G}}\right) \times C\left([t_0,T];\mathcal{P}(\mathbb{T}_{\mathbb{G}})\right)$ to the MFG system \eqref{MFG system} with initial condition $m(t_0)=m_0$ for any $\left(t_0,m_0\right)\in[0,T]\times\mathcal{P}\left(\mathbb{T}_{\mathbb{G}}\right)$, then we define
\begin{equation}\label{U_Def}
U\left(t_0, x, m_0\right):=u\left(t_0, x\right).
\end{equation}
The existence can be obtained by proving that $U$ is a solution to the master equation \eqref{ME}. In order to do this, we will establish some regularity results of $U$ in Propositions \ref{Prop_MFG system wellposed regularity}-\ref{Prop_U C^1 w.r.t. m} that follow. More precisely, in Proposition \ref{Prop_MFG system wellposed regularity}, we have an insight into properties of the solution to the MFG system \eqref{MFG system}, which are uniformly in $(t_0,m_{0})$. Thus we can get the Schauder estimates of $U$ with respect to $x$ and the differentiability of $U$ in $x$. In addition, we also need to prove the Lipschitz property of $U$ with respect to the measure, which can be seen in Proposition \ref{Prop_Lip. ctn. of U}. Furthermore, we obtain the $C^1$ differentiability of $U$ with respect to the measure in Proposition \ref{Prop_U C^1 w.r.t. m}. 
In the following way, we give the statements and the idea of proof for each proposition.
\begin{Prop}[Regularity results of $U$ in $x$]\label{Prop_MFG system wellposed regularity}
Let assumptions \ref{assum1}-\ref{assum2} hold and $H(x,p)=\frac{1}{2}\left|p\right|^2$. Then, for any initial condition $\left(t_0, m_0\right) \in[0, T] \times \mathcal{P}\left(\mathbb{T}_{\mathbb{G}}\right)$, the MFG system (\ref{MFG system}) has a unique solution $(u, m) \in C_{\mathcal{X}}^{\frac{2+\alpha}{2},2+\alpha}\left([t_0,T]\times\mathbb{T}_{\mathbb{G}}\right) \times C^{\frac{1}{2}}\left([t_0,T];\mathcal{P}(\mathbb{T}_{\mathbb{G}})\right)$, satisfying
\begin{equation}\label{MFG_regularity}
\sup _{t_1 \neq t_2} \frac{d_1\left(m\left(t_1\right), m\left(t_2\right)\right)}{\left|t_1-t_2\right|^{\frac{1}{2}}}+\left\|u\right\|_{C^{\frac{\alpha}{2},k+\alpha}_{\mathcal{X}}\left([t_0,T]\times\mathbb{T}_{\mathbb{G}}\right)} +\left\|u\right\|_{C^{\frac{2+\alpha}{2},2+\alpha}_{\mathcal{X}}\left([t_0,T]\times\mathbb{T}_{\mathbb{G}}\right)} \leq C,
\end{equation}
where the constant $C>0$ does not depend on $\left(t_0, m_0\right)$. Here, for any $k\in\mathbb{N}$, the space $C_{\mathcal{X}}^{\frac{\alpha}{2},k+\alpha}$ is defined in \eqref{time-space-Holder space_Def}, and $C^{\frac{k+\alpha}{2},k+\alpha}_{\mathcal{X}}$ is defined in \eqref{P-Holder space_Def} below.

Moreover, if $m_0$ is absolutely continuous with respect to the Lebesgue measure and has a smooth positive density, then $m\in C_{\mathcal{X}}^{\frac{2+\alpha}{2},2+\alpha}\left([t_0,T]\times\mathbb{T}_{\mathbb{G}}\right)\cap C^{\frac{\alpha}{2},k+\alpha}_{\mathcal{X}}\left([t_0,T]\times\mathbb{T}_{\mathbb{G}}\right)$ and $m>0$.

Furthermore, the solution is stable: if $m_0^{i} \to m_0$ in $\mathcal{P}\left(\mathbb{T}_{\mathbb{G}}\right)$ as $i\to+\infty$, then the corresponding solutions $\left(u^{i},m^{i}\right) \to (u,m)$ in $C_{\mathcal{X}}^{1,2}\left([t_0,T]\times\mathbb{T}_{\mathbb{G}}\right) \times C\left([t_0,T];\mathcal{P}(\mathbb{T}_{\mathbb{G}})\right)$ as $i\to+\infty$.

Finally, let $U$ be the function defined in \eqref{U_Def}, then $U$ satisfies
$$
\sup _{t_0 \in[0, T]} \sup _{m_0 \in \mathcal{P}\left(\mathbb{T}_{\mathbb{G}}\right)}\|U(t_0, \cdot, m_0)\|_{C^{k+\alpha}_{\mathcal{X}}\left(\mathbb{T}_{\mathbb{G}}\right)} \leq C,
$$
where $C>0$ depends on $\mathbb{G}$, $\alpha$, $k$, $T$, $c_F$ and $c_G$ only. Moreover, $U(t_0,\cdot,m_0)\in C_{\mathcal{X}}^2\left(\mathbb{T}_{\mathbb{G}}\right)$ with $U$ and its Lie derivatives being continuous in $(t_0,x,m_0)$.
\end{Prop}
\begin{Prop}[Lipschitz continuity of $U$ in $m_0$]\label{Prop_Lip. ctn. of U}
Let assumptions \ref{assum1}-\ref{assum3} and the conclusions of Proposition \ref{Prop_MFG system wellposed regularity} hold, $t_0 \in[0, T]$, $m_0^1$, $m_0^2 \in \mathcal{P}\left(\mathbb{T}_{\mathbb{G}}\right)$, and $\left(u_1, m_1\right)$, $\left(u_2, m_2\right)$ be the solutions to the MFG system \eqref{MFG system} with initial conditions $\left(t_0, m_0^1\right)$ and $\left(t_0, m_0^2\right)$ respectively. Then
$$
\sup _{t \in[t_0, T]}\left\{d_1\left(m_1(t), m_2(t)\right)+\left\|u_1(t, \cdot)-u_2(t, \cdot)\right\|_{C_{\mathcal{X}}^{k+\alpha}(\mathbb{T}_{\mathbb{G}})}\right\} \leq C d_1\left(m_0^1, m_0^2\right),
$$
where the constant $C>0$ does not depend on $t_0$, $m_0^1$ and $m_0^2$.

Finally, let $U$ be the function defined in \eqref{U_Def}, then
\begin{equation}\label{U_Lip. in m}
\left\|U\left(t_0, \cdot, m_0^1\right)-U\left(t_0, \cdot, m_0^2\right)\right\|_{{C_{\mathcal{X}}^{k+\alpha}(\mathbb{T}_{\mathbb{G}})}} \leq C d_1\left(m_0^1, m_0^2\right),
\end{equation}
where the constant $C>0$ depends on $\mathbb{G}$, $\alpha$, $k$, $T$, $c_{F}$, $c_{G}$ and $H$ only.
\end{Prop}
To prove Proposition \ref{Prop_MFG system wellposed regularity}, we need Lemmas \ref{Lem_LHJB well-posedness}-\ref{Lem_general FPK wellposed regularity} below and then use the Schauder fixed point theorem. To prove Proposition \ref{Prop_Lip. ctn. of U}, we need Lemmas \ref{Lem_LHJB well-posedness}-\ref{Lem_LHJB Lipschitz} below, which are essential to get the estimates of $\left\|u_1(t, \cdot)-u_2(t, \cdot)\right\|_{2+\alpha}$ and $d_1\left(m_1(t), m_2(t)\right)$. Specifically, Lemmas \ref{Lem_LHJB wellposed regularity}-\ref{Lem_general FPK wellposed regularity} deal with two classes of linear parabolic equations on Carnot tori: the linearized HJB equation
\begin{equation}\label{LHJB}
\begin{cases}
-\partial_t z-\Delta_{\mathcal{X}} z+b(t,x) \cdot D_{\mathcal{X}} z=f(t, x),& \text {in }[0, T) \times \mathbb{T}_{\mathbb{G}}, \\
z(T, x)=z_T(x),& \text {in }\mathbb{T}_{\mathbb{G}}
\end{cases}
\end{equation}
and the FPK equation with the following general form:
\begin{equation}\label{general FPK}
\begin{cases}
\partial_t\rho-\Delta_\mathcal{X}\rho-\operatorname{div}_{\mathcal{X}}(\rho b)=\upsilon, & \text{ in } [0,T] \times \mathbb{T}_{\mathbb{G}}, \\
\rho(0)=\rho_{0}, & \text{ in } \mathbb{T}_{\mathbb{G}}.
\end{cases}
\end{equation}
Now we give the detailed statements of these lemmas as follows:
\begin{Lem}[see {\cite[Corollary 1.1]{25JWY} and \cite[Proposition 4.2]{JWY}}]\label{Lem_LHJB well-posedness}
Let $\alpha\in(0,1)$. For any $i,j\in\{1,2,\ldots,n_1\}$, assume $a_{i,j}(t,x)$, $b_i(t,x)$, $c(t,x)$, $f(t,x)$ are continuous on $[0,T]\times\mathbb{T}_{\mathbb{G}}$, satisfying $a_{i,j}\in C_{{\mathcal{X}}}^{\frac{\alpha}{2},\alpha} \left([0,T]\times\mathbb{T}_{\mathbb{G}}\right)$, $b_i$, $c$, $f\in B\left([0,T];C_{\mathcal{X}}^{\alpha}\left(\mathbb{T}_{\mathbb{G}}\right)\right)$, and $g(x)\in C_{\mathcal{X}}^{2+\alpha}\left(\mathbb{T}_{\mathbb{G}}\right)$. Then there exists a unique solution $z(t,x)\in C_{{\mathcal{X}}}^{1,2} \left([0,T]\times\mathbb{T}_{\mathbb{G}}\right)$ to the equation
\begin{equation*}\label{LHJB 2}
\begin{cases}
\mathbf{H} z(t,x)=f(t,x),& \text {in }[0, T] \times \mathbb{T}_{\mathbb{G}}, \\
z(0,x)=g(x),& \text {in }\mathbb{T}_{\mathbb{G}},
\end{cases}
\end{equation*}
where the operator
\begin{equation*}
\mathbf{H}:=\partial_t-\sum_{i,j=1}^{n_1}a_{i,j}(t,x)X_i X_j-\sum_{i=1}^{n_1}b_i(t,x)X_i-c(t,x).
\end{equation*}
Moreover, $z$ has the form
\begin{equation*}
z(t,x)=\int_{\mathbb{T}_{\mathbb{G}}}\sum_{k\in\mathbb{Z}^{n}}\Gamma(t,k\circ x;0,y)g(y)d y + \int_{0}^{t}\int_{\mathbb{T}_{\mathbb{G}}}\sum_{k\in\mathbb{Z}^{n}}\Gamma(t,k\circ x;s,y)f(s,y)d y d s
\end{equation*}
for any $(t,x)\in[0,T]\times\mathbb{T}_{\mathbb{G}}$. Here, $\Gamma:(\mathbb{R}\times\mathbb{G})\times(\mathbb{R}\times\mathbb{G})\to\mathbb{R}$ is the fundamental solution for $\mathbf{H}$ on $\mathbb{R}\times\mathbb{G}$ and $\Gamma\geq 0$.
\end{Lem}
\begin{Lem}[see {\cite[Theorem 1.1]{JWY}}]\label{Lem_LHJB wellposed regularity}
Let $k\in\mathbb{Z}_+$ and $\alpha\in(0,1)$. Assume $b(t,x)$ and $f(t,x)$ are continuous on $[0,T]\times\mathbb{T}_{\mathbb{G}}$, satisfying $b\in B\left([0,T];C_{\mathcal{X}}^{k-1+\alpha}\left(\mathbb{T}_{\mathbb{G}};\mathbb{R}^{n_1}\right)\right)$, $f\in B\left([0,T];C_{\mathcal{X}}^{k-1+\alpha}\left(\mathbb{T}_{\mathbb{G}}\right)\right)$ and $z_T(x)\in C_{\mathcal{X}}^{k+\alpha}\left(\mathbb{T}_{\mathbb{G}}\right)$. Then equation \eqref{LHJB} has a unique solution $z$ which belongs to $C_{\mathcal{X}}^{1,2} \left([0,T) \times \mathbb{T}_{\mathbb{G}} \right) \cap C([0, T] \times \mathbb{T}_{\mathbb{G}})$ and satisfies
\begin{equation}\label{LHJB_USE}
\sup_{t \in[0, T]}\|z(t, \cdot)\|_{C_{\mathcal{X}}^{k+\alpha}\left(\mathbb{T}_{\mathbb{G}}\right)} \leq C\left(\left\|z_T\right\|_{C_{\mathcal{X}}^{k+\alpha}\left(\mathbb{T}_{\mathbb{G}}\right)}+\sup_{t\in(0,T)}\|f(t, \cdot)\|_{C_{\mathcal{X}}^{k-1+\alpha}\left(\mathbb{T}_{\mathbb{G}}\right)}\right).
\end{equation}
In addition, for any constant $\epsilon \in (0,T)$, $z$ satisfies
\begin{equation}\label{LHJB_HSE 1}
\sup_{\substack{t \neq t'\\t,t' \in [0,T-\epsilon]}} \frac{\left\|z\left(t^{\prime}, \cdot\right)-z(t, \cdot)\right\|_{C_{\mathcal{X}}^{k+\alpha}\left(\mathbb{T}_{\mathbb{G}}\right)}}{\left|t^{\prime}-t\right|^{\frac{1}{2}}} \leq C\left(\epsilon^{-\frac{1}{2}}\left\|z_T\right\|_{C_{\mathcal{X}}^{k+\alpha}\left(\mathbb{T}_{\mathbb{G}}\right)}+\sup_{t \in(0, T)}\|f(t, \cdot)\|_{C_{\mathcal{X}}^{k-1+\alpha}\left(\mathbb{T}_{\mathbb{G}}\right)}\right).
\end{equation}
Moreover, if $z_T(x)\in C_{\mathcal{X}}^{k+1+\alpha}\left(\mathbb{T}_{\mathbb{G}}\right)$, then $z$ satisfies
\begin{equation}\label{LHJB_HSE 2}
\sup _{\substack{t \neq t'\\t,t' \in [0,T]}} \frac{\left\|z\left(t^{\prime}, \cdot\right)-z(t, \cdot)\right\|_{C_{\mathcal{X}}^{k+\alpha}\left(\mathbb{T}_{\mathbb{G}}\right)}}{\left|t^{\prime}-t\right|^{\frac{1}{2}}} \leq C\left(\left\|z_T\right\|_{C_{\mathcal{X}}^{k+1+\alpha}\left(\mathbb{T}_{\mathbb{G}}\right)}+\sup_{t \in(0, T)}\|f(t, \cdot)\|_{C_{\mathcal{X}}^{k-1+\alpha}\left(\mathbb{T}_{\mathbb{G}}\right)}\right).
\end{equation}
Here, the constants $C>0$ depend on $\mathbb{G}$, $\alpha$, $k$, $T$ and $\sup_{t\in(0,T)}\|b(t,\cdot)\|_{C_{\mathcal{X}}^{k-1+\alpha}\left(\mathbb{T}_{\mathbb{G}};\mathbb{R}^{n_1}\right)}$ only.
\end{Lem}
\begin{Lem}[see {\cite[Theorem 1.2]{JWY}}]\label{Lem_LHJB Lipschitz}
Assume $b(t,x)$ and $f(t,x)$ are continuous on $[0,T]\times\mathbb{T}_{\mathbb{G}}$, satisfying $b\in B\left([0,T];C_{\mathcal{X}}^{\alpha}\left(\mathbb{T}_{\mathbb{G}};\mathbb{R}^{n_1}\right)\right)$, $f\in B\left([0,T];C_{\mathcal{X}}^{\alpha}\left(\mathbb{T}_{\mathbb{G}}\right)\right)$ and $z_T(x)\in C_{\mathcal{X}}^{0+1}\left(\mathbb{T}_{\mathbb{G}}\right)$. Then there exists a unique solution $z\in C_{\mathcal{X}}^{1,2} \left([0, T) \times \mathbb{T}_{\mathbb{G}} \right) \cap C([0, T] \times \mathbb{T}_{\mathbb{G}})$ to the equation \eqref{LHJB}, satisfying
\begin{align}\label{LHJB_Holder&Lipschitz}
& \sup_{\substack{t \neq t'\\t,t' \in [0,T]}} \frac{\left\|z\left(t^{\prime}, \cdot\right)-z(t, \cdot)\right\|_{L^{\infty}\left(\mathbb{T}_{\mathbb{G}}\right)}}{\left|t^{\prime}-t\right|^{\frac{1}{2}}} +\sup_{t \in [0,T]}\left[z(t,\cdot)\right]_{C_{\mathcal{X}}^{0+1}\left(\mathbb{T}_{\mathbb{G}}\right)}\notag \\
\leq & C\left(\left\|z_T\right\|_{C_{\mathcal{X}}^{0+1}\left(\mathbb{T}_{\mathbb{G}}\right)} +\|f\|_{L^{\infty}\left((0,T)\times\mathbb{T}_{\mathbb{G}}\right)}\right),
\end{align}
where the constant $C>0$ depends on $\mathbb{G}$, $T$ and $\|b\|_{L^{\infty}\left((0,T)\times\mathbb{T}_{\mathbb{G}}\right)}$ only.
\end{Lem}
\begin{Def}[Weak solution of the FPK equation]\label{Def_general FPK weak sol.}
Let $k\in\mathbb{Z}_+$ and $\alpha\in(0,1)$. Assume $b \in C_{\mathcal{X}}^{\frac{\alpha}{2},k-1+\alpha}\left([0,T]\times\mathbb{T}_{\mathbb{G}};\mathbb{R}^{n_1}\right)$, $\upsilon \in L^1\left([0,T];C_{\mathcal{X}}^{-k}\left([0,1)^n\right)\cap C_{\mathcal{X}}^{-(k+\alpha)}(\mathbb{T}_{\mathbb{G}})\right)$ and $\rho_0 \in C_{\mathcal{X}}^{-k}\left([0,1)^n\right)\cap C_{\mathcal{X}}^{-(k+\alpha)}(\mathbb{T}_{\mathbb{G}})$. For a given function $\rho \in C\left([0,T];C_\mathcal{X}^{-(k+\alpha)}(\mathbb{T}_{\mathbb{G}})\right)$, if for all $f \in C\left([0,t]\times\mathbb{T}_{\mathbb{G}}\right)\cap B\left([0,t];C_\mathcal{X}^{k+\alpha}(\mathbb{T}_{\mathbb{G}})\right)$, $\xi \in C_\mathcal{X}^{k+\alpha}(\mathbb{T}_{\mathbb{G}})$ and the solution $z \in C_{\mathcal{X}}^{1,2}\left([0,t)\times\mathbb{T}_{\mathbb{G}}\right) \cap C\left([0,t]\times\mathbb{T}_{\mathbb{G}}\right)$ to the linear equation as follows
\begin{equation}\label{dual FPK}
\begin{cases}
-\partial_tz-\Delta_\mathcal{X}z+b\cdot D_{\mathcal{X}}z=f, & \text{ in } [0,t) \times \mathbb{T}_{\mathbb{G}}, \\
z(t)=\xi, & \text{ in } \mathbb{T}_{\mathbb{G}},
\end{cases}
\end{equation}
the following weak formulation holds true:
\begin{equation}\label{general FPK_weak formulation}
\left\langle\rho(t),\xi\right\rangle +\int_{0}^{t} \left\langle \rho(s),f(s,\cdot) \right\rangle ds =\left\langle\rho_0,z(0,\cdot)\right\rangle +\int_{0}^{t}\left\langle \upsilon(s),z(s,\cdot)\right\rangle ds,
\end{equation}
then we say that $\rho$ is a weak solution to equation \eqref{general FPK}. Here, $C_{\mathcal{X}}^{-k}\left([0,1)^n\right)$ and $C_{\mathcal{X}}^{-(k+\alpha)}(\mathbb{T}_{\mathbb{G}})$ are dual spaces of $C_{\mathcal{X}}^{k}\left([0,1)^n\right)$ and $C_{\mathcal{X}}^{k+\alpha}(\mathbb{T}_{\mathbb{G}})$ respectively, and $\left\langle \cdot,\cdot \right\rangle$ denotes the duality between $C_\mathcal{X}^{-(k+\alpha)}(\mathbb{T}_{\mathbb{G}})$ and $C_\mathcal{X}^{k+\alpha}(\mathbb{T}_{\mathbb{G}})$.
\end{Def}
\begin{Lem}[see {\cite[Theorem 1.3]{JWY}}]\label{Lem_general FPK wellposed regularity}
Let $k\in\mathbb{Z}_+$ and $\alpha\in(0,1)$. Assume $b \in C_{\mathcal{X}}^{\frac{\alpha}{2},k-1+\alpha}\left([0,T]\times\mathbb{T}_{\mathbb{G}};\mathbb{R}^{n_1}\right)$, $\upsilon \in L^1\left([0,T];C_{\mathcal{X}}^{-k}\left([0,1)^n\right)\cap C_{\mathcal{X}}^{-(k+\alpha)}(\mathbb{T}_{\mathbb{G}})\right)$ and $\rho_0 \in C_{\mathcal{X}}^{-k}\left([0,1)^n\right)\cap C_{\mathcal{X}}^{-(k+\alpha)}(\mathbb{T}_{\mathbb{G}})$. Then there exists a unique weak solution $\rho$ in the sense of Definition \ref{Def_general FPK weak sol.} to the equation \eqref{general FPK}, satisfying
\begin{equation}\label{general FPK sol. regularity}
\sup_{t \in [0,T]}\left\|\rho(t)\right\|_{C_{\mathcal{X}}^{-(k+\alpha)}(\mathbb{T}_{\mathbb{G}})} \leq C\left(\left\|\rho_0\right\|_{C_{\mathcal{X}}^{-(k+\alpha)}(\mathbb{T}_{\mathbb{G}})} + \left\|\upsilon\right\|_{L^1\left([0,T];C_{\mathcal{X}}^{-(k+\alpha)}(\mathbb{T}_{\mathbb{G}})\right)}\right),
\end{equation}
where the constant $C>0$ depends on $\mathbb{G}$, $\alpha$, $k$, $T$ and $\sup_{t\in(0,T)}\|b(t,\cdot)\|_{C_{\mathcal{X}}^{k-1+\alpha}\left(\mathbb{T}_{\mathbb{G}};\mathbb{R}^{n_1}\right)}$ only.

Moreover, the solution is stable: if $b^{i} \to b$ in $C_{\mathcal{X}}^{\frac{\alpha}{2},k-1+\alpha}\left([0,T] \times \mathbb{T}_{\mathbb{G}};\mathbb{R}^{n_1}\right)$, $\upsilon^{i} \to \upsilon$ in $L^1\left([0,T];C_{\mathcal{X}}^{-(k+\alpha)}(\mathbb{T}_{\mathbb{G}})\right)$ and $\rho_{0}^{i} \to \rho_{0}$ in $C_{\mathcal{X}}^{-(k+\alpha)}(\mathbb{T}_{\mathbb{G}})$ as $i\to +\infty$, with $\upsilon^{i}\in L^1\left([0,T];C_{\mathcal{X}}^{-k}\left([0,1)^n\right)\right)$ and $\rho_{0}^{i}\in C_{\mathcal{X}}^{-k}\left([0,1)^n\right)$, then, calling $\rho^{i}$ and $\rho$ the solutions related to $\left(\rho_{0}^{i},b^{i},\upsilon^{i}\right)$ and $\left(\rho_{0},b,\upsilon\right)$ respectively, we have $\rho^{i} \to \rho$ in $C\left([0,T];C_{\mathcal{X}}^{-(k+\alpha)}(\mathbb{T}_{\mathbb{G}})\right)$ as $i\to +\infty$.
\end{Lem}
With Propositions \ref{Prop_MFG system wellposed regularity}-\ref{Prop_Lip. ctn. of U} in place, we can proceed to investigate Proposition \ref{Prop_U C^1 w.r.t. m}:
\begin{Prop}[$C^1$ differentiability of $U$ in $m_0$]\label{Prop_U C^1 w.r.t. m}
Let assumptions \ref{assum1}-\ref{assum3} and the conclusions of Proposition \ref{Prop_MFG system wellposed regularity} hold. Fix any $t_{0} \in [0, T]$ and $m_{0}, \hat{m}_{0} \in \mathcal{P}\left(\mathbb{T}_{\mathbb{G}}\right)$, let $(u,m)$ and $(\hat{u}, \hat{m})$ be the solutions to the MFG system \eqref{MFG system} with initial conditions $\left(t_{0}, m_{0}\right)$ and $\left(t_{0},\hat{m}_{0}\right)$ respectively, and $(z,\rho)$ be the solution to the system \eqref{linear MFG system} with initial condition $\left(t_{0}, \hat{m}_{0}-m_{0}\right)$, related to $(u,m)$. Then
\begin{align}\label{(hat u-u-z,hat m-m-rho)_regularity}
& \sup_{t \in\left[t_{0}, T\right]} \left(\left\|\hat{u}(t, \cdot)-u(t, \cdot)-z(t,\cdot)\right\|_{C_{\mathcal{X}}^{k+\alpha}\left(\mathbb{T}_{\mathbb{G}}\right)} +\left\|\hat{m}(t)-m(t)-\rho(t)\right\|_{C_\mathcal{X}^{-(k-1+\alpha)}(\mathbb{T}_{\mathbb{G}})}\right) \\
\leq & C d_{1}^{2}\left(\hat{m}_{0},m_{0}\right),\notag
\end{align}
where the constant $C>0$ does not depend on $t_0$, $\hat{m}_0$ and $m_0$.

Finally, let $U$ be the function defined in \eqref{U_Def}, then one can obtain the $C^1$ differentiability of $U(t_0,x,m_0)$ in $m_0$, with the derivative $\frac{\delta U}{\delta m}\left(t_0,x,m_0,y\right)$ satisfying
\begin{equation}\label{delta U/delta m_regularity}
\sup_{(t_0,m_0) \in [0,T] \times \mathcal{P}(\mathbb{T}_{\mathbb{G}})}\left\|\frac{\delta U}{\delta m}\left(t_0, \cdot, m_0, \cdot\right)\right\|_{C_{\mathcal{X},y}^{k-1+\alpha}\left(\mathbb{T}_{\mathbb{G}};C_{\mathcal{X},x}^{k+\alpha}\left(\mathbb{T}_{\mathbb{G}}\right)\right)} \leq C,
\end{equation}
and $\frac{\delta U}{\delta m}$ and its Lie derivatives with respect to $(x,y)$ are continuous on $[0,T] \times \mathbb{T}_{\mathbb{G}} \times \mathcal{P}\left(\mathbb{T}_{\mathbb{G}}\right) \times \mathbb{T}_{\mathbb{G}}$. Moreover,
\begin{align*}
& \left\|U\left(t_0, \cdot, \hat{m}_0\right)-U\left(t_0, \cdot, m_0\right)-\int_{\mathbb{T}_{\mathbb{G}}} \frac{\delta U}{\delta m}\left(t_0, \cdot, m_0, y\right) d\left(\hat{m}_0-m_0\right)(y)\right\|_{C_\mathcal{X}^{k+\alpha}(\mathbb{T}_{\mathbb{G}})} \\
\leq & C d_{1}^{2}\left(\hat{m}_0,m_0\right).
\end{align*}
Here, the constants $C>0$ depend on $\mathbb{G}$, $\alpha$, $k$, $T$, $c_{F}$, $c_{G}$ and $H$ only.
\end{Prop}
To prove Proposition \ref{Prop_U C^1 w.r.t. m}, we shall construct the $C^1$ derivative of $U$ with respect to the measure in Lemma \ref{Lem_relation of z and rho_0} below. To do this, we differentiate the MFG system \eqref{MFG system} to get the linearized MFG system, which is a coupled system of a linear degenerate backward equation and a degenerate FPK equation as follows:
\begin{equation}\label{linear MFG system}
\begin{cases}
-\partial_t z-\Delta_{\mathcal{X}} z+D_pH(x,D_{\mathcal{X}}u) \cdot D_{\mathcal{X}} z=\frac{\delta F}{\delta m}(x,m(t))(\rho(t)),& \text {in }[t_0, T] \times \mathbb{T}_{\mathbb{G}}, \\
\partial_t\rho-\Delta_\mathcal{X}\rho-\operatorname{div}_{\mathcal{X}}(\rho D_pH(x,D_{\mathcal{X}}u)) =\operatorname{div}_{\mathcal{X}}(m D_{pp}^2H(x,D_{\mathcal{X}}u) D_{\mathcal{X}}z),& \text {in }[t_0, T] \times \mathbb{T}_{\mathbb{G}}, \\
z(T, x)=\frac{\delta G}{\delta m}(x,m(T))(\rho(T)),\quad \rho(t_0)=\rho_0,& \text {in }\mathbb{T}_{\mathbb{G}},
\end{cases}
\end{equation}
where $(u,m)$ is the solution to the MFG system \eqref{MFG system} with initial condition $m(t_0)=m_0$ for any fixed $(t_0,m_0) \in [0,T] \times \mathcal{P}(\mathbb{T}_{\mathbb{G}})$, and $\rho_0$ can be supposed in a suitable space.

We aim at proving that $U$ is of class $C^1$ in $m_0$ satisfying
\begin{equation*}
z(t_0,x)=\int_{\mathbb{T}_{\mathbb{G}}}\frac{\delta U}{\delta m}(t_0,x,m_0,y)d \rho_0(y)=:\frac{\delta U}{\delta m}(t_0,x,m_0)(\rho_0).
\end{equation*}
That is the following lemma.
\begin{Lem}\label{Lem_relation of z and rho_0}
Let assumptions \ref{assum1}-\ref{assum3} and the conclusions of Proposition \ref{Prop_MFG system wellposed regularity} hold, then, for any $\left(t_{0}, m_{0}\right) \in [0,T]\times\mathcal{P}\left(\mathbb{T}_{\mathbb{G}}\right)$, there exists a $C_{\mathcal{X}}^{k+\alpha}(\mathbb{T}_{\mathbb{G}}) \times C_{\mathcal{X}}^{k-1+\alpha}(\mathbb{T}_{\mathbb{G}})$ map $(x, y) \mapsto K\left(t_{0}, x, m_{0}, y\right)$ such that, for any $\rho_{0} \in C_{\mathcal{X}}^{-k}\left([0,1)^n\right)\cap C_{\mathcal{X}}^{-(k-1+\alpha)}(\mathbb{T}_{\mathbb{G}})$, the $z$ component of the solution to the system \eqref{linear MFG system} is given by
\begin{equation}\label{linear MFG system_rep.formula}
z\left(t_{0}, x\right)=\left\langle\rho_{0}, K\left(t_{0}, x, m_{0}, \cdot\right)\right\rangle.
\end{equation}
Moreover,
\begin{equation*}
\sup_{(t_0,m_0)\in[0,T]\times\mathcal{P}\left(\mathbb{T}_{\mathbb{G}}\right)}\left\|K\left(t_{0}, \cdot, m_{0}, \cdot\right)\right\|_{C_{\mathcal{X},y}^{k-1+\alpha}\left(\mathbb{T}_{\mathbb{G}};C_{\mathcal{X},x}^{k+\alpha}\left(\mathbb{T}_{\mathbb{G}}\right)\right)} \leq C,
\end{equation*}
where the constant $C>0$ depends on $\mathbb{G}$, $\alpha$, $k$, $T$, $c_F$, $c_G$ and $H$ only, and $K$ and its Lie derivatives with respect to $(x,y)$ are continuous on $[0,T] \times \mathbb{T}_{\mathbb{G}} \times \mathcal{P}\left(\mathbb{T}_{\mathbb{G}}\right) \times \mathbb{T}_{\mathbb{G}}$.
\end{Lem}

We highlight three main contributions of this paper as follows. First, as far as we know, this paper represents the first study on the non-totally degenerate MFG master equation, namely by addressing the periodic MFG problem on Carnot groups. These groups are non-commutative groups with stratified Lie algebra structures (e.g., the Heisenberg group) and also typical examples of sub-Riemannian manifolds. From an applied perspective, the master equation studied in this paper is a kind of PDE defined on the space of probability measures. It is related to a second-order MFG on Carnot tori, where each player can move periodically only along admissible trajectories given by the vector fields generating the Carnot group. Unlike the MFG system, the master equation is a single equation that can describe the Nash equilibria of the associated MFG problem. Second, this paper generalizes the classical (non-degenerate) first-order master equation for MFGs proposed in \cite{19CDLL} to the degenerate case. We first establish the well-posedness and regularity of solutions to the corresponding degenerate MFG system \eqref{MFG system} and the linearized degenerate MFG system \eqref{linear MFG system}, and then proceed to prove the well-posedness and regularity of the solution to the degenerate master equation \eqref{ME}. As a byproduct, we rigorously illustrate the equivalent characterization between the master equation \eqref{ME} and the MFG system \eqref{MFG system}. Last, the degenerate nature of the PDEs considered in this paper brings challenges to our research. To overcome the non-commutativity of vector fields and the non-isotropy of space, we exploit the regularity results established in \cite{JWY} for solutions to the linear degenerate backward parabolic equation and the general-form FPK equation on Carnot tori, including Schauder estimates, H\"{o}lder continuity estimates and weak regularity estimates, to solve the degenerate MFG system \eqref{MFG system} and the linearized degenerate MFG system \eqref{linear MFG system} over the product space of the non-isotropic H\"{o}lder space and its dual space.

The rest of the paper is organized as follows. In Section \ref{Sec_2}, we provide some preliminaries on Carnot groups and Carnot tori, and then define the non-isotropic H\"{o}lder spaces and their dual spaces. In Section \ref{Sec_3}, we focus on proving the space-regularity and measure-Lipschitz continuity of the constructed function $U$ in \eqref{U_Def}, i.e. Proposition \ref{Prop_MFG system wellposed regularity}-\ref{Prop_Lip. ctn. of U}. In Section \ref{Sec_4}, first we prove Lemma \ref{Lem_relation of z and rho_0} which concerns the linearized MFG system \eqref{linear MFG system}. Next we prove the $C^1$ differentiability of $U$ with respect to the measure, as stated in Proposition \ref{Prop_U C^1 w.r.t. m}. Finally, in Section \ref{Sec_5}, we complete the proof of the well-posedness of the solution to the master equation \eqref{ME}, i.e. Theorem \ref{Thm_ME wellposed regularity}.

\section{Brief preliminaries on Carnot groups and Carnot Tori}\label{Sec_2}

This section is devoted to introducing the concepts of Carnot groups and Carnot Tori, as well as defining the non-isotropic H\"{o}lder spaces and their dual spaces. We start with a review of fundamental notations and preliminaries related to Carnot groups; for a more comprehensive overview of this topic, readers are referred to the monographs \cite{07BLU,22BB}.
\begin{Def}[Homogeneous group, see \cite{07BLU,22BB}]\label{Def_Homogeneous group}
Let $\mathbb{G}=\left(\mathbb{R}^n,\circ\right)$ be a Lie group on $\mathbb{R}^n$ with $\circ$ being a given Lie group law on $\mathbb{R}^n$, called ``translation''. We say that $\mathbb{G}$ is a homogeneous (Lie) group on $\mathbb{R}^n$ if there exists an $n$-tuple of real numbers $(\alpha_1,\ldots,\alpha_n)$, with $1=\alpha_1 \leq \alpha_2 \leq \ldots \leq \alpha_n$, such that the ``dilation'' $D_{\lambda}:\mathbb{R}^n \to \mathbb{R}^n$, defined as
$$
D_{\lambda}(x):=\left(\lambda^{\alpha_1} x_1,\ldots,\lambda^{\alpha_n} x_n\right)
$$
is an automorphism of the group $\mathbb{G}$ for every $\lambda>0$.

We denote by $\mathbb{G}=(\mathbb{R}^n,\circ,D_{\lambda})$ the datum of a homogeneous group on $\mathbb{R}^n$ with group law $\circ$ and dilation group $\{D_{\lambda}\}_{\lambda>0}$. Moreover, the number
$$
Q:=\sum_{i=1}^{n}\alpha_i
$$
is called the homogeneous dimension of the homogeneous group. 
\end{Def}
A differential operator $P$ on $\mathbb{G}$ is said to be \textit{left-invariant} if
$$
P^x \left(f(y \circ x)\right)=\left(P f\right)(y \circ x)
$$
for every test function $f$ and $x,y \in \mathbb{R}^n$. Similarly, a differential operator $P$ on $\mathbb{G}$ is said to be \textit{right-invariant} if
$$
P^y \left(f(y \circ x)\right)=\left(P f\right)(y \circ x)
$$
for every test function $f$ and $x,y \in \mathbb{R}^n$.

For $\delta\in\mathbb{R}$, $P$ is said to be \textit{$D_{\lambda}$-homogeneous of degree $\delta$} (or simply ``\textit{$\delta$-homogeneous}'') if
$$
P^x\left(f(D_{\lambda}(x))\right)=\lambda^{\delta} \left(P f\right) \left(D_{\lambda}(x)\right)
$$
for every test function $f$, $\lambda>0$ and $x \in \mathbb{R}^n\setminus\{0\}$. A real function $f$ defined on $\mathbb{R}^n$ is called \textit{$D_{\lambda}$-homogeneous of degree $\delta$} (or simply ``\textit{$\delta$-homogeneous}'') if $f\not\equiv 0$ and $f$ satisfies
$$
f(D_{\lambda}(x))=\lambda^{\delta}f(x)
$$
for any $\lambda>0$ and $x \in \mathbb{R}^n$.

The Haar measure, denoted by $dx$, of the homogeneous group $\mathbb{G}=(\mathbb{R}^n,\circ, D_{\lambda})$ is known to be left-right-invariant and to coincide with the Lebesgue measure on $\mathbb{R}^n$.
\begin{Lem}[see {\cite[Theorem 3.29, Remark 3.30]{22BB}}]\label{Lem_LI&RI vector field}
Let $X_i$ be the left-invariant vector field which coincides with $\partial_{x_i}$ at the origin, i.e. $X_{i}(0)=\partial_{x_{i}}|_{0}$, for $i=1,\ldots,n$. Then $X_i$ is $\alpha_i$-homogeneous and has the following structure:
\begin{equation*}
X_i=\partial_{x_i}+\sum_{i<k}q_{ik}(x)\partial_{x_k},\quad i=1,\ldots,n
\end{equation*}
where $q_{ik}(\cdot)$ is a $(\alpha_k-\alpha_i)$-homogeneous polynomial. In particular, $q_{ik}$ can only depend on the variable $x_1,\ldots,x_{k-1}$. Moreover, the transpose of the vector field is just its opposite:
\begin{equation*}
X_i^*=-X_i.
\end{equation*}

Analogous properties hold for the right-invariant vector field $X_i^R$ which coincides with $\partial_{x_i}$ at the origin.
\end{Lem}

In addition, we turn to introduce the H\"{o}rmander's condition.
\begin{Def}[H\"{o}rmander's condition, see \cite{67Ho}]
Let $\Omega$ be an open subset of $\mathbb{R}^n$, and let $Y_{1},Y_{2},\ldots, Y_{m}$ be real smooth vector fields defined on $\Omega$. We say $Y_{1},Y_{2},\ldots, Y_{m}$ satisfy the H\"{o}rmander's condition on $\Omega$ if there exists a smallest integer $r\geq 1$ such that $Y_{1},Y_{2},\ldots, Y_{m}$ together with their commutators of length at most $r$ span the tangent space $T_{x}(\Omega)$ at each point $x\in\Omega$. The integer $r$ is called the H\"{o}rmander's index of $\Omega$.
\end{Def}
We can define the Carnot-Carath\'{e}odory distance (also referred to as the subunit metric or control distance) associated with the family of H\"{o}rmander's vector fields $\mathcal{Y}=(Y_1,Y_2,\ldots,Y_m)$. For further details regarding its definition, we refer the reader to \cite[Definition 1.29]{22BB}.
\begin{Def}[Carnot-Carath\'{e}odory distance]\label{Def_CC distance}
For any points $x,y\in \Omega$ and $\delta>0$, let $C_{x,y}(\delta)$ be the collection of absolutely continuous mapping $\varphi:[0,\delta]\to \Omega$, which satisfies $\varphi(0)=x,\varphi(\delta)=y$ and
\[ \varphi'(t)=\sum_{i=1}^{m}a_{i}(t)Y_{i}(\varphi(t)),~~\sum_{i=1}^{m}{a_{i}(t)}^2\leq 1,~~a.e.~~ t\in [0,\delta]. \]
The Carnot-Carath\'{e}odory distance $d_{cc}(x,y)$ is defined as
\begin{equation*}\label{2-14}
d_{cc}(x,y):=\inf\{\delta>0~|~ \exists~\varphi\in C_{x,y}(\delta)\}.
\end{equation*}
\end{Def}
The well-definedness of the Carnot-Carath\'{e}odory distance is ensured by the Chow-Rashevskii theorem and the H\"{o}rmander's condition (see \cite[Theorem 57]{14Br}). Moreover, the local topological equivalence between this distance and the Euclidean distance is a known result: for any compact set $K \subset \Omega$, there exists a constant $C>0$ such that, for any $x,y \in K$, we have
\begin{equation}\label{d_cc VS Euclidean}
C^{-1} |x-y| \leq d_{cc}(x, y) \leq C |x-y|^{\frac{1}{r}},
\end{equation}
where $|\cdot|$ denotes the Euclidean norm on $\mathbb{R}^n$ and $r$ is the H\"{o}rmander's index of $\Omega$.

We are now ready to state the definitions of the Carnot group and Carnot torus, which are given as follows:
\begin{Def}[Homogeneous Carnot group, see \cite{07BLU,22BB}]\label{Def_Carnot group}
The Lie group $\mathbb{G}=(\mathbb{R}^{n},\circ)$ is called the (homogeneous) Carnot group (or a stratified Lie group) if the following two conditions are fulfilled:
\begin{enumerate}[label=(\arabic*), leftmargin=*, itemindent=1.7em]
\item the decomposition $\mathbb{R}^{n}=\mathbb{R}^{n_{1}}\times \mathbb{R}^{n_{2}}\times \cdots\times \mathbb{R}^{n_{r}}$ holds for some integers $n_{1},\ldots, n_{r}$ such that      $n_{1}+n_{2}+\cdots+n_{r}=n$, and for each $\lambda>0$ there exists a dilation
\begin{equation*}
D_{\lambda}(x)=D_{\lambda}\left(x^{(1)},x^{(2)},\ldots,x^{(r)}\right)=\left(\lambda x^{(1)},\lambda^{2}x^{(2)},\ldots,\lambda^{r}x^{(r)}\right),
\end{equation*}
which is an automorphism of the group $\mathbb{G}$. Here $x^{(i)}\in \mathbb{R}^{n_{i}}$ for $i=1,2,\ldots,r$;\label{Carnot group C1}
\item let $X_{1},\ldots,X_{n_{1}}$ be the left invariant vector fields on $\mathbb{G}$ such that $X_{k}(0)=\partial_{x_{k}}|_{0}$ for $k=1,\ldots,n_{1}$. Then
\begin{equation*}
\operatorname{rank}(\operatorname {Lie}\{X_1,\ldots,X_{n_1}\}(x))=n
\end{equation*}
for every $x\in\mathbb{R}^n$, i.e., $X_1,\ldots,X_{n_1}$ satisfy the H\"{o}rmander's condition of step $r$ on $\mathbb{R}^{n}$.\label{Carnot group C2}
\end{enumerate}

If \ref{Carnot group C1} and \ref{Carnot group C2} are satisfied, we shall say that the triple $\mathbb{G}=(\mathbb{R}^{n},\circ, D_{\lambda})$ is a (homogeneous) Carnot group with homogeneous dimension 
\begin{equation*}
Q=\sum_{j=1}^{r}j n_{j}.
\end{equation*}
We also say that $\mathbb{G}$ has step $r$ and $n_1$ generators. The vector fields $X_1,\ldots,X_{n_{1}}$ are called the Jacobian generators of $\mathbb{G}$.  
\end{Def}
\begin{Def}[Carnot torus]
The torus in the Carnot group $\left(\mathbb{G},\circ\right)$, denoted by $\mathbb{T}_{\mathbb{G}}$, is defined as the quotient space $\mathbb{G}/\mathbb{Z}^{n}$, which is determined by the following equivalence relation:
$$
x \sim y \text{ if there exists } k \in \mathbb{Z}^{n} \text{ such that } k \circ x=y.
$$
\end{Def}
Analogous to the Euclidean torus, functions on $\mathbb{T}_{\mathbb{G}}$ are functions $f$ on $\mathbb{G}$ that satisfy $f(k \circ x)=f(x)$ for all $x \in \mathbb{G}$ and $k \in \mathbb{Z}^{n}$, and are thus called $1_{\mathbb{G}}$-periodic functions. Further, the following lemma indicates that the torus $\mathbb{T}_{\mathbb{G}}$ can be considered as the cube $[0,1)^n$, the proof of which is provided in \cite[Lemma 2.2]{JWY}.
\begin{Lem}
For every point $x\in\mathbb{G}$, there exists a unique point $x_0\in[0,1)^n$ and a finite number of group actions generated by elements of the form $(k,0)\in\mathbb{Z}^{n_1}\times\{0\}^{n-n_1}$ such that applying these actions to $x_0$ yields $x$.
\end{Lem}
Additional details concerning periodicity on Carnot groups can be found in \cite{07St,21MMT,24CMM}. Notably, the Carnot torus does not coincide with the Euclidean torus. For example, $\mathbb{T}_{\mathbb{G}}$ is not obtained identifying the points of two opposite faces of $[0,1]^n$ with the same two coordinates. Moreover, $\mathbb{T}_{\mathbb{G}}$ is a bounded compact space, naturally endowed with the distance induced by any distance $d$ in $\mathbb{G}$ as
$$
d^{\mathbb{T}_{\mathbb{G}}}(x,y):=\inf_{\substack{x', y' \in \mathbb{G} \\ x' \sim x, y' \sim y}}d\left(x',y'\right), \text{ any } x,y \in \mathbb{T}_{\mathbb{G}}.
$$

\begin{Rem}\label{Rem_Holder norm R^n=T_mathbb G}
Since a $1_{\mathbb{G}}$-periodic function $f$ on $\mathbb{G}=\mathbb{R}^{n}$ can be regarded as a function on $\mathbb{T}_{\mathbb{G}}$ (still denoted as $f$), we note that the H\"{o}lder norm induced by $d$ on $\mathbb{G}$ is equal to the one induced by $d^{\mathbb{T}_{\mathbb{G}}}$ on $\mathbb{T}_{\mathbb{G}}$, namely $\left\|f\right\|_{C_{d}^{\alpha}\left(\mathbb{R}^{n}\right)} =\left\|f\right\|_{C_{d^{\mathbb{T}_{\mathbb{G}}}}^{\alpha}\left(\mathbb{T}_{\mathbb{G}}\right)},\alpha\in(0,1]$. This is because, on the one hand, for any $x,y \in \mathbb{G}$, $x \neq y$,
\begin{equation*}
\frac{|f(x)-f(y)|}{d(x,y)^{\alpha}} \leq\frac{|f(x)-f(y)|}{d^{\mathbb{T}_{\mathbb{G}}}(x,y)^{\alpha}}.
\end{equation*}
Hence $\left[f\right]_{C_{d}^{\alpha}\left(\mathbb{R}^{n}\right)} \leq\left[f\right]_{C_{d^{\mathbb{T}_{\mathbb{G}}}}^{\alpha}\left(\mathbb{T}_{\mathbb{G}}\right)}$. On the other hand, for any $x,y \in \mathbb{T}_{\mathbb{G}}$, $x \neq y$,
\begin{equation*}
\frac{|f(x)-f(y)|}{d^{\mathbb{T}_{\mathbb{G}}}(x,y)^{\alpha}} =\inf_{\substack{x', y' \in \mathbb{G} \\ x' \sim x, y' \sim y}}\frac{|f(x')-f(y')|}{d(x',y')^{\alpha}}.
\end{equation*}
Hence $\left[f\right]_{C_{d^{\mathbb{T}_{\mathbb{G}}}}^{\alpha}\left(\mathbb{T}_{\mathbb{G}}\right)} \leq\left[f\right]_{C_{d}^{\alpha}\left(\mathbb{R}^{n}\right)}$. Since it's easy to find that $\left\|f\right\|_{C\left(\mathbb{R}^{n}\right)} =\left\|f\right\|_{C\left(\mathbb{T}_{\mathbb{G}}\right)}$, finally we obtain $\left\|f\right\|_{C_{d}^{\alpha}\left(\mathbb{R}^{n}\right)} =\left\|f\right\|_{C_{d^{\mathbb{T}_{\mathbb{G}}}}^{\alpha}\left(\mathbb{T}_{\mathbb{G}}\right)}$.
\end{Rem}

Finally, we present an important class of non-isotropic H\"{o}lder spaces associated with the family of vector fields $\mathcal{X}=\{X_1,\ldots,X_{n_1}\}$ (see \cite{07BB,10BBLU}).

Let $\Omega\subseteq\mathbb{R}^n$ be any open subset and $X_I:=X_{i_1} \cdots X_{i_{|I|}}$, where $I$ is any multi-index $I=\left(i_1, \ldots, i_{|I|}\right)$ with $i_j \in \{1,\cdots,n_1\}, j=1,\cdots,|I|$. For any $k \in \mathbb{N}$, we define the non-isotropic space
\begin{equation*}
C_{\mathcal{X}}^{k}\left(\Omega\right):=\left\{\phi \in C\left(\Omega\right)~|~X_I \phi \in C\left(\Omega\right),\,\forall\,|I| \leq k\right\}.
\end{equation*}
For any $k \in \mathbb{N}$ and $\alpha \in(0,1]$ we define the non-isotropic H\"{o}lder spaces
\begin{equation*}
C_{\mathcal{X}}^{\alpha}\left(\Omega\right):=\left\{\phi \in L^{\infty}\left(\Omega\right)~\bigg|~ \sup _{\substack{x, y \in \Omega \\
x \neq y}} \frac{|\phi(x)-\phi(y)|}{d_{cc}(x, y)^\alpha}<+\infty\right\},
\end{equation*}
\begin{equation}\label{non-isotropic Holder spaces_Def}
C_{\mathcal{X}}^{k+\alpha}\left(\Omega\right):=\left\{\phi \in L^{\infty}\left(\Omega\right)~|~X_I \phi \in C_{\mathcal{X}}^{\alpha}\left(\Omega\right),\,\forall\,|I| \leq k\right\}.
\end{equation}
For any function $\phi \in C_{\mathcal{X}}^{ \alpha}\left(\Omega\right)$, the H\"{o}lder seminorm can be defined as
$$
[\phi]_{C_{\mathcal{X}}^{ \alpha}\left(\Omega\right)}:=\sup_{\substack{x, y \in \Omega \\ x \neq y}} \frac{|\phi(x)-\phi(y)|}{d_{cc}(x, y)^\alpha}.
$$
Furthermore, for any $\phi \in C_{\mathcal{X}}^{k+ \alpha}\left(\Omega\right)$, the H\"{o}lder norm is defined as
\begin{equation}\label{Holder norm}
\|\phi\|_{C_{\mathcal{X}}^{k+\alpha}\left(\Omega\right)}:=\|\phi\|_{C_{\mathcal{X}}^k\left(\Omega\right)}+\sum_{0 \leq|I| \leq k}\left[X_I \phi\right]_{C_{\mathcal{X}}^{ \alpha}\left(\Omega\right)},
\end{equation}
where $\|\phi\|_{C_{\mathcal{X}}^k\left(\Omega\right)}:=\sum\limits_{0 \leq |I|\leq k}\|X_I\phi\|_{L^{\infty}\left(\Omega\right)}$.

Endowed with the above norm, $C_{\mathcal{X}}^{k+ \alpha}\left(\Omega\right)$ is a Banach space and it follows from \eqref{d_cc VS Euclidean} that, for any compact set $K \subset \Omega$,
$$
C^{-1}\|\phi\|_{C^{ \frac{\alpha}{k}}\left(K\right)} \leq\|\phi\|_{C_{\mathcal{X}}^{\alpha}\left(K\right)} \leq C\|\phi\|_{C^{ \alpha}\left(K\right)},
$$
where $\|\phi\|_{C^{\alpha}\left(K\right)}$ is the standard H\"{o}lder norm, and $C>0$ is a constant depending only on the dimension $n$ and the family of vector fields $\mathcal{X}$.


For any $T>0$, $l,k \in \mathbb{N}$ and $\alpha,\beta \in(0,1]$, we define the parabolic non-isotropic H\"{o}lder spaces on $[0,T] \times \Omega$ as
\begin{align*}
& C_{\mathcal{X}}^{\beta,\alpha}\left([0,T]\times \Omega \right)\\
:=&\left\{\phi \in L^{\infty}\left([0,T] \times \Omega\right)~\bigg|~\sup _{\substack{(t,x), (s,y) \in [0,T]\times \Omega \\
(t,x) \neq (s,y)}} \frac{|\phi(t,x)-\phi(s,y)|}{|t-s|^\beta+d_{cc}(x,y)^\alpha}<+\infty\right\},
\end{align*}
\begin{align}\label{time-space-Holder space_Def}
& C_{\mathcal{X}}^{l+\beta,k+\alpha}\left([0,T] \times \Omega\right)\\
:=&\left\{\phi \in L^{\infty}\left([0,T]\times \Omega\right)~\big|~\partial_t^i X_I \phi \in C_{\mathcal{X}}^{\beta,\alpha}\left([0,T]\times \Omega \right),\,\forall\, i,|I|\in\mathbb{N},i\leq l,|I|\leq k\right\} \notag
\end{align}
with the seminorm
$$
[\phi]_{C_{\mathcal{X}}^{\beta,\alpha}\left([0,T]\times \Omega\right)}:=\sup_{\substack{(t,x), (s,y) \in [0,T]\times \Omega \\
(t,x) \neq (s,y)}} \frac{|\phi(t,x)-\phi(s,y)|}{|t-s|^\beta+d_{cc}(x,y)^\alpha},
$$
and the norm
\begin{equation}\label{para. Holder norm}
\|\phi\|_{C_{\mathcal{X}}^{l+\beta,k+\alpha}\left([0,T]\times \Omega\right)}:=\sum_{i\leq l,|I|\leq k}\left(\\|\partial_t^i X_I\phi\|_{L^{\infty}\left([0,T]\times \Omega\right)} +\left[\partial_t^i X_I \phi\right]_{C_{\mathcal{X}}^{\beta,\alpha}\left([0,T]\times \Omega\right)}\right).
\end{equation}
We also define the spaces
\begin{align}\label{P-Holder space_Def}
& C_{\mathcal{X}}^{\frac{k+\alpha}{2},k+\alpha}\left([0,T] \times \Omega\right) \\
:= &\left\{\phi \in L^{\infty}\left([0,T]\times \Omega\right)~\big|~\partial_t^i X_I \phi \in C_{\mathcal{X}}^{\frac{\alpha}{2},\alpha}\left([0,T]\times \Omega \right),\,\forall\, i,|I|\in\mathbb{N},\,2i+|I|\leq k\right\}\notag
\end{align}
and
\begin{equation*}
B\left([0,T];C_{\mathcal{X}}^{k+\alpha}\left(\Omega\right)\right):=\left\{\phi:[0,T]\to C_{\mathcal{X}}^{k+\alpha}\left(\Omega\right)~\bigg|~ \sup_{t\in[0,T]}\|\phi(t,\cdot)\|_{C_{\mathcal{X}}^{k+\alpha}\left(\Omega\right)}<+\infty\right\}.
\end{equation*}


In addition, let us introduce the spaces of functions depending on two space variables as follows. For any $k_1,k_2 \in \mathbb{N}$ and $\alpha_1,\alpha_2 \in(0,1]$, we define
\begin{equation*}
C_{\mathcal{X},(x,y)}^{k_1,k_2}\left(\Omega \times \Omega\right) :=\left\{\phi:\Omega\times\Omega\to\mathbb{R} ~|~ X_{J}^x X_{J'}^y \phi(x,y) \in C\left(\Omega\times\Omega\right),\,\forall\,|I| \leq k_1,\left|J'\right| \leq k_2\right\},
\end{equation*}
and
\begin{equation*}
C_{\mathcal{X},x}^{k_1+\alpha_1}\left(\Omega;C_{\mathcal{X},y}^{k_2+\alpha_2}\left(\Omega\right)\right) :=\left\{\phi:\Omega\to C_{\mathcal{X}}^{k_2+\alpha_2}\left(\Omega\right)~\big|~\|\phi(x,\cdot)\|_{C_{\mathcal{X}}^{k_2+\alpha_2}\left(\Omega\right)} \in C_{\mathcal{X}}^{k_1+\alpha_1}\left(\Omega\right)\right\}
\end{equation*}
with the norm
\begin{align}\label{Holder norm_two space variables}
\|\phi\|_{C_{\mathcal{X},x}^{k_1+\alpha_1}\left(\Omega;C_{\mathcal{X},y}^{k_2+\alpha_2}\left(\Omega\right)\right)} := & \|\phi\|_{C_{\mathcal{X},(x,y)}^{k_1,k_2}\left(\Omega \times \Omega\right)} +\sum_{|I|\leq k_1}\bigg(\sup_{x\in\Omega}\left\|X_{J}^x\phi(x,\cdot)\right\|_{C_{\mathcal{X}}^{k_2+\alpha_2}\left(\Omega\right)} \\
& +\sup_{\substack{x_1,x_2 \in \Omega \\
x_1 \neq x_2}}\frac{\left\|X_{J}^x\phi(x,\cdot)|_{x=x_1}-X_{J}^x\phi(x,\cdot)|_{x=x_2}\right\|_{C_{\mathcal{X}}^{k_2+\alpha_2}\left(\Omega\right)}}{d_{cc}\left(x_1,x_2\right)^{\alpha_1}}\bigg),\notag
\end{align}
where
$$
\|\phi\|_{C_{\mathcal{X},(x,y)}^{k_1,k_2}\left(\Omega \times \Omega\right)}:=\sum_{|I| \leq k_1,\left|J'\right| \leq k_2}\left\|X_{J}^x X_{J'}^y \phi\right\|_{L^{\infty}\left(\Omega\times\Omega\right)}.
$$
Here, $X_i^x \phi(x,y)$ and $X_i^y \phi(x,y)$ respectively mean that the differential operator $X_i$ acts on $x$ and $y$, i.e. $X_i^x \phi(x,y)=X_i \phi(\cdot,y)|_{x}$ and $X_i^y \phi(x,y)=X_i \phi(x,\cdot)|_{y}$.

The dual space of $C_{\mathcal{X}}^{k+\alpha}\left(\Omega\right)$ is denoted by $C_{\mathcal{X}}^{-(k+\alpha)}\left(\Omega\right)$ with norm
\begin{equation*}\label{dual Holder space norm_Def}
\|\psi\|_{C_{\mathcal{X}}^{-(k+\alpha)}\left(\Omega\right)}:=\sup_{\|\phi\|_{C_{\mathcal{X}}^{k+\alpha}\left(\Omega\right)} \leq 1}\left\langle\psi, \phi\right\rangle_{C_{\mathcal{X}}^{-(k+\alpha)}\left(\Omega\right), C_{\mathcal{X}}^{k+\alpha}\left(\Omega\right)}\text{ for any }\psi \in C_{\mathcal{X}}^{-(k+\alpha)}\left(\Omega\right).
\end{equation*}
In the following derivation, we denote $\left\langle\psi, \phi\right\rangle_{C_{\mathcal{X}}^{-(k+\alpha)}\left(\Omega\right), C_{\mathcal{X}}^{k+\alpha}\left(\Omega\right)}$ simply as $\left\langle\psi, \phi\right\rangle$. 
\begin{Rem}
Following the same way as above, we can also define the non-isotropic H\"{o}lder spaces on $\mathbb{T}_{\mathbb{G}}$ and their dual spaces, where it suffices to replace the distance $d_{cc}$ with $d_{cc}^{\mathbb{T}_{\mathbb{G}}}$.
\end{Rem}

\section{Space-regularity and Lipschitz continuity of \texorpdfstring{$U$}{U}}\label{Sec_3}

In this section, we dedicate the following part to proving Proposition \ref{Prop_MFG system wellposed regularity}-\ref{Prop_Lip. ctn. of U}.

\subsection{Regularity results of \texorpdfstring{$U$}{U} in \texorpdfstring{$x$}{x}}

\begin{proof}[\bf{Proof of Proposition \ref{Prop_MFG system wellposed regularity}}]
Without loss of generality, let $t_{0}=0$; other cases reduce to this via $\tilde{t}=t-t_0$. We prove the existence by using the Schauder fixed point theorem. Set
\begin{equation}\label{SFPT_E Def.}
E:=\left\{m \in C\left([t_0,T];\mathcal{P}\left(\mathbb{T}_{\mathbb{G}}\right)\right)\text{ s.t. }d_1\left(m\left(t_1\right), m\left(t_2\right)\right) \leq c_E|t_1-t_2|^\frac{1}{2}\right\},
\end{equation}
where the constant $c_E>0$ (to be determined) does not depend on $m$. It is easy to verify that E is a convex compact space with the norm
\begin{equation*}
\|m\|_{E}:=\sup_{t\in[t_0,T]}d_1(m(t),0)+\sup_{t_1 \neq t_2}\frac{d_1(m(t_1),m(t_2))}{|t_1-t_2|^{\frac{1}{2}}},\,m\in E.
\end{equation*}
We construct a map $\Phi: E \to E$ as follows.

Fix any $\mu \in E$ and consider the following HJB equation
\begin{equation}\label{SFPT_HJB}
\begin{cases}
-\partial_t u-\Delta_{\mathcal{X}}u+\frac{1}{2}\left|D_{\mathcal{X}} u\right|^2=F(x, \mu(t)), & \text{in } [t_0,T] \times \mathbb{T}_{\mathbb{G}}, \\
u(T,x)=G(x, \mu(T)), & \text{in } \mathbb{T}_{\mathbb{G}}.
\end{cases}
\end{equation}
From assumptions \ref{assum1} and \ref{assum2}, we know that $F\left(x,\mu(t)\right) \in C_{\mathcal{X}}^{\frac{\alpha}{2},k-2+\alpha}\left([t_0,T]\times\mathbb{T}_{\mathbb{G}}\right)$ and $G\left(x,\mu(T)\right) \in C_{\mathcal{X}}^{k+\alpha}\left(\mathbb{T}_{\mathbb{G}}\right)$ with the corresponding norms independent of $\mu$. We use the Hopf transform to turn \eqref{SFPT_HJB} equivalently into a linear form, which is satisfied by $w:=\exp\left(-\frac{u}{2}\right)$, that is
\begin{equation*}
\begin{cases}
-\partial_t w-\Delta_{\mathcal{X}}w+\frac{1}{2}F(x, \mu(t))w=0, & \text{in } [t_0,T] \times \mathbb{T}_{\mathbb{G}}, \\
w(T,x)=\exp\left(-\frac{G(x, \mu(T))}{2}\right), & \text{in } \mathbb{T}_{\mathbb{G}}.
\end{cases}
\end{equation*}
It follows from Lemma \ref{Lem_LHJB well-posedness} that there exists a unique solution $w\in C_{\mathcal{X}}^{1,2} \left([t_0,T] \times \mathbb{T}_{\mathbb{G}} \right)$ to the above equation, and
\begin{equation}\label{w>0}
w(t,x)=\int_{\mathbb{T}_{\mathbb{G}}}\sum_{k\in\mathbb{Z}^{n}}\Gamma_w(T-t+t_0,k\circ x;t_0,y)\exp\left(-\frac{G(y, \mu(T))}{2}\right)dy \geq \exp\left(-\frac{c_G}{2}\right)>0
\end{equation}
for any $(t,x)\in[t_0,T]\times\mathbb{T}_{\mathbb{G}}$, where $\Gamma_w$ is the fundamental solution for $\partial_t-\Delta_{\mathcal{X}}+\frac{1}{2}F(x, \mu(t))$.

Therefore, there exists a unique solution $u\in C_{\mathcal{X}}^{1,2} \left([t_0,T] \times \mathbb{T}_{\mathbb{G}} \right)$ to equation \eqref{SFPT_HJB}. Moreover, according to \cite[Proposition 4.1]{JWY}, we obtain
\begin{align}\label{HJB_Schauder estimate}
& \left\|u\right\|_{C^{\frac{\alpha}{2},k+\alpha}_{\mathcal{X}}\left([t_0,T]\times\mathbb{T}_{\mathbb{G}}\right)} +\left\|u\right\|_{C^{\frac{2+\alpha}{2},2+\alpha}_{\mathcal{X}}\left([t_0,T]\times\mathbb{T}_{\mathbb{G}}\right)} \\
\leq & C\left(\left\|w\right\|_{C^{\frac{\alpha}{2},k+\alpha}_{\mathcal{X}}\left([t_0,T]\times\mathbb{T}_{\mathbb{G}}\right)} +\left\|w\right\|_{C^{1+\frac{\alpha}{2},\alpha}_{\mathcal{X}}\left([t_0,T]\times\mathbb{T}_{\mathbb{G}}\right)}\right) \leq C\sup_{m \in \mathcal{P}(\mathbb{T}_{\mathbb{G}})} \left\|G(\cdot,m)\right\|_{C_{\mathcal{X}}^{k+\alpha}\left(\mathbb{T}_{\mathbb{G}}\right)}, \notag
\end{align}
where $C>0$ depends on $\mathbb{G}$, $\sup_{m \in \mathcal{P}(\mathbb{T}_{\mathbb{G}})}\left\|G(\cdot,m)\right\|_{L^{\infty}\left(\mathbb{T}_{\mathbb{G}}\right)}$, $\sup_{m \in \mathcal{P}(\mathbb{T}_{\mathbb{G}})}\left\|F\left(\cdot,m\right)\right\|_{C_{\mathcal{X}}^{k-2+\alpha}\left(\mathbb{T}_{\mathbb{G}}\right)}$, $\sup_{m \in \mathcal{P}(\mathbb{T}_{\mathbb{G}})}\left\|\frac{\delta F}{\delta m}(\cdot,m,\cdot)\right\|_{C_{\mathcal{X},y}^{0+1}\left(\mathbb{T}_{\mathbb{G}};C_{\mathcal{X},x}^{k-2+\alpha}\left(\mathbb{T}_{\mathbb{G}}\right)\right)}$, $T$, $k$, $\alpha$ and $c_E$ only.

Consider the following FPK equation
\begin{equation}\label{SFPT_FPK}
\begin{cases}
\partial_t m-\Delta_{\mathcal{X}} m-\operatorname{div}_{\mathcal{X}}\left(m D_{\mathcal{X}}u\right)=0, & \text{in } [t_0,T] \times \mathbb{T}_{\mathbb{G}}, \\
m(t_0)=m_0, & \text{in } \mathbb{T}_{\mathbb{G}}.
\end{cases}
\end{equation}
Since $m_0 \in \mathcal{P}(\mathbb{T}_{\mathbb{G}}) \subset C_{\mathcal{X}}^{-k}\left([0,1)^n\right)\cap C_{\mathcal{X}}^{-(k+\alpha)}(\mathbb{T}_{\mathbb{G}})$, by Lemma \ref{Lem_general FPK wellposed regularity}, there exists a unique distributional solution $m(t) \in C\left([t_0,T];C_{\mathcal{X}}^{-(k+\alpha)}(\mathbb{T}_{\mathbb{G}})\right)$ to equation \eqref{SFPT_FPK}. Then we can define $\Phi(\mu)=m$. Let us check that $m \in E$.

From \eqref{general FPK_weak formulation} and Lemma \ref{Lem_LHJB well-posedness}, we have that for any $t\in[t_0,T]$ and $\xi\in C_\mathcal{X}^{k+\alpha}(\mathbb{T}_{\mathbb{G}})$,
\begin{equation*}
\left\langle m(t),\xi\right\rangle =\left\langle m_0,z(t_0,\cdot)\right\rangle =\int_{\mathbb{T}_{\mathbb{G}}}\int_{\mathbb{T}_{\mathbb{G}}}\sum_{k\in\mathbb{Z}^{n}}\Gamma_z(t,k\circ x;t_0,y)\xi(y)dy dm_0(x),
\end{equation*}
where $\Gamma_z$ is the fundamental solution for $\partial_t-\Delta_\mathcal{X}+D_{\mathcal{X}}u\cdot D_{\mathcal{X}}$. For any Borel set $A \subset \mathbb{T}_{\mathbb{G}}$, there exists a smooth non-negative sequence $\{\xi^{i}\}_{i=1}^{+\infty} \subset C^{\infty}\left(\mathbb{T}_{\mathbb{G}}\right)$ converges to $1_A$. Then, for any $t\in[t_0,T]$,
\begin{equation*}
\left\langle m(t),1_A\right\rangle =\lim_{i\to+\infty}\left\langle m(t),\xi^{i}\right\rangle =\int_{\mathbb{T}_{\mathbb{G}}}\int_{A}\sum_{k\in\mathbb{Z}^{n}}\Gamma_z(t,k\circ x;t_0,y)dy dm_0(x) \in [0,1].
\end{equation*}
In particular, we can get $\left\langle m(t),1_{\mathbb{T}_{\mathbb{G}}}\right\rangle=1$. Hence $m(t) \in \mathcal{P}(\mathbb{T}_{\mathbb{G}})$ for any $t\in[t_0,T]$.

As for the H\"{o}lder continuity, fix any $t_1,t_2\in[t_0,T]$ with $t_1\neq t_2$, and choose any $\xi\in C_\mathcal{X}^{0+1}(\mathbb{T}_{\mathbb{G}})$, we can find that
\begin{align}\label{d_1(m(t_1),m(t_2))}
& \int_{\mathbb{T}_{\mathbb{G}}}\xi(x) d\left(m(t_1)-m(t_2)\right)(x) \\
= & \int_{\mathbb{T}_{\mathbb{G}}}\int_{\mathbb{T}_{\mathbb{G}}}\sum_{k\in\mathbb{Z}^{n}}\left(\Gamma_z(t_1,k\circ x;t_0,y)-\Gamma_z(t_2,k\circ x;t_0,y)\right)\xi(y)dy dm_0(x) \notag\\
= & \int_{\mathbb{T}_{\mathbb{G}}}\left(z(T-t_1+t_0,x)-z(T-t_2+t_0,x)\right)dm_0(x).\notag
\end{align}
Here, by Lemma \ref{Lem_LHJB Lipschitz}, $z\in C_{{\mathcal{X}}}^{1,2} \left([t_0,T)\times\mathbb{T}_{\mathbb{G}}\right)\cap C([t_0,T] \times \mathbb{T}_{\mathbb{G}})$ is the unique solution to the equation
\begin{equation*}
\begin{cases}
-\partial_tz-\Delta_\mathcal{X}z+D_{\mathcal{X}}u\cdot D_{\mathcal{X}}z=0, & \text{ in } [t_0,T) \times \mathbb{T}_{\mathbb{G}}, \\
z(T)=\xi, & \text{ in } \mathbb{T}_{\mathbb{G}},
\end{cases}
\end{equation*}
and
\begin{equation*} \left\|z\left(T-t_1+t_0,\cdot\right)-z(T-t_2+t_0,\cdot)\right\|_{L^{\infty}\left(\mathbb{T}_{\mathbb{G}}\right)}
\leq C\left\|\xi\right\|_{C_{\mathcal{X}}^{0+1}\left(\mathbb{T}_{\mathbb{G}}\right)} \left|t_1-t_2\right|^{\frac{1}{2}},
\end{equation*}
where the constant $C>0$ depends on $\mathbb{G}$, $T$ and $\|D_{\mathcal{X}}u\|_{L^{\infty}\left((t_0,T)\times\mathbb{T}_{\mathbb{G}}\right)}$ only. Since for a fixed $x_0\in\mathbb{T}_{\mathbb{G}}$,
\begin{equation*}
\int_{\mathbb{T}_{\mathbb{G}}}\xi(x) d\left(m(t_1)-m(t_2)\right)(x)=\int_{\mathbb{T}_{\mathbb{G}}}\left(\xi(x)-\xi(x_0)\right) d\left(m(t_1)-m(t_2)\right)(x),
\end{equation*}
we can assume without loss of generality that $\xi(x_0)=0$ for some $x_0\in\mathbb{T}_{\mathbb{G}}$. This yields that $\left\|\xi\right\|_{C_{\mathcal{X}}^{0+1}\left(\mathbb{T}_{\mathbb{G}}\right)}\leq c\left[\xi\right]_{C_{\mathcal{X}}^{0+1}\left(\mathbb{T}_{\mathbb{G}}\right)}$. Hence, taking the supremum over $\xi\in C_\mathcal{X}^{0+1}(\mathbb{T}_{\mathbb{G}})$ with $d_{cc}^{\mathbb{T}_{\mathbb{G}}}$-Lipschitz constant bounded by $1$ in \eqref{d_1(m(t_1),m(t_2))}, we get
\begin{equation}\label{FPK_Horlder ctn.}
\sup_{t_1 \neq t_2}\frac{d_1(m(t_1),m(t_2))}{|t_1-t_2|^{\frac{1}{2}}} \leq C=:c_E,
\end{equation}
where the constant $C>0$ depends on $\mathbb{G}$, $T$ and $\|D_{\mathcal{X}}u\|_{L^{\infty}\left((t_0,T)\times\mathbb{T}_{\mathbb{G}}\right)}$ only. Thus leading to $m\in E$.

It remains for us to prove that $\Phi$ is continuous. For any $\mu^{i}\to \mu$ in $E$ as $i\to+\infty$, let $u^{i}$ and $m^{i}$ be the corresponding solutions to the equations \eqref{SFPT_HJB} and \eqref{SFPT_FPK} respectively. Using Lemma \ref{Lem_LHJB well-posedness} again, then, for any $i,j \in \mathbb{Z}_+$, $\bar{u}^{i,j}:=u^{i}-u^{j}\in C_{\mathcal{X}}^{1,2} \left([t_0,T] \times \mathbb{T}_{\mathbb{G}} \right)$ is the unique solution to the following linear degenerate PDE:
\begin{equation*}
\begin{cases}
-\partial_t \bar{u}^{i,j}-\Delta_{\mathcal{X}}\bar{u}^{i,j}+\frac{1}{2}D_{\mathcal{X}}\left(u^{i}+u^{j}\right)\cdot D_{\mathcal{X}}\bar{u}^{i,j}=F(x, \mu^{i}(t))-F(x, \mu^{j}(t)), & \text{in } [t_0,T] \times \mathbb{T}_{\mathbb{G}}, \\ \bar{u}^{i,j}(T,x)=G(x, \mu^{i}(T))-G(x, \mu^{j}(T)), & \text{in } \mathbb{T}_{\mathbb{G}}.
\end{cases}
\end{equation*}
From \cite[Proposition 4.1]{JWY}, \ref{assum1} and \ref{assum2}, we further obtain
\begin{align*}
& \left\|\bar{u}^{i,j}\right\|_{C^{\frac{\alpha}{2},k+\alpha}_{\mathcal{X}}\left([t_0,T]\times\mathbb{T}_{\mathbb{G}}\right)} +\left\|\bar{u}^{i,j}\right\|_{C^{\frac{2+\alpha}{2},2+\alpha}_{\mathcal{X}}\left([t_0,T]\times\mathbb{T}_{\mathbb{G}}\right)} \\
\leq & C\left(\left\|G(\cdot,\mu^{i}(T))-G(\cdot,\mu^{j}(T))\right\|_{C_{\mathcal{X}}^{k+\alpha}\left(\mathbb{T}_{\mathbb{G}}\right)} +\left\|F(\cdot, \mu^{i}(\cdot))-F(\cdot, \mu^{j}(\cdot))\right\|_{C_{\mathcal{X}}^{\frac{\alpha}{2},k-2+\alpha}\left([t_0,T]\times \mathbb{T}_{\mathbb{G}}\right)}\right) \\
\leq & C\left(\sup_{m \in \mathcal{P}(\mathbb{T}_{\mathbb{G}})}\left\|\frac{\delta G}{\delta m}(\cdot,m,\cdot)\right\|_{C_{\mathcal{X},y}^{0+1}\left(\mathbb{T}_{\mathbb{G}};C_{\mathcal{X},x}^{k+\alpha}\left(\mathbb{T}_{\mathbb{G}}\right)\right)} d_1(\mu^{i}(T),\mu^{j}(T))\right. \\
& \left.+\sup_{m \in \mathcal{P}(\mathbb{T}_{\mathbb{G}})}\left\|\frac{\delta F}{\delta m}(\cdot,m,\cdot)\right\|_{C_{\mathcal{X},y}^{0+1}\left(\mathbb{T}_{\mathbb{G}};C_{\mathcal{X},x}^{k-2+\alpha}\left(\mathbb{T}_{\mathbb{G}}\right)\right)} \left\|\mu^{i}-\mu^{j}\right\|_{E}\right. \\
& \left.+\operatorname{Lip}\left(\frac{\delta F}{\delta m}\right)\int_{0}^{1}\left\|(1-s)\mu^{i}+s\mu^{j}\right\|_{E}ds\sup_{t\in[t_0,T]}d_1(\mu^{i}(t),\mu^{j}(t))\right),
\end{align*}
where $C>0$ depends on $\mathbb{G}$, $\left\|u^{i}+u^{j}\right\|_{C_{\mathcal{X}}^{\frac{\alpha}{2},k+\alpha}\left([t_0,T]\times \mathbb{T}_{\mathbb{G}}\right)}$, $T$, $k$ and $\alpha$ only. 
Letting $i,j \to +\infty$, since $\left\{u^{i}\right\}_{i=1}^{+\infty}$ is uniformly bounded in $C_{\mathcal{X}}^{\frac{\alpha}{2},k+\alpha}\left([t_0,T]\times \mathbb{T}_{\mathbb{G}}\right)$ by \eqref{HJB_Schauder estimate}, we obtain that $\left\{u^{i}\right\}_{i=1}^{+\infty}$ is a Cauchy sequence. Thus, we have $u^{i} \to u$ in $C_{\mathcal{X}}^{\frac{2+\alpha}{2},2+\alpha}\left([t_0,T]\times\mathbb{T}_{\mathbb{G}}\right)\cap C^{\frac{\alpha}{2},k+\alpha}_{\mathcal{X}}\left([t_0,T]\times\mathbb{T}_{\mathbb{G}}\right)$ as $i\to+\infty
$ due to the uniqueness of the solution to equation \eqref{SFPT_HJB}.

Now we need to show $m^{i} \to m$ in $E$ as $i\to+\infty$. Choose any $\xi\in C_\mathcal{X}^{0+1}(\mathbb{T}_{\mathbb{G}})$ with $\xi(x_0)=0$ for some $x_0\in\mathbb{T}_{\mathbb{G}}$, and let $\xi^{\varepsilon}$ be the mollified versions of $\xi$, satisfying
\begin{equation*}
\left[\xi^{\varepsilon}\right] _{C_{\mathcal{X}}^{0+1}\left(\mathbb{T}_{\mathbb{G}}\right)} \leq C\left[\xi\right] _{C_{\mathcal{X}}^{0+1}\left(\mathbb{T}_{\mathbb{G}}\right)},\,\varepsilon>0
\end{equation*}
and $\left\|\xi^{\varepsilon}-\xi\right\|_{L^{\infty}\left(\mathbb{T}_{\mathbb{G}}\right)}\to 0$ as $\varepsilon \to 0$ (see \cite[Proposition 1.2]{JWY}). From Lemma \ref{Lem_LHJB well-posedness}, there exist unique solutions $z^{\varepsilon,i},z^{\varepsilon}\in C_{{\mathcal{X}}}^{1,2}\left([t_0,T]\times\mathbb{T}_{\mathbb{G}}\right)$ to the equation
\begin{equation*}\label{dual FPK with xi_varepsilon}
\begin{cases}
-\partial_tz-\Delta_\mathcal{X}z+b\cdot D_{\mathcal{X}}z=0, & \text{ in } [t_0,T] \times \mathbb{T}_{\mathbb{G}}, \\
z(T)=\xi^{\varepsilon}, & \text{ in } \mathbb{T}_{\mathbb{G}}
\end{cases}
\end{equation*}
with $b=D_{\mathcal{X}}u^{i}$ and $b=D_{\mathcal{X}}u$ respectively. Set $\bar{m}^{i}:=m^{i}-m$ and $\bar{z}^{\varepsilon,i}:=z^{\varepsilon,i}-z^{\varepsilon}$. In the same way as in \eqref{d_1(m(t_1),m(t_2))}, we obtain that for any $t\in[t_0,T]$,
\begin{equation*}
\int_{\mathbb{T}_{\mathbb{G}}}\xi(x) d\bar{m}^{i}(t)(x) =\lim_{\varepsilon\to0}\int_{\mathbb{T}_{\mathbb{G}}}\xi^{\varepsilon}(x) d\bar{m}^{i}(t)(x) =\lim_{\varepsilon\to0}\int_{\mathbb{T}_{\mathbb{G}}}\bar{z}^{\varepsilon,i}(T-t+t_0,x)dm_0(x).
\end{equation*}
Likewise, we know that $\bar{z}^{\varepsilon,i}\in C_{{\mathcal{X}}}^{1,2} \left([t_0,T]\times\mathbb{T}_{\mathbb{G}}\right)$ is the unique solution to the equation
\begin{equation*}
\begin{cases}
-\partial_t\bar{z}^{\varepsilon,i}-\Delta_\mathcal{X}\bar{z}^{\varepsilon,i}+D_{\mathcal{X}}u^{i}\cdot D_{\mathcal{X}}\bar{z}^{\varepsilon,i}=D_{\mathcal{X}}(u-u^{i})\cdot D_{\mathcal{X}}z^{\varepsilon}, & \text{ in } [t_0,T] \times \mathbb{T}_{\mathbb{G}}, \\
\bar{z}^{\varepsilon,i}(T)=0, & \text{ in } \mathbb{T}_{\mathbb{G}}.
\end{cases}
\end{equation*}
Using Lemma \ref{Lem_LHJB Lipschitz} and \cite[Theorem 1.3(1)]{25JWY}, we get
\begin{align*}
\int_{\mathbb{T}_{\mathbb{G}}}\xi(x) d\bar{m}^{i}(t)(x) \leq \lim_{\varepsilon\to0}\left\|\bar{z}^{\varepsilon,i}\right\|_{L^{\infty}\left([t_0,T]\times \mathbb{T}_{\mathbb{G}}\right)} \leq & \lim_{\varepsilon\to0}C\left\|D_{\mathcal{X}}(u-u^{i})\cdot D_{\mathcal{X}}z^{\varepsilon}\right\|_{L^{\infty}\left([t_0,T]\times \mathbb{T}_{\mathbb{G}}\right)} \\
\leq & C\left\|u-u^{i}\right\|_{C_{\mathcal{X}}^{0,1}\left([t_0,T]\times\mathbb{T}_{\mathbb{G}}\right)}\left[\xi\right]_{C_\mathcal{X}^{0+1}(\mathbb{T}_{\mathbb{G}})},
\end{align*}
where the constant $C>0$ depends on $\mathbb{G}$, $T$, $\|D_{\mathcal{X}}u^{i}\|_{L^{\infty}\left((0,T)\times\mathbb{T}_{\mathbb{G}}\right)}$ and $\|D_{\mathcal{X}}u\|_{L^{\infty}\left((0,T)\times\mathbb{T}_{\mathbb{G}}\right)}$ only. This leads to $\sup_{t\in[t_0,T]}d_1(m^{i}(t),m(t))\to0$ as $i\to+\infty$. Moreover, for any $t_1,t_2\in[t_0,T]$ with $t_1\neq t_2$, we also have
\begin{align*}
\int_{\mathbb{T}_{\mathbb{G}}}\xi(x) d\left(\bar{m}^{i}(t_1)-\bar{m}^{i}(t_2)\right)(x) = & \lim_{\varepsilon\to0}\int_{\mathbb{T}_{\mathbb{G}}}\left(\bar{z}^{\varepsilon,i}(T-t_1+t_0,x)-\bar{z}^{\varepsilon,i}(T-t_2+t_0,x)\right)dm_0(x) \\
\leq & \lim_{\varepsilon\to0}C\left\|D_{\mathcal{X}}(u-u^{i})\cdot D_{\mathcal{X}}z^{\varepsilon}\right\|_{L^{\infty}\left([t_0,T]\times \mathbb{T}_{\mathbb{G}}\right)}|t_1-t_2|^{\frac{1}{2}} \\
\leq & C\left\|u-u^{i}\right\|_{C_{\mathcal{X}}^{0,1}\left([t_0,T]\times\mathbb{T}_{\mathbb{G}}\right)}\left[\xi\right]_{C_\mathcal{X}^{0+1}(\mathbb{T}_{\mathbb{G}})}|t_1-t_2|^{\frac{1}{2}},
\end{align*}
where the constant $C>0$ depends on $\mathbb{G}$, $T$, $\|D_{\mathcal{X}}u^{i}\|_{L^{\infty}\left((0,T)\times\mathbb{T}_{\mathbb{G}}\right)}$ and $\|D_{\mathcal{X}}u\|_{L^{\infty}\left((0,T)\times\mathbb{T}_{\mathbb{G}}\right)}$ only. Therefore, we finally get $\|m^{i}-m\|_{E}\to0$ as $i\to+\infty$. This concludes the proof of continuity.

Applying the Schauder fixed point theorem, we obtain that there exists a solution
$$
(u,m) \in C_{\mathcal{X}}^{\frac{2+\alpha}{2},2+\alpha}\left([t_0,T]\times\mathbb{T}_{\mathbb{G}}\right) \times C^{\frac{1}{2}}\left([t_0,T];\mathcal{P}(\mathbb{T}_{\mathbb{G}})\right)
$$
to the MFG system \eqref{MFG system}, and the regularity results are given by \eqref{HJB_Schauder estimate} and \eqref{FPK_Horlder ctn.}. To prove the uniqueness, let $\left(u_{1},m_{1}\right)$ and $\left(u_{2},m_{2}\right)$ be two solutions to the MFG system \eqref{MFG system} with initial conditions $\left(t_0, m_0^1\right)$ and $\left(t_0, m_0^2\right)$ respectively. Then $\left(\bar{u},\bar{m}\right):=\left(u_{1}-u_{2},m_{1}-m_{2}\right)$ satisfies
\begin{equation*}
\begin{cases}
-\partial_t \bar{u}-\Delta_{\mathcal{X}}\bar{u}+\frac{1}{2}D_{\mathcal{X}}(u_{1}+u_{2})\cdot D_{\mathcal{X}}\bar{u}=F(x,m_{1}(t))-F(x,m_{2}(t)), & \mbox{in } [t_0,T] \times \mathbb{T}_{\mathbb{G}}, \\
\partial_t \bar{m}-\Delta_{\mathcal{X}} \bar{m}-\operatorname{div}_{\mathcal{X}}\left(\bar{m} D_{\mathcal{X}}u_{1}\right)=\operatorname{div}_{\mathcal{X}}\left(m_{2}D_{\mathcal{X}}\bar{u}\right), & \mbox{in } [t_0,T] \times \mathbb{T}_{\mathbb{G}}, \\
\bar{u}(T,x)=G(x,m_{1}(T))-G(x,m_{2}(T)),\quad \bar{m}(t_0)=m_0^1-m_0^2, & \mbox{in } \mathbb{T}_{\mathbb{G}}.
\end{cases}
\end{equation*}
Following the same derivation as for \eqref{general FPK_weak formulation} and using \eqref{monotonicity_F}, \eqref{monotonicity_G}, we can obtain
\begin{align}\label{MFG_L-L monoton. argument}
& \int_{\mathbb{T}_{\mathbb{G}}}\bar{u}(t_0,x)d\bar{m}(t_0)(x) -\int_{t_0}^{T}\int_{\mathbb{T}_{\mathbb{G}}}\left|D_{\mathcal{X}}\bar{u}(t,x)\right|^2dm_{2}(t)(x)dt \\
= & \int_{\mathbb{T}_{\mathbb{G}}}\left(G(x,m_{1}(T))-G(x,m_{2}(T))\right)d\bar{m}(T)(x) \notag\\
& +\int_{t_0}^{T}\int_{\mathbb{T}_{\mathbb{G}}}\left(F(x,m_{1}(t))-F(x,m_{2}(t))\right)d\bar{m}(t)(x)dt \geq 0.\notag
\end{align}
Since $\bar{m}(t_0)=0$, we get $D_{\mathcal{X}}\bar{u}=0$, hence $\bar{m}=0$ by the uniqueness of the weak solution to the FPK equation \eqref{general FPK}. This leads to $\bar{u}=0$ due to the uniqueness of the solution to the equation \eqref{LHJB}.

Moreover, if $m_0$ is absolutely continuous with a smooth positive density, then, by Lemma \ref{Lem_LHJB well-posedness}, there exists a unique solution $m\in C_{\mathcal{X}}^{1,2} \left([t_0, T] \times \mathbb{T}_{\mathbb{G}} \right)$ to the linear degenerate equation
\begin{equation*}
\begin{cases}
\partial_t m-\Delta_{\mathcal{X}} m-D_{\mathcal{X}}u\cdot D_{\mathcal{X}}m-\operatorname{div}_{\mathcal{X}}\left(D_{\mathcal{X}}u\right)m=0, & \text{ in } [t_0,T] \times \mathbb{T}_{\mathbb{G}}, \\
m(t_0)=m_0, & \text{ in } \mathbb{T}_{\mathbb{G}}.
\end{cases}
\end{equation*}
By \cite[Proposition 4.1]{JWY}, we have $m \in C_{\mathcal{X}}^{\frac{2+\alpha}{2},2+\alpha}\left([t_0,T]\times\mathbb{T}_{\mathbb{G}}\right)\cap C^{\frac{\alpha}{2},k+\alpha}_{\mathcal{X}}\left([t_0,T]\times\mathbb{T}_{\mathbb{G}}\right)$. Since $m_0>0$, similar to \eqref{w>0}, it can be obtained that $m>0$.

Furthermore, if $m_0^{i} \to m_0$ in $\mathcal{P}\left(\mathbb{T}_{\mathbb{G}}\right)$, it follows from \eqref{MFG_regularity} that $\{(u^{i},m^{i})\}_{i=1}^{+\infty}$ is uniformly bounded in $C_{\mathcal{X}}^{\frac{2+\alpha}{2},2+\alpha}\left([t_0,T]\times\mathbb{T}_{\mathbb{G}}\right)\cap C^{\frac{\alpha}{2},k+\alpha}_{\mathcal{X}}\left([t_0,T]\times\mathbb{T}_{\mathbb{G}}\right)\times C^{\frac{1}{2}}\left([t_0,T];\mathcal{P}(\mathbb{T}_{\mathbb{G}})\right)$. For any $i,j\in\mathbb{Z}_+$, $\left(\hat{u}^{i,j},\hat{m}^{i,j}\right):=\left(u^{i}-u^{j},m^{i}-m^{j}\right)$ satisfies
\begin{equation}\label{linear MFG system_hat u,hat m}
\begin{cases}
-\partial_t \hat{u}^{i,j}-\Delta_{\mathcal{X}}\hat{u}^{i,j}+\frac{1}{2}D_{\mathcal{X}}(u^{i}+u^{j})\cdot D_{\mathcal{X}}\hat{u}^{i,j}=F(x,m^{i}(t))-F(x,m^{j}(t)), & \mbox{in } [t_0,T] \times \mathbb{T}_{\mathbb{G}}, \\
\partial_t \hat{m}^{i,j}-\Delta_{\mathcal{X}} \hat{m}^{i,j}-\operatorname{div}_{\mathcal{X}}\left(\hat{m}^{i,j} D_{\mathcal{X}}u^{i}\right)=\operatorname{div}_{\mathcal{X}}\left(m^{j}D_{\mathcal{X}}\hat{u}^{i,j}\right), & \mbox{in } [t_0,T] \times \mathbb{T}_{\mathbb{G}}, \\
\hat{u}^{i,j}(T,x)=G(x,m^{i}(T))-G(x,m^{j}(T)),\quad \hat{m}^{i,j}(t_0)=m_0^i-m_0^j, & \mbox{in } \mathbb{T}_{\mathbb{G}}.
\end{cases}
\end{equation}
Using Lemma \ref{Lem_LHJB wellposed regularity}, \ref{assum1} and \ref{assum2}, we have that for any $t\in[t_0,T]$,
\begin{align}\label{hat u_i,j_USE}
& \sup_{\tau\in[t_0, t]}\|\hat{u}^{i,j}(\tau, \cdot)\|_{C_{\mathcal{X}}^{k+\alpha}\left(\mathbb{T}_{\mathbb{G}}\right)} \notag\\
\leq & C\left(\left\|G(x,m^{i}(t))-G(x,m^{j}(t))\right\|_{C_{\mathcal{X}}^{k+\alpha}\left(\mathbb{T}_{\mathbb{G}}\right)}\right. \notag\\
& \left.+\sup_{\tau\in(t_0,t)}\left\|F(\cdot,m^{i}(\tau))-F(\cdot,m^{j}(\tau))\right\|_{C_{\mathcal{X}}^{k-1+\alpha}\left(\mathbb{T}_{\mathbb{G}}\right)}\right) \notag\\
\leq & C\left(\sup_{m \in \mathcal{P}(\mathbb{T}_{\mathbb{G}})}\left\|\frac{\delta G}{\delta m}(\cdot,m,\cdot)\right\|_{C_{\mathcal{X},y}^{0+1}\left(\mathbb{T}_{\mathbb{G}};C_{\mathcal{X},x}^{k+\alpha}\left(\mathbb{T}_{\mathbb{G}}\right)\right)} d_1(m^{i}(t),m^{j}(t))\right. \\
& +\left.\sup_{m \in \mathcal{P}(\mathbb{T}_{\mathbb{G}})}\left\|\frac{\delta F}{\delta m}(\cdot,m,\cdot)\right\|_{C_{\mathcal{X},y}^{0+1}\left(\mathbb{T}_{\mathbb{G}};C_{\mathcal{X},x}^{k-1+\alpha}\left(\mathbb{T}_{\mathbb{G}}\right)\right)} \sup_{\tau\in(t_0,t)}d_1(m^{i}(\tau),m^{j}(\tau))\right),\notag
\end{align}
where $C>0$ depends on $\mathbb{G}$, $\alpha$, $k$, $T$ and $\sup_{t\in(t_0,T)}\|(u^{i}+u^{j})(t,\cdot)\|_{C_{\mathcal{X}}^{k+\alpha}\left(\mathbb{T}_{\mathbb{G}}\right)}$ only.

Choose any $\xi\in C_\mathcal{X}^{0+1}(\mathbb{T}_{\mathbb{G}})$ with $\xi(x_0)=0$ for some $x_0\in\mathbb{T}_{\mathbb{G}}$, and let $\xi^{\varepsilon}$ be the mollified versions, satisfying
\begin{equation*}
\left[\xi^{\varepsilon}\right] _{C_{\mathcal{X}}^{0+1}\left(\mathbb{T}_{\mathbb{G}}\right)} \leq C\left[\xi\right] _{C_{\mathcal{X}}^{0+1}\left(\mathbb{T}_{\mathbb{G}}\right)},\,\varepsilon>0
\end{equation*}
and $\left\|\xi^{\varepsilon}-\xi\right\|_{L^{\infty}\left(\mathbb{T}_{\mathbb{G}}\right)}\to 0$ as $\varepsilon \to 0$ (see \cite[Proposition 1.2]{JWY}). From Lemma \ref{Lem_LHJB well-posedness}, there exist a unique solution $z^{\varepsilon}\in C_{{\mathcal{X}}}^{1,2}\left([t_0,t]\times\mathbb{T}_{\mathbb{G}}\right)$ to the equation \eqref{dual FPK} with $b=D_{\mathcal{X}}u^{i}$, $f=0$ and $\xi=\xi^{\varepsilon}$. By referring to \eqref{general FPK_weak formulation} and using \cite[Theorem 1.3(1)]{25JWY}, we have
\begin{align*}
\int_{\mathbb{T}_{\mathbb{G}}}\xi(x) d\hat{m}^{i,j}(t)(x) = & \lim_{\varepsilon\to0}\int_{\mathbb{T}_{\mathbb{G}}}z^{\varepsilon}(t_0,x)d\hat{m}^{i,j}(t_0)(x) \\
& +\lim_{\varepsilon\to0}\int_{t_0}^{t}\int_{\mathbb{T}_{\mathbb{G}}} D_{\mathcal{X}}z^{\varepsilon}(s,x)\cdot D_{\mathcal{X}}\hat{u}^{i,j}(s,x)dm^{j}(s)(x)ds \\
\leq & C\left[\xi\right]_{C_\mathcal{X}^{0+1}(\mathbb{T}_{\mathbb{G}})}\left(d_1(m_0^{i},m_0^{j}) +\int_{t_0}^{t}\left\|\hat{u}^{i,j}\right\|_{C_{\mathcal{X}}^{0,1}\left([t_0,s]\times\mathbb{T}_{\mathbb{G}}\right)}ds\right),
\end{align*}
where $C>0$ depends on $\mathbb{G}$, $\|u^{i}\|_{C_{\mathcal{X}}^{0,1}\left((t_0,T)\times\mathbb{T}_{\mathbb{G}}\right)}$ and $T$ only. Hence,
\begin{equation*}
d_1(m^{i}(t),m^{j}(t)) \leq C\left(d_1(m_0^{i},m_0^{j}) +\int_{t_0}^{t}\left\|\hat{u}^{i,j}\right\|_{C_{\mathcal{X}}^{0,1}\left([t_0,s]\times\mathbb{T}_{\mathbb{G}}\right)}ds\right).
\end{equation*}
Putting this into \eqref{hat u_i,j_USE} yields that
\begin{equation*}
\sup_{\tau\in[t_0, t]}\|\hat{u}^{i,j}(\tau, \cdot)\|_{C_{\mathcal{X}}^{k+\alpha}\left(\mathbb{T}_{\mathbb{G}}\right)} \leq C\left(d_1(m_0^{i},m_0^{j}) +\int_{t_0}^{t}\left\|\hat{u}^{i,j}\right\|_{C_{\mathcal{X}}^{0,1}\left([t_0,s]\times\mathbb{T}_{\mathbb{G}}\right)}ds\right)
\end{equation*}
where $C>0$ depends on $\mathbb{G}$, $\alpha$, $k$, $T$, $c_F$ and $c_G$ only. Applying Gronwall's inequality, we get
\begin{equation*}
\sup_{t\in[t_0,T]}\|\hat{u}^{i,j}(t, \cdot)\|_{C_{\mathcal{X}}^{k+\alpha}\left(\mathbb{T}_{\mathbb{G}}\right)} \leq C d_1(m_0^{i},m_0^{j})\to0 \text{ as }i,j\to+\infty,
\end{equation*}
thus $\sup_{t\in[t_0,T]}d_1(m^{i}(t),m^{j}(t))\to0$ as $i,j\to+\infty$. Letting $i,j\to+\infty$ in \eqref{linear MFG system_hat u,hat m}, we obtain that $\{(u^{i},m^{i})\}_{i=1}^{+\infty}$ is the Cauchy sequence in $C_{\mathcal{X}}^{1,2}\left([t_0,T]\times\mathbb{T}_{\mathbb{G}}\right)\times C\left([t_0,T];\mathcal{P}(\mathbb{T}_{\mathbb{G}})\right)$. According to the stability of the weak solution to the equation \eqref{general FPK} and the uniqueness of the solution to the MFG system \eqref{MFG system}, we finally obtain that $(u^{i},m^{i})\to(u,m)$ in $C_{\mathcal{X}}^{1,2}\left([t_0,T]\times\mathbb{T}_{\mathbb{G}}\right)\times C\left([t_0,T];\mathcal{P}(\mathbb{T}_{\mathbb{G}})\right)$ as $i\to+\infty$.

Finally, the H\"{o}lder regularity of $U$ with respect to $x$ is directly obtained from \eqref{HJB_Schauder estimate}. For any multi-index $I$ with $|I|\leq 2$, the continuity of $X_I U$ in $(t_0,m_0)$ uniformly with respect to $x$ can be simply deduced from the stability of the solution to the MFG system \eqref{MFG system}. Hence $X_I U\in C\left([0,T]\times\mathbb{T}_{\mathbb{G}}\times \mathcal{P}(\mathbb{T}_{\mathbb{G}})\right)$. We have thus completed the proof.
\end{proof}

\subsection{Lipschitz continuity of \texorpdfstring{$U$}{U} in \texorpdfstring{$m_0$}{m0}}
\begin{proof}[\bf{Proof of Proposition \ref{Prop_Lip. ctn. of U}}]
We may assume without loss of generality that $t_{0}=0$.

\emph{Step 1: Monotonicity argument.} Let $\bar{u}:=u_1-u_2$. According to \eqref{MFG_L-L monoton. argument}, since
$$
\left|\bar{u}(t_0,x)-\bar{u}(t_0,y)\right| \leq \left\|D_{\mathcal{X}}\bar{u}(t_0,\cdot)\right\|_{L^{\infty}\left(\mathbb{T}_{\mathbb{G}}\right)} d_{cc}^{\mathbb{T}_{\mathbb{G}}}(x,y),\,t_0\in[0,T],\,x,y\in\mathbb{T}_{\mathbb{G}}
$$
based on \cite[Proposition 4.2(ii)]{07BB}, we have
\begin{equation}\label{Mono. argu.}
\int_{t_0}^{T}\int_{\mathbb{T}_{\mathbb{G}}}\left|D_{\mathcal{X}}\bar{u}(t,x)\right|^2dm_{2}(t)(x)dt \leq \left\|D_{\mathcal{X}}\bar{u}(t_0,\cdot)\right\|_{L^{\infty}\left(\mathbb{T}_{\mathbb{G}}\right)} d_1(m_0^1,m_0^2).
\end{equation}

\emph{Step 2: The estimate on $m_1-m_2$.} Let $\bar{m}:=m_1-m_2$. We assume without loss of generality that $m_0^2$ has a smooth density. Choose any $\xi\in C_\mathcal{X}^{0+1}(\mathbb{T}_{\mathbb{G}})$ with $\xi(x_0)=0$ for some $x_0\in\mathbb{T}_{\mathbb{G}}$, and let $\xi^{\varepsilon}$ be the mollified versions, satisfying
\begin{equation*}
\left[\xi^{\varepsilon}\right] _{C_{\mathcal{X}}^{0+1}\left(\mathbb{T}_{\mathbb{G}}\right)} \leq C\left[\xi\right] _{C_{\mathcal{X}}^{0+1}\left(\mathbb{T}_{\mathbb{G}}\right)},\,\varepsilon>0
\end{equation*}
and $\left\|\xi^{\varepsilon}-\xi\right\|_{L^{\infty}\left(\mathbb{T}_{\mathbb{G}}\right)}\to 0$ as $\varepsilon \to 0$ (see \cite[Proposition 1.2]{JWY}). By Lemma \ref{Lem_LHJB well-posedness}, there exists a unique solution $z^{\varepsilon}\in C_{{\mathcal{X}}}^{1,2} \left([t_0,t]\times\mathbb{T}_{\mathbb{G}}\right)$ to the equation \eqref{dual FPK} with $b=D_{\mathcal{X}}u_1$, $f=0$ and $\xi=\xi^{\varepsilon}$. We refer to \eqref{general FPK_weak formulation} to obtain that for any $t\in[t_0,T]$,
\begin{align*}
\int_{\mathbb{T}_{\mathbb{G}}}\xi(x) d\bar{m}(t)(x) = & \lim_{\varepsilon\to0}\left(\int_{\mathbb{T}_{\mathbb{G}}}z^{\varepsilon}(t_0,x)d\bar{m}(t_0)(x) +\int_{t_0}^{t}\int_{\mathbb{T}_{\mathbb{G}}}\left(D_{\mathcal{X}}\bar{u}\cdot D_{\mathcal{X}}z^{\varepsilon}\right)(s,x)dm_2(s)(x)ds\right) \\
\leq & C\left[\xi\right]_{C_{\mathcal{X}}^{0+1}\left(\mathbb{T}_{\mathbb{G}}\right)}\left(d_1\left(m_0^1,m_0^2\right) +\int_{t_0}^{t}\int_{\mathbb{T}_{\mathbb{G}}}\left|D_{\mathcal{X}}\bar{u}(s,x)\right|dm_2(s)(x)ds\right),
\end{align*}
where the last inequality follows from Lemma \ref{Lem_LHJB Lipschitz}, \cite[Theorem 1.3(1)]{25JWY} and $\left\|\xi\right\|_{C_{\mathcal{X}}^{0+1}\left(\mathbb{T}_{\mathbb{G}}\right)}\leq c\left[\xi\right]_{C_{\mathcal{X}}^{0+1}\left(\mathbb{T}_{\mathbb{G}}\right)}$. Here, the constant $C>0$ depends on $\mathbb{G}$, $\|D_{\mathcal{X}}u_1\|_{L^{\infty}\left((t_0,T)\times\mathbb{T}_{\mathbb{G}}\right)}$ and $T$ only. Using Jensen's inequality and \eqref{Mono. argu.}, we get that for any $t\in[t_0,T]$,
\begin{align*}
\int_{\mathbb{T}_{\mathbb{G}}}\xi(x) d\bar{m}(t)(x) & \leq C\left[\xi\right]_{C_{\mathcal{X}}^{0+1}\left(\mathbb{T}_{\mathbb{G}}\right)}\left(d_1\left(m_0^1,m_0^2\right) +\left(\int_{t_0}^{t}\int_{\mathbb{T}_{\mathbb{G}}}\left|D_{\mathcal{X}}\bar{u}(s,x)\right|^2dm_2(s)(x)ds\right)^{\frac{1}{2}}\right) \\
& \leq C\left[\xi\right]_{C_{\mathcal{X}}^{0+1}\left(\mathbb{T}_{\mathbb{G}}\right)}\left(\left\|D_{\mathcal{X}}\bar{u}(t_0,\cdot)\right\|_{L^{\infty}\left(\mathbb{T}_{\mathbb{G}}\right)}^{\frac{1}{2}} d_1\left(m_0^1,m_0^2\right)^{\frac{1}{2}}+d_1\left(m_0^1,m_0^2\right)\right).
\end{align*}
Hence, we finally get
\begin{equation}\label{m_1-m_2}
\sup_{t \in [t_0,T]} d_1\left(m_1(t), m_2(t)\right) \leq C\left(\left\|D_{\mathcal{X}}\bar{u}(t_0,\cdot)\right\|_{L^{\infty}\left(\mathbb{T}_{\mathbb{G}}\right)}^{\frac{1}{2}} d_1\left(m_0^1,m_0^2\right)^{\frac{1}{2}}+d_1\left(m_0^1,m_0^2\right)\right),
\end{equation}
where the constant $C>0$ depends on $\mathbb{G}$, $\|D_{\mathcal{X}}u_1\|_{L^{\infty}\left((t_0,T)\times\mathbb{T}_{\mathbb{G}}\right)}$ and $T$ only.

\emph{Step 3: The estimate on $u_1-u_2$.} We note that $\bar{u}$ satisfies
$$
\begin{cases}
-\partial_t \bar{u}(t,x)-\Delta_{\mathcal{X}} \bar{u}(t,x)+b(t,x) \cdot D_{\mathcal{X}} \bar{u}(t,x)=R_1(t, x),& \text { in }[t_0,T] \times \mathbb{T}_{\mathbb{G}}, \\
\bar{u}(T,x)=R_T(x),& \text { in }\mathbb{T}_{\mathbb{G}},
\end{cases}
$$
where, for any $(t,x) \in [t_0,T] \times \mathbb{T}_{\mathbb{G}}$,
$$
b(t,x)=\int_{0}^1 D_pH\left(x,sD_\mathcal{X} u_1(t,x)+(1-s) D_\mathcal{X} u_2(t,x)\right)ds,
$$
$$
R_1(t,x)=\int_{0}^1 \int_{\mathbb{T}_{\mathbb{G}}} \frac{\delta F}{\delta m}\left(x,s m_1(t)+(1-s) m_2(t), y\right)d\bar{m}(t)(y)ds
$$
and
$$
R_T(t,x)=\int_{0}^1 \int_{\mathbb{T}_{\mathbb{G}}} \frac{\delta G}{\delta m}\left(x,s m_1(T)+(1-s)m_2(T),y\right)d\bar{m}(T)(y)ds.
$$
From \ref{assum1} and \eqref{m_1-m_2}, we obtain that for any $t \in [t_0,T]$,
\begin{align*}
\left\|R_1(t,\cdot)\right\|_{C_{\mathcal{X}}^{k-1+\alpha}(\mathbb{T}_{\mathbb{G}})} \leq & \sup _{m \in \mathcal{P}(\mathbb{T}_{\mathbb{G}})}\left\|\frac{\delta F}{\delta m}(\cdot,m,\cdot)\right\|_{C_{\mathcal{X},y}^{0+1}\left(\mathbb{T}_{\mathbb{G}};C_{\mathcal{X},x}^{k-1+\alpha}\left(\mathbb{T}_{\mathbb{G}}\right)\right)} d_1(m_1(t),m_2(t))\\
\leq & c_{F}C\left(\left\|D_{\mathcal{X}}\bar{u}(t_0,\cdot)\right\|_{L^{\infty}\left(\mathbb{T}_{\mathbb{G}}\right)}^{\frac{1}{2}} d_1\left(m_0^1,m_0^2\right)^{\frac{1}{2}}+d_1\left(m_0^1,m_0^2\right)\right).
\end{align*}
Similarly, from \ref{assum2}, we also have
$$
\left\|R_T\right\|_{C_{\mathcal{X}}^{k+\alpha}(\mathbb{T}_{\mathbb{G}})} \leq c_{G}C\left(\left\|D_{\mathcal{X}}\bar{u}(t_0,\cdot)\right\|_{L^{\infty}\left(\mathbb{T}_{\mathbb{G}}\right)}^{\frac{1}{2}} d_1\left(m_0^1,m_0^2\right)^{\frac{1}{2}}+d_1\left(m_0^1,m_0^2\right)\right).
$$
Hence, using Lemma \ref{Lem_LHJB wellposed regularity} and $\epsilon$-Cauchy inequality, we obtain
\begin{align*}
\sup_{t \in [t_0,T]}\|\bar{u}(t,\cdot)\|_{C_{\mathcal{X}}^{k+\alpha}(\mathbb{T}_{\mathbb{G}})} \leq & C\left(\left\|R_{T}\right\|_{C_{\mathcal{X}}^{k+\alpha}(\mathbb{T}_{\mathbb{G}})}+\sup_{t \in [t_0,T]}\left\|R_1(t,\cdot)\right\|_{C_{\mathcal{X}}^{k-1+\alpha}(\mathbb{T}_{\mathbb{G}})}\right)\\
\leq & C\left(\left\|D_{\mathcal{X}}\bar{u}(t_0,\cdot)\right\|_{L^{\infty}\left(\mathbb{T}_{\mathbb{G}}\right)}^{\frac{1}{2}} d_1\left(m_0^1,m_0^2\right)^{\frac{1}{2}}+d_1\left(m_0^1,m_0^2\right)\right)\\
\leq & C\left(\epsilon\left\|D_\mathcal{X}\bar{u}(t_0,\cdot)\right\|_{L^{\infty}\left(\mathbb{T}_{\mathbb{G}}\right)}+\left(1+\frac{1}{\epsilon}\right) d_1\left(m_0^1, m_0^2\right)\right),
\end{align*}
where the constants $C>0$ depend on $\mathbb{G}$, $\alpha$, $k$, $T$, $c_F$, $c_G$, $\sup_{t\in(t_0,T)}\|b(t,\cdot)\|_{C_{\mathcal{X}}^{k-1+\alpha}\left(\mathbb{T}_{\mathbb{G}};\mathbb{R}^{n_1}\right)}$ and $\|D_{\mathcal{X}}u_1\|_{L^{\infty}\left((t_0,T)\times\mathbb{T}_{\mathbb{G}}\right)}$ only. Since $H\in C^{\infty}\left(\mathbb{T}_{\mathbb{G}}\times\mathbb{R}^{n_1}\right)$, we note that
\begin{align*}
& \sup_{t\in(t_0,T)}\|b(t,\cdot)\|_{C_{\mathcal{X}}^{k-1+\alpha}\left(\mathbb{T}_{\mathbb{G}};\mathbb{R}^{n_1}\right)} \\
\leq & \sup_{t\in(t_0,T),s\in(0,1)}\left\|D_pH\left(\cdot,sD_\mathcal{X} u_1(t,\cdot)+(1-s) D_\mathcal{X} u_2(t,\cdot)\right)\right\|_{C_{\mathcal{X}}^{k-1+\alpha}\left(\mathbb{T}_{\mathbb{G}}\right)}<+\infty.
\end{align*}
Choose $\epsilon>0$ small enough, we therefore get
\begin{equation}\label{u1-u2_regularity}
\sup_{t \in[t_0, T]}\|\bar{u}(t,\cdot)\|_{C_{\mathcal{X}}^{k+\alpha}(\mathbb{T}_{\mathbb{G}})} \leq Cd_1\left(m_0^1, m_0^2\right).
\end{equation}
where the constant $C>0$ depends on $\mathbb{G}$, $\alpha$, $k$, $T$, $c_{F}$, $c_{G}$ and $H$ only. Putting the above inequality into \eqref{m_1-m_2}, we can find
$$
\sup_{t \in[t_0,T]} d_1\left(m_1(t), m_2(t)\right) \leq C d_1\left(m_0^1,m_0^2\right).
$$

Finally, it is evident that \eqref{U_Lip. in m} can be directly derived from \eqref{u1-u2_regularity}. We have completed the proof.
\end{proof}

\section{Linearized MFG system and differentiability of \texorpdfstring{$U$}{U}}\label{Sec_4}

In this section, we first study the linearized MFG system \eqref{linear MFG system} to derive Lemma \ref{Lem_relation of z and rho_0}. Subsequently, we provide the proof of Proposition \ref{Prop_U C^1 w.r.t. m}.

\subsection{Results for the linearized MFG system}
Let us start to work on obtaining some regularity estimates for the system \eqref{linear MFG system}. To do this, we consider a more general linearized system of the following form:
\begin{equation}\label{general linear MFG system}
\begin{cases}
-\partial_t z-\Delta_{\mathcal{X}} z+b \cdot D_{\mathcal{X}} z=\frac{\delta F}{\delta m}(x,m(t))(\rho(t))+f,& \text { in }[t_0, T] \times \mathbb{T}_{\mathbb{G}}, \\
\partial_t\rho-\Delta_\mathcal{X}\rho-\operatorname{div}_{\mathcal{X}}(\rho b) =\operatorname{div}_{\mathcal{X}}(m P D_{\mathcal{X}}z+\Upsilon),& \text { in }[t_0, T] \times \mathbb{T}_{\mathbb{G}}, \\
z(T, x)=\frac{\delta G}{\delta m}(x,m(T))(\rho(T))+g,\quad \rho(t_0)=\rho
_0,& \text { in }\mathbb{T}_{\mathbb{G}}.
\end{cases}
\end{equation}
We provide the definition of the solution to this system as follows.
\begin{Def}\label{Def_linear system sol.}
Let $k\in\{2,3,\ldots\}$. A couple $(z,\rho)$ is said to be a solution to the system \eqref{general linear MFG system} if
\begin{enumerate}[label=(\arabic*), leftmargin=*, itemindent=1.7em]
\item $z \in C_{\mathcal{X}}^{1,2}([t_0,T] \times \mathbb{T}_{\mathbb{G}})$ is a classical solution to the first linear equation;
\item $\rho \in C\left([t_0,T];C_\mathcal{X}^{-(k-1+\alpha)}(\mathbb{T}_{\mathbb{G}})\right)$ is a weak solution to the FPK equation in the sense of Definition \ref{Def_general FPK weak sol.}.
\end{enumerate}
\end{Def}
\begin{Rem}\label{Rem_upsilon=div(Upsilon)}
The second equation in system \eqref{general linear MFG system} is a special type of FPK equation \eqref{general FPK}, where the distribution $\upsilon$ takes the form
$$
\upsilon(t)=\operatorname{div}_{\mathcal{X}}(\Upsilon(t,\cdot)),\,t\in[0,T],
$$
for some integrable function $\Upsilon:[0,T] \times \mathbb{T}_{\mathbb{G}} \to \mathbb{R}^{n_1}$. Under this case, we can guarantee condition $\upsilon \in L^{1}\left([0,T];C_{\mathcal{X}}^{-k}\left([0,1)^n\right)\cap C_{\mathcal{X}}^{-(k+\alpha)}(\mathbb{T}_{\mathbb{G}})\right)$ by simply requiring $\Upsilon \in L^{1}\left([0,T] \times \mathbb{T}_{\mathbb{G}}\right)$ and $\int_{0}^{T}\int_{\partial [0,1)^n}|\Upsilon(t,x)|dS dt<+\infty$. Indeed, by integration by parts, we obtain
\begin{align*}
\left\|\upsilon\right\|_{L^{1}\left([0,T];C_{\mathcal{X}}^{-(k+\alpha)}\left(\mathbb{T}_{\mathbb{G}}\right)\right)} = & \int_{0}^{T}\sup_{\left\|w\right\|_{C_{\mathcal{X}}^{k+\alpha}\left(\mathbb{T}_{\mathbb{G}}\right)} \leq 1}\left\langle\operatorname{div}_{\mathcal{X}}(\Upsilon(t)),w\right\rangle dt \\
= & \int_{0}^{T}\left(\sup_{\left\|w\right\|_{C_{\mathcal{X}}^{k+\alpha}\left(\mathbb{T}_{\mathbb{G}}\right)} \leq 1}-\int_{\mathbb{T}_{\mathbb{G}}}\Upsilon(t,x) \cdot D_{\mathcal{X}}w(x)dx\right)dt \\
\leq & \int_{0}^{T}\int_{\mathbb{T}_{\mathbb{G}}}\left|\Upsilon(t,x)\right|dxdt=\left\|\Upsilon\right\|_{L^1\left([0,T] \times \mathbb{T}_{\mathbb{G}}\right)},
\end{align*}
where $w$ is the test function, and
\begin{align*}
\left\|\upsilon\right\|_{L^{1}\left([0,T];C_{\mathcal{X}}^{-k}\left([0,1)^n\right)\right)} 
\leq & \int_{0}^{T}\left(\int_{[0,1)^n}\left|\Upsilon(t,x)\right|dx+\int_{\partial [0,1)^n}|\Upsilon(t,x)|dS\right)dt.
\end{align*}
Particularly, if $\Upsilon(t)\in\mathcal{P}\left(\mathbb{T}_{\mathbb{G}}\right)$ for any $t\in[0,T]$, assume without loss of generality that $\Upsilon(t)$ is absolutely continuous with respect to the Lebesgue measure, then the above inequality can be written as
\begin{equation*}
\left\|\upsilon\right\|_{L^{1}\left([0,T];C_{\mathcal{X}}^{-k}\left([0,1)^n\right)\right)} \leq C\int_{0}^{T}\int_{[0,1)^n}\left|\Upsilon(t,x)\right|dxdt=C\left\|\Upsilon\right\|_{L^1\left([0,T] \times \mathbb{T}_{\mathbb{G}}\right)}.
\end{equation*}
Without ambiguity, we still denote $\int_{0}^{T}\int_{\mathbb{T}_{\mathbb{G}}}d\left|\Upsilon(t)\right|(x)dt=\left\|\Upsilon\right\|_{L^1\left([0,T] \times \mathbb{T}_{\mathbb{G}}\right)}$.
\end{Rem}
Now, we proceed to state the existence, uniqueness and regularity results for the system \eqref{general linear MFG system}.
\begin{Lem}\label{Lem_GL system regularity}
Let assumptions \ref{assum1}-\ref{assum2} hold and $t_0 \in[0, T]$. Assume that $b\in C_{\mathcal{X}}^{\frac{\alpha}{2},k-2+\alpha}\left([t_0, T] \times \mathbb{T}_{\mathbb{G}};\mathbb{R}^{n_1}\right)\cap B\left([t_0,T];C_{\mathcal{X}}^{k-1+\alpha}\left(\mathbb{T}_{\mathbb{G}};\mathbb{R}^{n_1}\right)\right)$, $m \in C^{\frac{1}{2}}([t_0,T];\mathcal{P}(\mathbb{T}_{\mathbb{G}}))$, $\rho_0 \in C_{\mathcal{X}}^{-(k-1)}\left([0,1)^n\right)\cap C_{\mathcal{X}}^{-(k-1+\alpha)}(\mathbb{T}_{\mathbb{G}})$, $P:[t_0,T] \times \mathbb{T}_{\mathbb{G}} \to \mathbb{R}^{n_1 \times n_1}$ is continuous which maps into the family of real symmetric matrices, $f\in C\left([t_0,T]\times\mathbb{T}_{\mathbb{G}}\right)\cap B\left([t_0,T];C_{\mathcal{X}}^{k-1+\alpha}\left(\mathbb{T}_{\mathbb{G}}\right)\right)$, $\Upsilon\in L^1\left([t_0,T]\times\mathbb{T}_{\mathbb{G}};\mathbb{R}^{n_1}\right)$ with $\int_{0}^{T}\int_{\partial [0,1)^n}|\Upsilon(t,x)|dS dt<+\infty$, $g\in C_{\mathcal{X}}^{k+\alpha}\left(\mathbb{T}_{\mathbb{G}}\right)$, and there is a constant $\bar{c}>0$ such that
\begin{equation}\label{Assump_linear system}
\begin{gathered}
d_1(m(t_1),m(t_2)) \leq \bar{c}|t_1-t_2|^{\frac{1}{2}},\, t_1,t_2 \in [t_0,T], \\
\bar{c}^{-1}I_{n_1 \times n_1} \leq P(t,x) \leq \bar{c} I_{n_1 \times n_1},\, (t,x) \in [t_0,T]\times\mathbb{T}_{\mathbb{G}}.
\end{gathered}
\end{equation}
Then, system \eqref{general linear MFG system} has a unique solution $(z,\rho)$ in the sense of Definition \ref{Def_linear system sol.}, satisfying
\begin{equation}\label{linear system_z regularity}
\sup _{t \in[t_0, T]}\|z(t, \cdot)\|_{C_{\mathcal{X}}^{k+\alpha}\left(\mathbb{T}_{\mathbb{G}}\right)} \leq CM,\,\sup _{\substack{t \neq t'\\t,t' \in [t_0,T-\epsilon]}} \frac{\left\|z\left(t^{\prime}, \cdot\right)-z(t, \cdot)\right\|_{C_{\mathcal{X}}^{k+\alpha}\left(\mathbb{T}_{\mathbb{G}}\right)}}{\left|t^{\prime}-t\right|^{\frac{1}{2}}} \leq \epsilon^{-\frac{1}{2}}CM
\end{equation}
for any $\epsilon\in(0,T-t_0)$, and
\begin{equation}\label{linear system_rho regularity}
\sup_{t \in \left[t_{0},T\right]}\|\rho(t)\|_{C_\mathcal{X}^{-(k-1+\alpha)}(\mathbb{T}_{\mathbb{G}})} + \sup _{\substack{t \neq t'\\t,t' \in [t_0,T]}} \frac{\left\|\rho\left(t^{\prime}\right)-\rho(t)\right\|_{C_\mathcal{X}^{-(k+\alpha)}(\mathbb{T}_{\mathbb{G}})}}{\left|t^{\prime}-t\right|^{\frac{1}{2}}} \leq CM,
\end{equation}
where the constants $C>0$ depend on $\mathbb{G}$, $\alpha$, $k$, $T$, $\sup_{t\in(t_0,T)}\|b(t,\cdot)\|_{C_{\mathcal{X}}^{k-1+\alpha}\left(\mathbb{T}_{\mathbb{G}};\mathbb{R}^{n_1}\right)}$, $c_F$, $c_G$ and $\bar{c}$ only, and $M$ is given by
\begin{align*}
M:= & \sup_{t \in \left(t_{0},T\right)}\|f(t,\cdot)\|_{C_{\mathcal{X}}^{k-1+\alpha}\left(\mathbb{T}_{\mathbb{G}}\right)}+\|\Upsilon\|_{L^1\left([t_0,T]\times\mathbb{T}_{\mathbb{G}};\mathbb{R}^{n_1}\right)}
+\left\|g\right\|_{C_{\mathcal{X}}^{k+\alpha}\left(\mathbb{T}_{\mathbb{G}}\right)}+\left\|\rho_{0}\right\|_{C_{\mathcal{X}}^{-(k-1+\alpha)}(\mathbb{T}_{\mathbb{G}})}.
\end{align*}

Moreover, the solution is stable: for any $i\in\mathbb{N}$, let $\left(z^i,\rho^i\right)$ be the solution to system \eqref{general linear MFG system} with $b=b^i$, $f=f^i$, $\Upsilon=\Upsilon^i$ and $g=g^i$. If $b^{i} \to b$ in $C_{\mathcal{X}}^{\frac{\alpha}{2},k-2+\alpha}\left([t_0, T] \times \mathbb{T}_{\mathbb{G}};\mathbb{R}^{n_1}\right)\cap B\left([t_0,T];C_{\mathcal{X}}^{k-1+\alpha}\left(\mathbb{T}_{\mathbb{G}};\mathbb{R}^{n_1}\right)\right)$, $f^{i} \to f$ in $B\left([t_0,T];C_{\mathcal{X}}^{k-1+\alpha}\left(\mathbb{T}_{\mathbb{G}}\right)\right)$, $\Upsilon^{i} \to \Upsilon$ in $L^1\left([t_0,T]\times\mathbb{T}_{\mathbb{G}};\mathbb{R}^{n_1}\right)$ and $g^{i} \to g$ in $C_{\mathcal{X}}^{k+\alpha}\left(\mathbb{T}_{\mathbb{G}}\right)$ as $i\to +\infty$, then $\left(z^i,\rho^i\right) \to (z,\rho)$ in $C_{\mathcal{X}}^{1,2}([t_0,T] \times \mathbb{T}_{\mathbb{G}})\times C\left([t_0,T];C_{\mathcal{X}}^{-(k-1+\alpha)}(\mathbb{T}_{\mathbb{G}})\right)$ as $i\to +\infty$.
\end{Lem}
\begin{proof}
Assume without loss of generality that $t_{0}=0$. We use the Leray-Schauder fixed point theorem to prove the existence, where the crucial point is to show the estimates \eqref{linear system_z regularity} and \eqref{linear system_rho regularity}.

\emph{Step 1: Definition of the map $\Psi$.} Set $X:=C\left([t_0,T];C_{\mathcal{X}}^{-(k-1+\alpha)}(\mathbb{T}_{\mathbb{G}})\right)$. For any $\rho \in X$, we define $\Psi(\rho)$ in the following way.

First, it can be known from assumptions \ref{assum1} and \ref{assum2} that $\frac{\delta F}{\delta m}(x,m(t))(\rho(t)) \in C\left([t_0,T]\times\mathbb{T}_{\mathbb{G}}\right)\cap B\left([t_0,T];C_{\mathcal{X}}^{k-1+\alpha}\left(\mathbb{T}_{\mathbb{G}}\right)\right)$ and $\frac{\delta G}{\delta m}(x, m(T))(\rho(T)) \in C_{\mathcal{X}}^{k+\alpha}\left(\mathbb{T}_{\mathbb{G}}\right)$. By Lemma \ref{Lem_LHJB well-posedness}, there exists a unique solution $z \in C_{\mathcal{X}}^{1,2} \left([t_0, T] \times \mathbb{T}_{\mathbb{G}}\right)$ to the equation
\begin{equation}\label{linear system_z}
\begin{cases}
-\partial_{t} z-\Delta_{\mathcal{X}}z+b\cdot D_{\mathcal{X}}z=\frac{\delta F}{\delta m}(x,m(t))(\rho(t))+f, & \text { in }[t_0,T] \times \mathbb{T}_{\mathbb{G}}, \\
z(T)=\frac{\delta G}{\delta m}(x, m(T))(\rho(T))+g, & \text { in } \mathbb{T}_{\mathbb{G}},
\end{cases}
\end{equation}
and it follows from Lemma \ref{Lem_LHJB wellposed regularity} that
\begin{align}\label{z_regularity 1}
& \sup _{t \in[t_0, T]}\|z(t, \cdot)\|_{C_{\mathcal{X}}^{k+\alpha}\left(\mathbb{T}_{\mathbb{G}}\right)} \notag\\
\leq & C\left(\left\|\frac{\delta G}{\delta m}(x, m(T))(\rho(T))\right\|_{C_{\mathcal{X}}^{k+\alpha}\left(\mathbb{T}_{\mathbb{G}}\right)}+\left\|g\right\|_{C_{\mathcal{X}}^{k+\alpha}\left(\mathbb{T}_{\mathbb{G}}\right)}\right. \notag\\
& \left.+\sup_{t \in(t_0, T)}\left\|\frac{\delta F}{\delta m}(\cdot,m(t))(\rho(t))\right\|_{C_{\mathcal{X}}^{k-1+\alpha}\left(\mathbb{T}_{\mathbb{G}}\right)}+\sup_{t \in(t_0,T)}\|f(t,\cdot)\|_{C_{\mathcal{X}}^{k-1+\alpha}\left(\mathbb{T}_{\mathbb{G}}\right)}\right) \notag\\
\leq & C\left(\left\|g\right\|_{C_{\mathcal{X}}^{k+\alpha}\left(\mathbb{T}_{\mathbb{G}}\right)} +\sup_{t \in(t_0,T]}\|\rho(t)\|_{C_{\mathcal{X}}^{-(k-1+\alpha)}(\mathbb{T}_{\mathbb{G}})} +\sup_{t \in(t_0,T)}\|f(t,\cdot)\|_{C_{\mathcal{X}}^{k-1+\alpha}\left(\mathbb{T}_{\mathbb{G}}\right)}\right) \notag\\
\leq & C\left(M + \sup _{t \in(t_0,T]}\|\rho(t)\|_{C_{\mathcal{X}}^{-(k-1+\alpha)}(\mathbb{T}_{\mathbb{G}})}\right),
\end{align}
\begin{align}\label{z_regularity 2}
& \sup _{\substack{t \neq t'\\t,t' \in [t_0,T-\epsilon]}} \frac{\left\|z\left(t^{\prime}, \cdot\right)-z(t, \cdot)\right\|_{C_{\mathcal{X}}^{k+\alpha}\left(\mathbb{T}_{\mathbb{G}}\right)}}{\left|t^{\prime}-t\right|^{\frac{1}{2}}} \notag\\
\leq & C\left(\epsilon^{-\frac{1}{2}}\bigg(M+\|\rho(T)\|_{C_{\mathcal{X}}^{-(k-1+\alpha)}(\mathbb{T}_{\mathbb{G}})}\right) +\sup_{t \in(t_0,T)}\|\rho(t)\|_{C_{\mathcal{X}}^{-(k-1+\alpha)}(\mathbb{T}_{\mathbb{G}})}\bigg)
\end{align}
for any $\epsilon\in(0,T-t_0)$, where the constant $C>0$ depends on $\mathbb{G}$, $\alpha$, $k$, $T$, $c_F$, $c_G$ and $\sup_{t\in(t_0,T)}\|b(t,\cdot)\|_{C_{\mathcal{X}}^{k-1+\alpha}\left(\mathbb{T}_{\mathbb{G}};\mathbb{R}^{n_1}\right)}$ only.

Next we define $\Psi(\rho):=\tilde{\rho}$ as the unique weak solution to the FPK equation
\begin{equation}\label{linear system_tilde rho}
\begin{cases}
\partial_t\tilde{\rho}-\Delta_\mathcal{X}\tilde{\rho}-\operatorname{div}_{\mathcal{X}}(\tilde{\rho} b)=\operatorname{div}_{\mathcal{X}}(m P D_{\mathcal{X}}z+\Upsilon), & \text{ in } [t_0,T] \times \mathbb{T}_{\mathbb{G}},\\
\tilde{\rho}(t_0)=\rho_{0}, & \text{ in } \mathbb{T}_{\mathbb{G}}.
\end{cases}
\end{equation}
It follows from Lemma \ref{Lem_general FPK wellposed regularity} that $\tilde{\rho} \in X$. We shall prove that the map $\Psi$ is compact and continuous.

As for the compactness, let $\{\rho^{i}\}_{i=1}^{+\infty} \subset X$ be a sequence with
$$
\sup_{t \in [t_0,T]}\left\|\rho^{i}\right\|_{C_{\mathcal{X}}^{-(k-1+\alpha)}(\mathbb{T}_{\mathbb{G}})} \leq C<+\infty
$$
for a certain constant $C>0$. For each $i$, we consider the corresponding solutions $z^{i}$ and $\tilde{\rho}^{i}$ to the equations \eqref{linear system_z} and \eqref{linear system_tilde rho} respectively. It is known from \eqref{z_regularity 1} and \eqref{z_regularity 2} that for any $\epsilon\in(0,T-t_0)$, $\left\|z^{i}\right\|_{B\left([t_0,T];C_{\mathcal{X}}^{k+\alpha}(\mathbb{T}_{\mathbb{G}})\right)}$ and $\left\|z^{i}\right\|_{C^{\frac{1}{2}}\left([t_0,T-\epsilon];C_{\mathcal{X}}^{k+\alpha}(\mathbb{T}_{\mathbb{G}})\right)}$ are uniformly bounded. We use the Arzel\`{a}-Ascoli theorem to obtain that there exists a function $w$ such that $z^{i} \to w$ up to a subsequence uniformly at least in $C\left([t_0,T-\epsilon];C_{\mathcal{X}}^{k}(\mathbb{T}_{\mathbb{G}})\right)$, and for any $t \in (T-\epsilon,T]$ and multi-index $I$ with $|I|\leq k$, there exists a subsequence $i_{j}(t)\to+\infty$, such that $X_Iz^{i_{j}(t)}(t,\cdot)\to X_Iw(t,\cdot)$ uniformly in $C(\mathbb{T}_{\mathbb{G}})$. Then, 
we can obtain
\begin{align*}
& \left\|\operatorname{div}_{\mathcal{X}}\left(m P \left(D_{\mathcal{X}}z^{i_{j}(t)}-D_{\mathcal{X}}w\right)\right)\right\|_{L^1\left([t_0,T];C_{\mathcal{X}}^{-k}\left([0,1)^n\right)\cap C_{\mathcal{X}}^{-(k-1+\alpha)}(\mathbb{T}_{\mathbb{G}})\right)} \\
\leq & \int_{t_0}^{T-\epsilon}\int_{\mathbb{T}_{\mathbb{G}}}\left|P\left(D_{\mathcal{X}}z^{i_{j}(t)}-D_{\mathcal{X}}w\right)\right|dm(t)(x)dt +\int_{T-\epsilon}^{T}\int_{\mathbb{T}_{\mathbb{G}}}\left|P\left(D_{\mathcal{X}}z^{i_{j}(t)}-D_{\mathcal{X}}w\right)\right|dm(t)(x)dt \\
\leq & C\left(\left\|D_{\mathcal{X}}z^{i_{j}(t)}-D_{\mathcal{X}}w\right\|_ {L^\infty\left([t_0,T-\epsilon]\times\mathbb{T}_{\mathbb{G}}\right)}+\epsilon\right),
\end{align*}
where $C>0$ is independent of $j$ and $\epsilon$. Let $j \to +\infty$ and $\epsilon\to 0$, we have
$$
\operatorname{div}_{\mathcal{X}}\left(mP D_{\mathcal{X}}z^{i_{j}(T)}+\Upsilon\right) \to \operatorname{div}_{\mathcal{X}}\left(mP D_{\mathcal{X}}w+\Upsilon\right)
$$
in $L^1\left([t_0,T];C_{\mathcal{X}}^{-(k-1+\alpha)}(\mathbb{T}_{\mathbb{G}})\right)$. Thus, the stability result proved in Lemma \ref{Lem_general FPK wellposed regularity} shows that $\tilde{\rho}^{i_{j}(T)} \to \tilde{\varrho}$ in $X$ as $j\to+\infty$, where $\tilde{\varrho}$ is the solution to equation \eqref{linear system_tilde rho} related to $w$. This proves that $\Psi$ is compact on $X$.

The continuity of $\Psi$ can be obtained by using the similar method as the proof of the stability in Proposition \ref{Prop_MFG system wellposed regularity}, where we only need to note that for any $\rho^{i}\to\rho$ in $X$ as $i\to+\infty$ and $t\in[t_0,T]$,
\begin{align*}
& \left\|\frac{\delta G}{\delta m}(x, m(t))(\rho^{i}(t))-\frac{\delta G}{\delta m}(x, m(t))(\rho^{j}(t))\right\|_{C_{\mathcal{X}}^{k+\alpha}\left(\mathbb{T}_{\mathbb{G}}\right)} \notag\\
& +\sup_{\tau\in(t_0,t)}\left\|\frac{\delta F}{\delta m}(x,m(\tau))(\rho^{i}(\tau))-\frac{\delta F}{\delta m}(x,m(\tau))(\rho^{j}(\tau))\right\|_{C_{\mathcal{X}}^{k-1+\alpha}\left(\mathbb{T}_{\mathbb{G}}\right)} \\
\leq & \sup_{m \in \mathcal{P}(\mathbb{T}_{\mathbb{G}})}\left\|\frac{\delta G}{\delta m}(\cdot,m,\cdot)\right\|_{C_{\mathcal{X},y}^{k-1+\alpha}\left(\mathbb{T}_{\mathbb{G}};C_{\mathcal{X},x}^{k+\alpha}\left(\mathbb{T}_{\mathbb{G}}\right)\right)} \left\|\rho^{i}(t)-\rho^{j}(t)\right\|_{C_{\mathcal{X}}^{-(k-1+\alpha)}(\mathbb{T}_{\mathbb{G}})} \\
& +\sup_{m \in \mathcal{P}(\mathbb{T}_{\mathbb{G}})}\left\|\frac{\delta F}{\delta m}(\cdot,m,\cdot)\right\|_{C_{\mathcal{X},y}^{k-1+\alpha}\left(\mathbb{T}_{\mathbb{G}};C_{\mathcal{X},x}^{k-1+\alpha}\left(\mathbb{T}_{\mathbb{G}}\right)\right)} \sup_{\tau\in(t_0,t)}\left\|\rho^{i}(\tau)-\rho^{j}(\tau)\right\|_{C_{\mathcal{X}}^{-(k-1+\alpha)}(\mathbb{T}_{\mathbb{G}})}\to0
\end{align*}
as $i,j\to+\infty$.

It remains to check that the Leray-Schauder fixed point theorem holds true. Fix any $(\rho,\sigma) \in X \times [0,1]$ such that $\rho=\sigma \Psi(\rho)$ and let $z$ be the solution to the equation \eqref{linear system_z}. It can be implied that $(z,\rho)$ satisfies
\begin{equation}\label{linear system_LST}
\begin{cases}
-\partial_t z-\Delta_{\mathcal{X}} z+b \cdot D_{\mathcal{X}} z=\frac{\delta F}{\delta m}(x,m(t))(\rho(t))+f,& \text {in }[t_0, T] \times \mathbb{T}_{\mathbb{G}}, \\
\partial_t\rho-\Delta_\mathcal{X}\rho-\operatorname{div}_{\mathcal{X}}(\rho b)=\sigma\operatorname{div}_{\mathcal{X}}(m P D_{\mathcal{X}}z+\Upsilon),& \text {in }[t_0, T] \times \mathbb{T}_{\mathbb{G}}, \\
z(T, x)=\frac{\delta G}{\delta m}(x,m(T))(\rho(T))+g,\quad \rho(t_0)=\sigma\rho_0,& \text {in }\mathbb{T}_{\mathbb{G}}.
\end{cases}
\end{equation}
We only need to show that $\rho$ satisfies \eqref{linear system_rho regularity}, and this will be proved in the next step.

\emph{Step 2: Estimate of $\rho$.} According to \eqref{general FPK_weak formulation}, we have
\begin{align*}
& \left\langle\rho(T),\frac{\delta G}{\delta m}(\cdot, m(T))(\rho(T))+g(\cdot)\right\rangle +\int_{t_0}^{T} \left\langle \rho(s),\frac{\delta F}{\delta m}(\cdot,m(s))(\rho(s))+f(s,\cdot) \right\rangle ds \\
= & \left\langle\sigma\rho_0,z(t_0,\cdot)\right\rangle -\int_{t_0}^{T}\int_{\mathbb{T}_{\mathbb{G}}}\sigma \left(P D_{\mathcal{X}}z\right)(s,x)\cdot D_{\mathcal{X}}z(s,x)dm(s)(x) ds \\
& -\int_{t_0}^{T}\int_{\mathbb{T}_{\mathbb{G}}}\sigma\Upsilon(s,x)\cdot D_{\mathcal{X}}z(s,x)dxds
\end{align*}
It follows from the properties of $F$ and $G$ in \eqref{property_F} and \eqref{property_G} that
\begin{align}\label{energy estimate}
& \int_{t_0}^{T}\int_{\mathbb{T}_{\mathbb{G}}}\sigma \left(P D_{\mathcal{X}}z\right)(s,x)\cdot D_{\mathcal{X}}z(s,x)dm(s)(x) ds \notag\\
\leq & \sup_{t \in (t_0,T]}\left\|\rho(t)\right\|_{C_{\mathcal{X}}^{-(k-1+\alpha)}(\mathbb{T}_{\mathbb{G}})} \left(\left\|g\right\|_{C_{\mathcal{X}}^{k-1+\alpha}(\mathbb{T}_{\mathbb{G}})} +T\sup_{t \in (t_0,T)}\left\|f(t,\cdot)\right\|_{C_{\mathcal{X}}^{k-1+\alpha}(\mathbb{T}_{\mathbb{G}})}\right) \notag\\
& + \sigma\sup_{t \in [t_0,T)}\left\|z(t,\cdot)\right\|_{C_{\mathcal{X}}^{k-1+\alpha}(\mathbb{T}_{\mathbb{G}})} \left(\left\|\rho_0\right\|_{C_{\mathcal{X}}^{-(k-1+\alpha)}(\mathbb{T}_{\mathbb{G}})} +\|\Upsilon\|_{L^1\left([t_0,T]\times\mathbb{T}_{\mathbb{G}};\mathbb{R}^{n_1}\right)}\right) \notag\\
\leq & CM\left(\sup_{t \in (t_0,T]}\left\|\rho(t)\right\|_{C_{\mathcal{X}}^{-(k-1+\alpha)}(\mathbb{T}_{\mathbb{G}})} + \sigma\sup_{t \in [t_0,T)}\left\|z(t,\cdot)\right\|_{C_{\mathcal{X}}^{k-1+\alpha}(\mathbb{T}_{\mathbb{G}})}\right).
\end{align}
On the other hand, using Lemma \ref{Lem_general FPK wellposed regularity}
, we have
\begin{align}\label{rho_space regularity 1}
& \sup_{t \in [t_0,T]}\left\|\rho(t)\right\|_{C_{\mathcal{X}}^{-(k-1+\alpha)}(\mathbb{T}_{\mathbb{G}})} \notag\\
\leq & \sigma C\left(\left\|\rho_0\right\|_{C_{\mathcal{X}}^{-(k-1+\alpha)}(\mathbb{T}_{\mathbb{G}})} +\int_{t_0}^{T}\int_{\mathbb{T}_{\mathbb{G}}} \left|P D_{\mathcal{X}}z\right|dm(s)(x)ds +\|\Upsilon\|_{L^1\left([t_0,T]\times\mathbb{T}_{\mathbb{G}};\mathbb{R}^{n_1}\right)}\right),
\end{align}
where $C>0$ depends on $\mathbb{G}$, $\alpha$, $k$, $T$ and $\sup_{t\in(t_0,T)}\|b(t,\cdot)\|_{C_{\mathcal{X}}^{k-1+\alpha}\left(\mathbb{T}_{\mathbb{G}};\mathbb{R}^{n_1}\right)}$ only.
Since $P$ is a real symmetric matrix and satisfies \eqref{Assump_linear system}, we use H\"{o}lder's inequality and \eqref{energy estimate} to get
\begin{align*}
& \sigma\int_{t_0}^{T}\int_{\mathbb{T}_{\mathbb{G}}} \left|P D_{\mathcal{X}}z\right|dm(s)(x)ds \\
= & \sigma\int_{t_0}^{T}\int_{\mathbb{T}_{\mathbb{G}}} \left(D_{\mathcal{X}}z\cdot\left(P^2 D_{\mathcal{X}}z\right)\right)^{\frac{1}{2}}dm(s)(x)ds \\
\leq & \sigma\int_{t_0}^{T}\int_{\mathbb{T}_{\mathbb{G}}} \bar{c}\left|D_{\mathcal{X}}z\right|dm(s)(x)ds \\
\leq & \sigma\int_{t_0}^{T}\int_{\mathbb{T}_{\mathbb{G}}} \bar{c}^{\frac{3}{2}}\left(D_{\mathcal{X}}z\cdot\left(P D_{\mathcal{X}}z\right)\right)^{\frac{1}{2}}dm(s)(x)ds \\
\leq & \sigma\left(\bar{c}^3\int_{t_0}^{T}\int_{\mathbb{T}_{\mathbb{G}}} D_{\mathcal{X}}z\cdot\left(P D_{\mathcal{X}}z\right)dm(s)(x)ds\right)^\frac{1}{2} \left(\int_{t_0}^{T}\int_{\mathbb{T}_{\mathbb{G}}} dm(s)(x)ds\right)^\frac{1}{2} \\
\leq & CM^{\frac{1}{2}} \left(\sup_{t \in (t_0,T]}\left\|\rho(t)\right\|_{C_{\mathcal{X}}^{-(k-1+\alpha)}(\mathbb{T}_{\mathbb{G}})}^{\frac{1}{2}} + \sup_{t \in [t_0,T)}\left\|z(t,\cdot)\right\|_{C_{\mathcal{X}}^{k-1+\alpha}(\mathbb{T}_{\mathbb{G}})}^{\frac{1}{2}}\right).
\end{align*}
Putting the above inequality into \eqref{rho_space regularity 1}, one has
\begin{align*}
& \sup_{t \in [t_0,T]}\left\|\rho(t)\right\|_{C_{\mathcal{X}}^{-(k-1+\alpha)}(\mathbb{T}_{\mathbb{G}})} \\
\leq & C\left(M^{\frac{1}{2}} \left(\sup_{t \in (t_0,T]}\left\|\rho(t)\right\|_{C_{\mathcal{X}}^{-(k-1+\alpha)}(\mathbb{T}_{\mathbb{G}})}^{\frac{1}{2}} + \sup_{t \in [t_0,T)}\left\|z(t,\cdot)\right\|_{C_{\mathcal{X}}^{k-1+\alpha}(\mathbb{T}_{\mathbb{G}})}^{\frac{1}{2}}\right)+M\right).
\end{align*}
Using $\epsilon$-Cauchy inequality with $\varepsilon$ sufficiently small, we get
\begin{equation*}
\sup_{t \in [t_0,T]}\left\|\rho(t)\right\|_{C_{\mathcal{X}}^{-(k-1+\alpha)}(\mathbb{T}_{\mathbb{G}})} \leq C\left(M^\frac{1}{2} \sup_{t \in [t_0,T)}\left\|z(t,\cdot)\right\|_{C_{\mathcal{X}}^{k-1+\alpha}(\mathbb{T}_{\mathbb{G}})}^{\frac{1}{2}}+M\right).
\end{equation*}
Combining this with the estimate of $z$ in \eqref{z_regularity 1} and using $\epsilon$-Cauchy inequality again, we finally obtain
\begin{equation}\label{rho_space regularity 2}
\sup_{t \in [t_0,T]}\left\|\rho(t)\right\|_{C_{\mathcal{X}}^{-(k-1+\alpha)}(\mathbb{T}_{\mathbb{G}})} \leq CM,
\end{equation}
where $C>0$ depends on $\mathbb{G}$, $\alpha$, $k$, $T$, $\sup_{t\in(t_0,T)}\|b(t,\cdot)\|_{C_{\mathcal{X}}^{k-1+\alpha}\left(\mathbb{T}_{\mathbb{G}};\mathbb{R}^{n_1}\right)}$, $c_F$, $c_G$ and $\bar{c}$ only. This also yields the estimates for $z$, i.e. \eqref{linear system_z regularity}. Thus we have completed the proof of existence.

We now show the time regularity of $\rho$, which can also be obtained by duality. Fix any $t,t'\in[t_0,T]$ with $t<t'$, and choose any $\xi\in C_\mathcal{X}^{k+\alpha}(\mathbb{T}_{\mathbb{G}})$ with $\left\|\xi\right\|_{C_{\mathcal{X}}^{k+\alpha}\left(\mathbb{T}_{\mathbb{G}}\right)}\leq 1$. Let $z$ be the solution to the equation
\begin{equation*}
\begin{cases}
-\partial_tz-\Delta_\mathcal{X}z+b\cdot D_{\mathcal{X}}z=0, & \text{ in } [t_0,t') \times \mathbb{T}_{\mathbb{G}}, \\
z(t')=\xi, & \text{ in } \mathbb{T}_{\mathbb{G}}.
\end{cases}
\end{equation*}
Lemma \ref{Lem_LHJB wellposed regularity} states that
\begin{equation}\label{test fun. regularity}
\sup_{t \in[t_0, t']}\|z(t, \cdot)\|_{C_{\mathcal{X}}^{k+\alpha}\left(\mathbb{T}_{\mathbb{G}}\right)} +\sup_{\substack{t_1 \neq t_2\\t_1,t_2 \in [t_0,t']}} \frac{\left\|z\left(t_2, \cdot\right)-z(t_1, \cdot)\right\|_{C_{\mathcal{X}}^{k-1+\alpha}\left(\mathbb{T}_{\mathbb{G}}\right)}}{\left|t_2-t_1\right|^{\frac{1}{2}}} \leq C\left\|\xi\right\|_{C_{\mathcal{X}}^{k+\alpha}\left(\mathbb{T}_{\mathbb{G}}\right)},
\end{equation}
where $C>0$ depends on $\mathbb{G}$, $\alpha$, $k$, $T$ and $\sup_{t\in(t_0,T)}\|b(t,\cdot)\|_{C_{\mathcal{X}}^{k-2+\alpha}\left(\mathbb{T}_{\mathbb{G}};\mathbb{R}^{n_1}\right)}$ only.
Choosing $z$ as a test function for the equation of $\rho$ in \eqref{general linear MFG system} to obtain the weak formulation, then using \eqref{Assump_linear system}, \eqref{test fun. regularity} and \eqref{rho_space regularity 2}, we have
\begin{align*}
& \left\langle\rho(t')-\rho(t),\xi(\cdot)\right\rangle \\
= & \left\langle\rho(t),z(t,\cdot)-z(t',\cdot)\right\rangle -\int_{t}^{t'}\int_{\mathbb{T}_{\mathbb{G}}} D_{\mathcal{X}}z(s,x)\cdot\left(PD_{\mathcal{X}}z\right)(s,x) dm(s)(x)ds \\
& -\int_{t}^{t'}\int_{\mathbb{T}_{\mathbb{G}}}\Upsilon(s,x)\cdot D_{\mathcal{X}}z(s,x)dxds \\
\leq & \left\|z\left(t', \cdot\right)-z(t, \cdot)\right\|_{C_{\mathcal{X}}^{k-1+\alpha}\left(\mathbb{T}_{\mathbb{G}}\right)} \sup_{t \in [t_0,T]} \|\rho(t)\|_{C_{\mathcal{X}}^{-(k-1+\alpha)}(\mathbb{T}_{\mathbb{G}})} \\
& +(t'-t) \left\|D_{\mathcal{X}}z\right\|_{L^{\infty}\left((t,t')\times\mathbb{T}_{\mathbb{G}}\right)} \left(\bar{c}^{-1}\left\|D_{\mathcal{X}}z\right\|_{L^{\infty}\left((t,t')\times\mathbb{T}_{\mathbb{G}}\right)}+\|\Upsilon\|_{L^1\left([t,t']\times\mathbb{T}_{\mathbb{G}};\mathbb{R}^{n_1}\right)}\right) \\
\leq & CM(t'-t)^{\frac{1}{2}}\left\|\xi\right\|_{C_{\mathcal{X}}^{k+\alpha}\left(\mathbb{T}_{\mathbb{G}}\right)}.
\end{align*}
Hence,
$$
\sup _{\substack{t \neq t'\\t,t' \in [t_0,T]}} \frac{\left\|\rho\left(t^{\prime}\right)-\rho(t)\right\|_{C_{\mathcal{X}}^{-(k+\alpha)}(\mathbb{T}_{\mathbb{G}})}}{\left|t^{\prime}-t\right|^\frac{1}{2}} \leq CM,
$$
where $C>0$ depends on $\mathbb{G}$, $\alpha$, $k$, $T$, $\sup_{t\in(t_0,T)}\|b(t,\cdot)\|_{C_{\mathcal{X}}^{k-1+\alpha}\left(\mathbb{T}_{\mathbb{G}};\mathbb{R}^{n_1}\right)}$ and $\bar{c}$ only.

\emph{Step 3: Uniqueness.} Let $\left(z_{1}, \rho_{1}\right)$ and $\left(z_{2}, \rho_{2}\right)$ be two solutions to the system \eqref{general linear MFG system}. Then, the couple $(\bar{z},\bar{\rho}):=\left(z_{1}-z_{2}, \rho_{1}-\rho_{2}\right)$ satisfies the linear system as follows:
\begin{equation*}
\begin{cases}
-\partial_t \bar{z}-\Delta_{\mathcal{X}}\bar{z}+b\cdot D_{\mathcal{X}}\bar{z}=\frac{\delta F}{\delta m}(x,m(t))(\bar{\rho}(t)), & \text{in } [t_0,T]\times\mathbb{T}_{\mathbb{G}},\\
\partial_t\bar{\rho}-\Delta_{\mathcal{X}}\bar{\rho}- \operatorname{div}_{\mathcal{X}}\left(\bar{\rho}b\right)=\operatorname{div}_{\mathcal{X}}\left(mP D_{\mathcal{X}}\bar{z}\right), & \text{in } [t_0,T]\times\mathbb{T}_{\mathbb{G}}, \\
\bar{z}(T,x)=\frac{\delta G}{\delta m}(x,m(T))(\bar{\rho}(T)),\quad \bar{\rho}\left(t_{0}\right)=0, & \text{in } \mathbb{T}_{\mathbb{G}}.
\end{cases}
\end{equation*}
It can be obtained from \eqref{linear system_z regularity} and \eqref{linear system_rho regularity} that
$$
\sup_{t \in \left[t_0,T\right]}\|\bar{z}(t,\cdot)\|_{C_{\mathcal{X}}^{k+\alpha}\left(\mathbb{T}_{\mathbb{G}}\right)} + \sup_{t \in \left[t_0,T\right]}\|\bar{\rho}(t)\|_{C_{\mathcal{X}}^{-(k-1+\alpha)}(\mathbb{T}_{\mathbb{G}})} \leq 0,
$$
hence $\bar{z}=0$ and $\bar{\rho}=0$. This concludes the uniqueness.

\emph{Step 4: Stability.} Similar as before, the couple $(\bar{z}^i,\bar{\rho}^i):=\left(z^{i}-z,\rho^{i}-\rho\right)$ satisfies the linear system as follows:
\begin{equation}\label{general linear MFG system_stable}
\begin{cases}
-\partial_t \bar{z}^i-\Delta_{\mathcal{X}}\bar{z}^i+b^i\cdot D_{\mathcal{X}}\bar{z}^i=\frac{\delta F}{\delta m}(x,m(t))(\bar{\rho}^i(t))+\bar{f}^i-\bar{b}^i\cdot D_{\mathcal{X}}z, & \text{in } [t_0,T]\times\mathbb{T}_{\mathbb{G}},\\
\partial_t\bar{\rho}^i-\Delta_{\mathcal{X}}\bar{\rho}^i- \operatorname{div}_{\mathcal{X}}\left(\bar{\rho}^ib^i\right)=\operatorname{div}_{\mathcal{X}}\left(mP D_{\mathcal{X}}\bar{z}^i+\bar{\Upsilon}^i+\rho\bar{b}^i\right), & \text{in } [t_0,T]\times\mathbb{T}_{\mathbb{G}}, \\
\bar{z}^i(T,x)=\frac{\delta G}{\delta m}(x,m(T))(\bar{\rho}^i(T))+\bar{g}^i,\quad \bar{\rho}^i\left(t_{0}\right)=0, & \text{in } \mathbb{T}_{\mathbb{G}},
\end{cases}
\end{equation}
where $\bar{b}^i:=b^i-b$, $\bar{f}^i:=f^i-f$, $\bar{\Upsilon}^i:=\Upsilon^i-\Upsilon$ and $\bar{g}^i:=g^i-g$. Following the same way as \eqref{linear system_z regularity}, we can obtain
\begin{align*}
\sup_{t \in \left[t_0,T\right]}\|\bar{z}^i(t,\cdot)\|_{C_{\mathcal{X}}^{k+\alpha}\left(\mathbb{T}_{\mathbb{G}}\right)} \leq & \bigg(\sup_{t \in \left(t_{0},T\right)}\|\left(\bar{f}^i-\bar{b}^i\cdot D_{\mathcal{X}}z\right)(t,\cdot)\|_{C_{\mathcal{X}}^{k-1+\alpha}\left(\mathbb{T}_{\mathbb{G}}\right)}+\|\bar{\Upsilon}^i\|_{L^1\left([t_0,T]\times\mathbb{T}_{\mathbb{G}};\mathbb{R}^{n_1}\right)} \\
& +\left\|\bar{g}^i\right\|_{C_{\mathcal{X}}^{k+\alpha}\left(\mathbb{T}_{\mathbb{G}}\right)}+\left\|\operatorname{div}_{\mathcal{X}}\left(\rho\bar{b}^i\right)\right\|_{L^1\left([t_0,T];C_{\mathcal{X}}^{-(k-1+\alpha)}(\mathbb{T}_{\mathbb{G}})\right)}\bigg) \to0
\end{align*}
as $i\to+\infty$, where we note that
\begin{equation*}
\left\|\operatorname{div}_{\mathcal{X}}(\rho\bar{b}^i)\right\|_{L^1\left([t_0,T];C_{\mathcal{X}}^{-(k-1+\alpha)}(\mathbb{T}_{\mathbb{G}})\right)}
\leq C\left\|\bar{b}^i\right\|_{C_{\mathcal{X}}^{0,k-2}\left([t_0,T]\times\mathbb{T}_{\mathbb{G}}\right)}
\end{equation*}
by referring to \cite[(5.5)]{JWY}. Further, based on the stability of equation \eqref{general FPK} (see Lemma \ref{Lem_general FPK wellposed regularity}), we have $\bar{\rho}^i\to 0$ in $C\left([t_0,T];C_{\mathcal{X}}^{-(k-1+\alpha)}(\mathbb{T}_{\mathbb{G}})\right)$ as $i\to +\infty$. Letting $i\to+\infty$ in \eqref{general linear MFG system_stable}, the stability can eventually be achieved.
\end{proof}
\begin{Rem}
Based on the linearity of \eqref{general linear MFG system}, the regularity results \eqref{linear system_z regularity} and \eqref{linear system_rho regularity}, and the uniqueness of the solution, we can conclude that the map $\rho_{0} \mapsto (z,\rho)$ is linear and continuous from $C_{\mathcal{X}}^{-k}\left([0,1)^n\right)\cap C_{\mathcal{X}}^{-(k-1+\alpha)}(\mathbb{T}_{\mathbb{G}})$ into $C\left(\left[t_{0},T\right];C_{\mathcal{X}}^{k+\alpha}(\mathbb{T}_{\mathbb{G}}) \times C_{\mathcal{X}}^{-(k-1+\alpha)}(\mathbb{T}_{\mathbb{G}})\right)$.
\end{Rem}
Applying Lemma \ref{Lem_GL system regularity} to the linearized MFG system \eqref{linear MFG system} gives the following important conclusion.
\begin{Cor}\label{Lem_LMFG system regularity}
Let assumptions \ref{assum1}-\ref{assum3} and the conclusions of Proposition \ref{Prop_MFG system wellposed regularity} hold, and $(t_0,m_{0})\in[0,T]\times\mathcal{P}\left(\mathbb{T}_{\mathbb{G}}\right)$. If $\rho_{0} \in C_{\mathcal{X}}^{-(k-1)}\left([0,1)^n\right)\cap C_{\mathcal{X}}^{-(k-1+\alpha)}(\mathbb{T}_{\mathbb{G}})$, then there exists a unique solution $(z,\rho)$ to the system \eqref{linear MFG system} in the sense of Definition \ref{Def_linear system sol.}, and satisfies
\begin{equation}\label{linear MFG system_z regularity}
\begin{gathered}
\sup_{t \in[t_0, T]}\|z(t, \cdot)\|_{C_{\mathcal{X}}^{k+\alpha}\left(\mathbb{T}_{\mathbb{G}}\right)} \leq C\left\|\rho_{0}\right\|_{C_{\mathcal{X}}^{-(k-1+\alpha)}(\mathbb{T}_{\mathbb{G}})}, \\
\sup _{\substack{t \neq t'\\t,t' \in [t_0,T-\epsilon]}} \frac{\left\|z\left(t^{\prime}, \cdot\right)-z(t, \cdot)\right\|_{C_{\mathcal{X}}^{k+\alpha}\left(\mathbb{T}_{\mathbb{G}}\right)}}{\left|t^{\prime}-t\right|^{\frac{1}{2}}} \leq \epsilon^{-\frac{1}{2}}C\left\|\rho_{0}\right\|_{C_{\mathcal{X}}^{-(k-1+\alpha)}(\mathbb{T}_{\mathbb{G}})}
\end{gathered}
\end{equation}
for any $\epsilon\in(0,T-t_0)$, and
\begin{equation}\label{linear MFG system_rho regularity}
\sup_{t \in \left[t_{0},T\right]}\|\rho(t)\|_{C_\mathcal{X}^{-(k-1+\alpha)}(\mathbb{T}_{\mathbb{G}})} + \sup _{\substack{t \neq t'\\t,t' \in [t_0,T]}} \frac{\left\|\rho\left(t^{\prime}\right)-\rho(t)\right\|_{C_\mathcal{X}^{-(k+\alpha)}(\mathbb{T}_{\mathbb{G}})}}{\left|t^{\prime}-t\right|^{\frac{1}{2}}} \leq C\left\|\rho_{0}\right\|_{C_{\mathcal{X}}^{-(k-1+\alpha)}(\mathbb{T}_{\mathbb{G}})},
\end{equation}
where the constants $C>0$ depend on $\mathbb{G}$, $\alpha$, $k$, $T$, $\sup_{t\in(t_0,T)}\|u(t,\cdot)\|_{C_{\mathcal{X}}^{k+\alpha}\left(\mathbb{T}_{\mathbb{G}}\right)}$, $c_F$, $c_G$ and $H$ only.
\end{Cor}
\begin{proof}
This is a direct application of Lemma \ref{Lem_GL system regularity} with $b(t,x)=D_pH(x,D_{\mathcal{X}}u(t,x))$, $P(t,x)=D_{pp}^2 H(x,D_{\mathcal{X}}u(t,x))$ and $f=\Upsilon=g=0$ in the system \eqref{general linear MFG system}. Here, we note that according to Proposition \ref{Prop_MFG system wellposed regularity}, $b$ belongs to $C_{\mathcal{X}}^{\frac{\alpha}{2},k-1+\alpha}\left([t_0,T] \times \mathbb{T}_{\mathbb{G}}\right)$.
\end{proof}
Using the above lemma we can then prove Lemma \ref{Lem_relation of z and rho_0}.
\begin{proof}[\bf{Proof of Lemma \ref{Lem_relation of z and rho_0}}]
Denote $z(t,x;\rho_0)$ as the solution to the first component of the system \eqref{linear MFG system} with $m(t_0)=m_0$ and $\rho(t_0)=\rho_0$. Setting $\rho_0=\delta_y$, which is the Dirac measure mass at $y\in\mathbb{T}_{\mathbb{G}}$, we can define
$$
K(t_0,x,m_0,y):=z(t_0,x;\delta_y),\,x,y\in\mathbb{T}_{\mathbb{G}}.
$$
Thanks to the linearity of the system and the density of simple measures, we easily get that the representation formula \eqref{linear MFG system_rep.formula} holds true. Moreover, for any $y \in \mathbb{R}^n$, let $\gamma(\tau)$ be the absolutely continuous integral curve of the vector field $X_i,i\in\{1,\ldots,n_1\}$ such that
$$
\begin{cases}
\gamma^{\prime}(\tau)=X_i(\gamma(\tau)), \\
\gamma(0)=y.
\end{cases}
$$
Then
\begin{equation*}
\frac{K(t_0,x,m_0,\gamma(\tau))-K(t_0,x,m_0,y)}{\tau}=z(t_0,x;\frac{\delta_{\gamma(\tau)}-\delta_y}{\tau}).
\end{equation*}
Denoting $\Delta_{\tau}^y f:=\frac{f(\gamma(\tau))-f(y)}{\tau}$ for any map $f$, we need to prove the limit of the above equality exists as $\tau \to 0$. From \eqref{linear MFG system_z regularity} and the Lagrange mean value theorem, we have that for any $\tau_1,\tau_2 >0$,
\begin{align*}
& \left\|\Delta_{\tau_1}^y K(t_0,\cdot,m_0,y)-\Delta_{\tau_2}^y K(t_0,\cdot,m_0,y)\right\|_{C_{\mathcal{X}}^{k+\alpha}\left(\mathbb{T}_{\mathbb{G}}\right)}\\
\leq & C\left\|\Delta_{\tau_1}^y \delta_{(\cdot)}-\Delta_{\tau_2}^y \delta_{(\cdot)}\right\|_{C_{\mathcal{X}}^{-(k-1+\alpha)}(\mathbb{T}_{\mathbb{G}})} \\
= & C\sup_{\|\xi\|_{C_{\mathcal{X}}^{k-1+\alpha}(\mathbb{T}_{\mathbb{G}})}\leq 1}\left(\Delta_{\tau_1}^y \xi-\Delta_{\tau_2}^y \xi\right) \\
= & C\sup_{\|\xi\|_{C_{\mathcal{X}}^{k-1+\alpha}(\mathbb{T}_{\mathbb{G}})}\leq 1}\left(\frac{d\xi\left(\gamma(\tau)\right)}{d\tau}\bigg|_{\tau=\lambda_1 \tau_1}-\frac{d\xi\left(\gamma(\tau)\right)}{d\tau}\bigg|_{\tau=\lambda_2 \tau_2}\right) \\
= & C\sup_{\|\xi\|_{C_{\mathcal{X}}^{k-1+\alpha}(\mathbb{T}_{\mathbb{G}})}\leq 1}\left(X_i\xi\left(\gamma(\lambda_1\tau_1)\right)-X_i\xi\left(\gamma(\lambda_2\tau_2)\right)\right)\\
\leq & Cd_{cc}\left(\gamma(\lambda_1\tau_1),\gamma(\lambda_2\tau_2)\right)^{\alpha} \leq C\left|\lambda_1\tau_1-\lambda_2\tau_2\right|^{\alpha},
\end{align*}
for some $\lambda_j \in (0,1), j=1,2$. Letting $\tau_1,\tau_2 \to 0$, we get that $\left\{\Delta_{\tau}^y K\right\}_{\tau>0}$ is Cauchy and the limit exists, that is
$$
X_i^y K(t_0,x,m_0,y)=z(t_0,x;X_i^* \delta_y).
$$

Replacing $K$ with $X_i K$ and repeating the above steps, we then obtain that for any multi-index $I=(i_1,\ldots,i_{|I|})$ with $|I|\leq k-1$,
$$
X_I^y K(t_0,x,m_0,y)=z(t_0,x;X_{i_1}^*\cdots X_{i_{|I|}}^* \delta_y).
$$
Moreover, fix any $y,y' \in \mathbb{T}_{\mathbb{G}}$ with $y\neq y'$, using \eqref{linear MFG system_z regularity} again, we get
\begin{align*}
\left\|X_I^y K(t_0,\cdot,m_0,y)-X_I^y K(t_0,\cdot,m_0,y')\right\|_{C_{\mathcal{X}}^{k+\alpha}\left(\mathbb{T}_{\mathbb{G}}\right)} \leq
& C\left\|\left(X_I\right)^*\delta_y-\left(X_I\right)^*\delta_{y'}\right\|_{C_{\mathcal{X}}^{-(k-1+\alpha)}(\mathbb{T}_{\mathbb{G}})} \\
\leq & C\left\|\delta_y-\delta_{y'}\right\|_{C_{\mathcal{X}}^{-\alpha}(\mathbb{T}_{\mathbb{G}})} \\
\leq & Cd_{cc}^{\mathbb{T}_{\mathbb{G}}}(y,y')^{\alpha},
\end{align*}
where $C>0$ depends on $\mathbb{G}$, $\alpha$, $k$, $T$, $\sup_{t\in(t_0,T)}\|u(t,\cdot)\|_{C_{\mathcal{X}}^{k+\alpha}\left(\mathbb{T}_{\mathbb{G}}\right)}$, $c_F$, $c_G$ and $H$ only. Consequently,
\begin{equation*}
\left\|K\left(t_{0}, \cdot, m_{0}, \cdot\right)\right\|_{C_{\mathcal{X},y}^{k-1+\alpha}\left(\mathbb{T}_{\mathbb{G}};C_{\mathcal{X},x}^{k+\alpha}\left(\mathbb{T}_{\mathbb{G}}\right)\right)} \leq C,
\end{equation*}
where $C>0$ does not depend on $(t_0,m_0)$. In addition, it follows from the stability of the solutions to the MFG system \eqref{MFG system} and the linearized system \eqref{general linear MFG system}, and the Lipschitz continuity of $\frac{\delta F}{\delta m}$ and $\frac{\delta G}{\delta m}$ with respect to the measure that, for any multi-index $J$ and $J'$ with $|J|\leq k$ and $|J'|\leq k-1$, $X_{J}^x X_{J'}^y K\left(t_{0},x,m_{0},y\right)$ is continuous in $(t_0,m_0)$ uniformly with respect to $(x,y)$.

Here ends the proof of the lemma.
\end{proof}

\subsection{Differentiability of \texorpdfstring{$U$}{U} in \texorpdfstring{$m_0$}{m0}}
Now we prove that the function $K$ in \eqref{linear MFG system_rep.formula} is in fact the derivative of $U$ in $m_0$, which is concluded in Proposition \ref{Prop_U C^1 w.r.t. m}.

\begin{proof}[\bf{Proof of Proposition \ref{Prop_U C^1 w.r.t. m}}]
Set $v:=\hat{u}-u-z$ and $\mu:=\hat{m}-m-\rho$, then the couple $(v,\mu)$ satisfies the following linear system:
\begin{equation}\label{linear system_(v,mu)}
\begin{cases}
-\partial_t v-\Delta_{\mathcal{X}} v+D_pH(x,D_{\mathcal{X}}u) \cdot D_{\mathcal{X}} v=\frac{\delta F}{\delta m}(x,m(t))(\mu(t))+f(t,x),& \text {in }[t_0, T] \times \mathbb{T}_{\mathbb{G}}, \\
\partial_t\mu-\Delta_\mathcal{X}\mu-\operatorname{div}_{\mathcal{X}}\left(\mu D_pH(x,D_{\mathcal{X}}u)\right) \\
=\operatorname{div}_{\mathcal{X}}\left(m D_{pp}^2H(x,D_{\mathcal{X}}u) D_{\mathcal{X}}v+\Upsilon\right),& \text {in }[t_0, T] \times \mathbb{T}_{\mathbb{G}}, \\
v(T, x)=\frac{\delta G}{\delta m}(x,m(T))(\mu(T))+g(x),\quad \mu(t_0)=0,& \text {in }\mathbb{T}_{\mathbb{G}},
\end{cases}
\end{equation}
where for any $t\in[t_0,T]$ and $x\in\mathbb{T}_{\mathbb{G}}$,
\begin{align*}
f(t, x):= & -\int_{0}^{1}\left(D_{p} H\left(x, s D_{\mathcal{X}} \hat{u}+(1-s) D_{\mathcal{X}} u\right)-D_{p} H\left(x, D_{\mathcal{X}} u\right)\right) \cdot D_{\mathcal{X}}(\hat{u}-u) d s \\
& + \int_{0}^{1} \int_{\mathbb{T}_{\mathbb{G}}}\left(\frac{\delta F}{\delta m}(x, s \hat{m}(t)+(1-s) m(t), y)-\frac{\delta F}{\delta m}(x, m(t), y)\right) d(\hat{m}(t)-m(t))(y) ds,
\end{align*}
\begin{align*}
\Upsilon(t,x):= & (\hat{m}-m)(t) D_{pp}^{2} H\left(x, D_{\mathcal{X}} u(t, x)\right)\left(D_{\mathcal{X}}\hat{u}-D_{\mathcal{X}} u\right)(t, x) \\
& + \hat{m}(t) \int_{0}^{1}\left(D_{pp}^{2} H\left(x, s D_{\mathcal{X}} \hat{u}(t, x)+(1-s) D_{\mathcal{X}} u(t, x)\right)\right. \\
& \qquad\qquad\left.-D_{pp}^{2} H\left(x, D_{\mathcal{X}} u(t, x)\right)\right)\left(D_{\mathcal{X}} \hat{u}-D_{\mathcal{X}} u\right)(t, x) d s
\end{align*}
and
\begin{equation*}
g(x):= \int_{0}^{1} \int_{\mathbb{T}_{\mathbb{G}}}\left(\frac{\delta G}{\delta m}(x, s \hat{m}(T)+(1-s) m(T), y) -\frac{\delta G}{\delta m}(x, m(T), y)\right) d(\hat{m}(T)-m(T))(y) ds.
\end{equation*}
From assumptions \ref{assum1}-\ref{assum3} and Proposition \ref{Prop_MFG system wellposed regularity}, it is easy to verify that the coefficients of system \eqref{linear system_(v,mu)} satisfy the conditions in Lemma \ref{Lem_GL system regularity}. In addition, using Proposition \ref{Prop_Lip. ctn. of U}, we have
\begin{align*}
& \sup_{t \in \left(t_{0},T\right)}\|f(t,\cdot)\|_{C_{\mathcal{X}}^{k-1+\alpha}\left(\mathbb{T}_{\mathbb{G}}\right)} \\
\leq & C\sup_{t\in\left(t_{0},T\right)}\left(\left\|\hat{u}(t,\cdot)-u(t,\cdot)\right\|_{C_{\mathcal{X}}^{k+\alpha}(\mathbb{T}_{\mathbb{G}})}^2 +\operatorname{Lip}\left(\frac{\delta F}{\delta m}\right)d_1\left(\hat{m}(t),m(t)\right)^2\right) \leq C d_1^2\left(\hat{m}_0,m_0\right),
\end{align*}
\begin{align*}
& \|\Upsilon\|_{L^1\left([t_0,T]\times\mathbb{T}_{\mathbb{G}};\mathbb{R}^{n_1}\right)} \\
\leq & C\sup_{t\in\left(t_{0},T\right)}\left(\left\|\left(\hat{u}-u\right)(t,\cdot)\right\|_{C_{\mathcal{X}}^{1+1}\left(\mathbb{T}_{\mathbb{G}}\right)} d_1\left(\hat{m}(t),m(t)\right) +\left\|\left(\hat{u}-u\right)(t,\cdot)\right\|_{C_{\mathcal{X}}^{1}\left(\mathbb{T}_{\mathbb{G}}\right)}^2\right) \\
\leq & C d_1^2\left(\hat{m}_0,m_0\right),
\end{align*}
and
\begin{align*}
\left\|g\right\|_{C_{\mathcal{X}}^{k+\alpha}\left(\mathbb{T}_{\mathbb{G}}\right)} \leq C\operatorname{Lip}\left(\frac{\delta G}{\delta m}\right)d_1\left(\hat{m}(T),m(T)\right)^2 \leq C d_1^2\left(\hat{m}_0,m_0\right),
\end{align*}
where $C>0$ depends on $\mathbb{G}$, $\alpha$, $k$, $T$, $c_{F}$, $c_{G}$ and $H$ only. Then, we apply Lemma \ref{Lem_GL system regularity} to obtain the estimate of the solution $(v,\mu)$, i.e. \eqref{(hat u-u-z,hat m-m-rho)_regularity} holds true.

Finally, for any $m_0, m_0^{\prime} \in \mathcal{P}\left(\mathbb{T}_{\mathbb{G}}\right)$, set $\hat{m}_0:=(1-s)m_0+sm_0^{\prime}$. Using Lemma \ref{Lem_relation of z and rho_0} and \eqref{(hat u-u-z,hat m-m-rho)_regularity}, we can get that there exists a $C_{\mathcal{X}}^{k+\alpha}(\mathbb{T}_{\mathbb{G}}) \times C_{\mathcal{X}}^{k-1+\alpha}(\mathbb{T}_{\mathbb{G}})$ map $\frac{\delta U}{\delta m}(t_0,\cdot,m_0,\cdot): \mathbb{T}_{\mathbb{G}} \times \mathbb{T}_{\mathbb{G}} \to \mathbb{R}$ such that
\begin{align*}
& \left\|\frac{U\left(t_0,\cdot,(1-s)m_0+sm_0^{\prime}\right)-U(t_0,\cdot,m_0)}{s} -\int_{\mathbb{T}_{\mathbb{G}}} \frac{\delta U}{\delta m}(t_0,\cdot,m_0,y) d\left(m_0^{\prime}-m_0\right)(y)\right\|_{C_{\mathcal{X}}^{k+\alpha}\left(\mathbb{T}_{\mathbb{G}}\right)} \\
= & \frac{1}{s}\left\|\hat{u}(t_0,\cdot)-u(t_0,\cdot)-z(t_0,\cdot)\right\|_{C_{\mathcal{X}}^{k+\alpha}\left(\mathbb{T}_{\mathbb{G}}\right)} \leq C s d_{1}^{2}\left(m_0^{\prime},m_0\right)\to0\text{ as }s\to0^+.
\end{align*}
Thus $U(t_0,x,m_0,y)$ is $C^1$ in $m_0$, and the regularity of the derivative $\frac{\delta U}{\delta m}(t_0,x,m_0,y)$ can be directly obtained from Lemma \ref{Lem_relation of z and rho_0} and \eqref{(hat u-u-z,hat m-m-rho)_regularity}.

Therefore, we have obtained that this proposition holds true.
\end{proof}

\section{Solvability of the master equation}\label{Sec_5}

We have by now obtained the desired properties of the solution to the degenerate MFG system \eqref{MFG system} and the $C^1$ differentiability of $U$ with respect to the measure. Accordingly, we can proceed to prove the main theorem.
\begin{proof}[\bf{Proof of Theorem \ref{Thm_ME wellposed regularity}}]

\emph{Step 1: Existence.}
Recall Proposition \ref{Prop_MFG system wellposed regularity}, let $(u,m)$ be the unique solution to the degenerate MFG system \eqref{MFG system} with the initial condition $(t_0,m_0) \in [0,T]\times\mathcal{P}\left(\mathbb{T}_{\mathbb{G}}\right)$, and set $U(t_0,x,m_0):=u(t_0,x)$ for any $x\in\mathbb{T}_{\mathbb{G}}$. We need to prove that $U$ is a solution to the master equation \eqref{ME} in the sense of Definition \ref{Def_ME sol.}. By referring to \cite[Corollary 3.1]{JWY}, we can take a smooth sequence $\{m_0^{i}\}_{i=1}^{+\infty}$ to approximate $m_0$. Denoting the corresponding solutions to the MFG system \eqref{MFG system} as $\left(u^{i},m^{i}\right)$, we then compute
\begin{align*}
\partial_{t_0} U\left(t_0,x,m_0^{i}\right) = & \lim_{\Delta t \to 0}\frac{U\left(t_0+\Delta t,x,m_0^{i}\right)-U\left(t_0,x,m_0^{i}\right)}{\Delta t} \\
= & \lim_{\Delta t \to 0}\frac{U\left(t_0+\Delta t,x,m_0^{i}\right)-U\left(t_0+\Delta t,x,m^{i}\left(t_0+\Delta t\right)\right)}{\Delta t} \\
& + \lim_{\Delta t \to 0}\frac{U\left(t_0+\Delta t,x,m^{i}\left(t_0+\Delta t\right)\right)-U\left(t_0,x,m_0^{i}\right)}{\Delta t} =: I+II.
\end{align*}
Set $m^{i,s}:=(1-s)m^{i}\left(t_0\right)+sm^{i}\left(t_0+\Delta t\right)$ for any $s\in(0,1)$. Since $U$ is $C^1$ differentiable with respect to the measure by Proposition \ref{Prop_U C^1 w.r.t. m} and $m^{i}\in C_{\mathcal{X}}^{\frac{2+\alpha}{2},2+\alpha}\left([t_0,T]\times\mathbb{T}_{\mathbb{G}}\right)$ by Proposition \ref{Prop_MFG system wellposed regularity}, using \eqref{delta U/delta m_regularity}, the continuity of $\frac{\delta U}{\delta m}$ and integration by parts, we have
\begin{align*}
I = & \lim_{\Delta t \to 0} \int_0^1 \int_{\mathbb{T}_{\mathbb{G}}} \frac{\delta U}{\delta m}\left(t_0+\Delta t, x, m^{i,s}, y\right)\left(\frac{m^{i}\left(t_0, y\right)-m^{i}\left(t_0+\Delta t, y\right)}{\Delta t}\right) d y d s \\
= & \int_{\mathbb{T}_{\mathbb{G}}} \bigg[\frac{\delta U}{\delta m}\left(t_0, x, m_0^{i}, y\right)\left(-\Delta_{\mathcal{X}} m^{i}(t_0,y)-\operatorname{div}_{\mathcal{X}}\left(m^{i}(t_0,y)D_pH(y, D_{\mathcal{X}}u^{i}(t_0,y))\right)\right)\bigg]d y \\
= & \int_{\mathbb{T}_{\mathbb{G}}} -\Delta_{\mathcal{X}}^y \frac{\delta U}{\delta m}\left(t_0, x, m_0^{i}, y\right) dm_0^{i}(y) \\
& +\int_{\mathbb{T}_{\mathbb{G}}} D_{\mathcal{X}}^y \frac{\delta U}{\delta m}\left(t_0, x, m_0^{i}, y\right) \cdot D_p H(y, D_{\mathcal{X}}U(t_0,y,m_0^{i})) dm_0^{i}(y).
\end{align*}
On the other hand,
\begin{align*}
II=& \lim_{\Delta t \to 0}\frac{u^{i}\left(t_0+\Delta t,x\right)-u^{i}\left(t_0,x\right)}{\Delta t}=\partial_{t_0} u^{i}(t_0,x) \\
= & -\Delta_{\mathcal{X}}u^{i}(t_0,x)+H\left(x, D_{\mathcal{X}} u^{i}(t_0,x)\right)-F(x, m_0^{i})\\
= & -\Delta_{\mathcal{X}}U(t_0,x,m_0^{i})+H\left(x, D_{\mathcal{X}} U(t_0,x,m_0^{i})\right)-F(x, m_0^{i}).
\end{align*}
Combining the above two equalities, we can then let $i \to +\infty$ thanks to the continuity of both sides of the equality, thus obtaining that $\partial_{t_0} U\left(t_0,x,m_0\right)$ exists, namely for any $(t_0,x,m_0) \in [0,T] \times \mathbb{T}_{\mathbb{G}} \times \mathcal{P}(\mathbb{T}_{\mathbb{G}})$,
\begin{align*}
\partial_{t_0} U\left(t_0,x,m_0\right)= & -\Delta_{\mathcal{X}}U(t_0,x,m_0)+H\left(x, D_{\mathcal{X}} U(t_0,x,m_0)\right)-F(x, m_0) \\
& -\int_{\mathbb{T}_{\mathbb{G}}} \Delta_{\mathcal{X}}^y \frac{\delta U}{\delta m}\left(t_0, x, m_0, y\right) dm_0(y) \\
& +\int_{\mathbb{T}_{\mathbb{G}}} D_{\mathcal{X}}^y \frac{\delta U}{\delta m}\left(t_0, x, m_0, y\right) \cdot D_p H(y, D_{\mathcal{X}}U(t_0,y,m_0)) dm_0(y).
\end{align*}
This concludes the existence part.

\emph{Step 2: Uniqueness.} Let $V$ be another solution to the master equation \eqref{ME} in the sense of Definition \ref{Def_ME sol.}. Fix any $(t_0,m_0)\in[0,T]\times\mathcal{P}\left(\mathbb{T}_{\mathbb{G}}\right)$, we employ the Schauder fixed point theorem to solve the FPK equation
\begin{equation}\label{ME_uniqueness_FPK 1}
\begin{cases}
\partial_t \tilde{m}-\Delta_{\mathcal{X}} \tilde{m}-\operatorname{div}_{\mathcal{X}}\left(\tilde{m}D_pH\left(x,
D_{\mathcal{X}}V(t,x,\tilde{m}(t))\right)\right)=0, & \text{in } [t_0,T] \times \mathbb{T}_{\mathbb{G}}, \\
\tilde{m}(t_0)=m_0, & \text{in } \mathbb{T}_{\mathbb{G}}.
\end{cases}
\end{equation}
Recall the space $E$ in \eqref{SFPT_E Def.}, fix any $\mu \in E$ and consider the following equation:
\begin{equation}\label{ME_uniqueness_FPK 2}
\begin{cases}
\partial_t m-\Delta_{\mathcal{X}} m-\operatorname{div}_{\mathcal{X}}\left(mD_pH\left(x,
D_{\mathcal{X}}V(t,x,\mu(t))\right)\right)=0, & \text{in } [t_0,T] \times \mathbb{T}_{\mathbb{G}}, \\
m(t_0)=m_0, & \text{in } \mathbb{T}_{\mathbb{G}}.
\end{cases}
\end{equation}
Since $\sup_{(t,m) \in [0,T] \times \mathcal{P}(\mathbb{T}_{\mathbb{G}})}\left\|\frac{\delta V}{\delta m}\left(t, \cdot, m, \cdot\right)\right\|_{C_{\mathcal{X},y}^{1}\left(\mathbb{T}_{\mathbb{G}};C_{\mathcal{X},x}^{1}\left(\mathbb{T}_{\mathbb{G}}\right)\right)}<+\infty$, it can be verified that for any $t\in[0,T]$, $x\in\mathbb{T}_{\mathbb{G}}$ and $m_1, m_2 \in \mathcal{P}\left(\mathbb{T}_{\mathbb{G}}\right)$,
\begin{align}\label{D_mathcal X V_Lip. in m}
& \left|D_{\mathcal{X}}\left(V\left(t,x,m_1\right)-V(t,x,m_2)\right)\right| \notag\\
= & \left|\int_0^1 \int_{\mathbb{T}_{\mathbb{G}}} D_{\mathcal{X}}^x\frac{\delta V}{\delta m}\left(t,x,(1-s) m_2+s m_1, y\right) d\left(m_1-m_2\right)(y)ds\right| \notag\\
\leq & \sup_{(t,m) \in [0,T] \times \mathcal{P}(\mathbb{T}_{\mathbb{G}})}\left\|\frac{\delta V}{\delta m}\left(t, \cdot, m, \cdot\right)\right\|_{C_{\mathcal{X},y}^{1}\left(\mathbb{T}_{\mathbb{G}};C_{\mathcal{X},x}^{1}\left(\mathbb{T}_{\mathbb{G}}\right)\right)}d_1(m_1,m_2),
\end{align}
hence $D_pH\left(x,
D_{\mathcal{X}}V(t,x,\mu(t))\right)\in C_{\mathcal{X}}^{\frac{1}{2},1+\alpha}\left([0,T]\times\mathbb{T}_{\mathbb{G}};\mathbb{R}^{n_1}\right)$. By Lemma \ref{Lem_general FPK wellposed regularity}, there exists a unique distributional solution $m(t) \in C\left([t_0,T];C_{\mathcal{X}}^{-(2+\alpha)}(\mathbb{T}_{\mathbb{G}})\right)$ to the equation \eqref{ME_uniqueness_FPK 2}. Define the map $\Phi(\mu):=m$, in the same way as the proof of Proposition \ref{Prop_MFG system wellposed regularity}, we can obtain that $m \in E$.

As for the continuity of $\Phi$, for any $\mu_i\to \mu$ in $E$ as $i\to+\infty$, let $m^{i}$ be the corresponding solution to the equation \eqref{ME_uniqueness_FPK 2}. Choose any $\xi\in C_\mathcal{X}^{0+1}(\mathbb{T}_{\mathbb{G}})$ with $\xi(x_0)=0$ for some $x_0\in\mathbb{T}_{\mathbb{G}}$, and let $\xi^{\varepsilon}$ be the mollified versions of $\xi$, satisfying
\begin{equation*}
\left[\xi^{\varepsilon}\right] _{C_{\mathcal{X}}^{0+1}\left(\mathbb{T}_{\mathbb{G}}\right)} \leq C\left[\xi\right] _{C_{\mathcal{X}}^{0+1}\left(\mathbb{T}_{\mathbb{G}}\right)},\,\varepsilon>0
\end{equation*}
and $\left\|\xi^{\varepsilon}-\xi\right\|_{L^{\infty}\left(\mathbb{T}_{\mathbb{G}}\right)}\to 0$ as $\varepsilon \to 0$ (see \cite[Proposition 1.2]{JWY}). It follows from Lemma \ref{Lem_LHJB well-posedness} that there exist the unique solutions $z^{\varepsilon,i},z^{\varepsilon}\in C_{{\mathcal{X}}}^{1,2}\left([t_0,T]\times\mathbb{T}_{\mathbb{G}}\right)$ to the equation
\begin{equation*}
\begin{cases}
-\partial_tz-\Delta_\mathcal{X}z+b\cdot D_{\mathcal{X}}z=0, & \text{ in } [t_0,T] \times \mathbb{T}_{\mathbb{G}}, \\
z(T)=\xi^{\varepsilon}, & \text{ in } \mathbb{T}_{\mathbb{G}}
\end{cases}
\end{equation*}
with $b=B^{i}:=D_pH\left(x,
D_{\mathcal{X}}V(t,x,\mu^{i}(t))\right)$ and $b=B:=D_pH\left(x,
D_{\mathcal{X}}V(t,x,\mu(t))\right)$ respectively. Set $\bar{m}^{i}:=m^{i}-m$ and $\bar{z}^{\varepsilon,i}:=z^{\varepsilon,i}-z^{\varepsilon}$. Same as \eqref{d_1(m(t_1),m(t_2))}, for any $t\in[t_0,T]$,
\begin{equation*}
\int_{\mathbb{T}_{\mathbb{G}}}\xi(x) d\bar{m}^{i}(t)(x) =\lim_{\varepsilon\to0}\int_{\mathbb{T}_{\mathbb{G}}}\xi^{\varepsilon}(x) d\bar{m}^{i}(t)(x) =\lim_{\varepsilon\to0}\int_{\mathbb{T}_{\mathbb{G}}}\bar{z}^{\varepsilon,i}(T-t+t_0,x)dm_0(x).
\end{equation*}
We also know that $\bar{z}^{\varepsilon,i}\in C_{{\mathcal{X}}}^{1,2} \left([t_0,T]\times\mathbb{T}_{\mathbb{G}}\right)$ is the unique solution to the equation
\begin{equation*}
\begin{cases}
-\partial_t\bar{z}^{\varepsilon,i}-\Delta_\mathcal{X}\bar{z}^{\varepsilon,i}+B^{i}\cdot D_{\mathcal{X}}\bar{z}^{\varepsilon,i} =\left(B-B^{i}\right)\cdot D_{\mathcal{X}}z^{\varepsilon}, & \text{ in } [t_0,T] \times \mathbb{T}_{\mathbb{G}}, \\
\bar{z}^{\varepsilon,i}(T)=0, & \text{ in } \mathbb{T}_{\mathbb{G}}.
\end{cases}
\end{equation*}
We use Lemma \ref{Lem_LHJB Lipschitz}, \cite[Theorem 1.3(1)]{25JWY} and \eqref{D_mathcal X V_Lip. in m} to get
\begin{align*}
\int_{\mathbb{T}_{\mathbb{G}}}\xi(x) d\bar{m}^{i}(t)(x) \leq \lim_{\varepsilon\to0}\left\|\bar{z}^{\varepsilon,i}\right\|_{L^{\infty}\left([t_0,T]\times \mathbb{T}_{\mathbb{G}}\right)} \leq & \lim_{\varepsilon\to0}C\left\|\left(B-B^{i}\right)\cdot D_{\mathcal{X}}z^{\varepsilon}\right\|_{L^{\infty}\left([t_0,T]\times \mathbb{T}_{\mathbb{G}}\right)} \\
\leq & C\sup_{t \in [t_0,T]}d_1(\mu(t),\mu^{i}(t))\left[\xi\right]_{C_\mathcal{X}^{0+1}(\mathbb{T}_{\mathbb{G}})},
\end{align*}
and for any $t_1,t_2\in[t_0,T]$ with $t_1\neq t_2$,
\begin{align*}
\int_{\mathbb{T}_{\mathbb{G}}}\xi(x) d\left(\bar{m}^{i}(t_1)-\bar{m}^{i}(t_2)\right)(x) = & \lim_{\varepsilon\to0}\int_{\mathbb{T}_{\mathbb{G}}}\left(\bar{z}^{\varepsilon,i}(T-t_1+t_0,x)-\bar{z}^{\varepsilon,i}(T-t_2+t_0,x)\right)dm_0(x) \\
\leq & \lim_{\varepsilon\to0}C\left\|\left(B-B^{i}\right)\cdot D_{\mathcal{X}}z^{\varepsilon}\right\|_{L^{\infty}\left([t_0,T]\times \mathbb{T}_{\mathbb{G}}\right)}|t_1-t_2|^{\frac{1}{2}} \\
\leq & C\sup_{t \in [t_0,T]}d_1(\mu(t),\mu^{i}(t))\left[\xi\right]_{C_\mathcal{X}^{0+1}(\mathbb{T}_{\mathbb{G}})}|t_1-t_2|^{\frac{1}{2}},
\end{align*}
where the constants $C>0$ depend on $\mathbb{G}$, $T$, $H$, $\sup_{m \in \mathcal{P}(\mathbb{T}_{\mathbb{G}})}\left\|D_{\mathcal{X}}V(\cdot,\cdot,m)\right\|_{L^{\infty}\left((0,T)\times\mathbb{T}_{\mathbb{G}}\right)}$ and $\sup_{(t,m) \in [0,T] \times \mathcal{P}(\mathbb{T}_{\mathbb{G}})}\left\|\frac{\delta V}{\delta m}\left(t, \cdot, m, \cdot\right)\right\|_{C_{\mathcal{X},y}^{1}\left(\mathbb{T}_{\mathbb{G}};C_{\mathcal{X},x}^{1}\left(\mathbb{T}_{\mathbb{G}}\right)\right)}$ only. Therefore, we obtain $\|m^{i}-m\|_{E}\to0$ as $i\to+\infty$, which concludes the proof of continuity.

According to the Schauder fixed point theorem, we obtain that there exists a solution $\tilde{m}\in C^{\frac{1}{2}}\left([t_0,T];\mathcal{P}(\mathbb{T}_{\mathbb{G}})\right)$ to the equation \eqref{ME_uniqueness_FPK 1}. Moreover, if $m_0$ has a smooth positive density, by Lemma \ref{Lem_LHJB well-posedness}, we find that $\tilde{m}\in C_{\mathcal{X}}^{1,2} \left([t_0, T] \times \mathbb{T}_{\mathbb{G}} \right)$.

Take a smooth sequence $\{m_0^{i}\}_{i=1}^{+\infty}$ to approximate $m_0$ in $\mathcal{P}\left(\mathbb{T}_{\mathbb{G}}\right)$ (see \cite[Corollary 3.1]{JWY}), and let $\tilde{m}^i$ be the corresponding solution to the equation \eqref{ME_uniqueness_FPK 1}. Now we define $\tilde{u}^i(t,x):=V(t,x,\tilde{m}^i(t))$ for any $(t,x)\in[0,T]\times\mathbb{T}_{\mathbb{G}}$. Due to the regularity of $V$, $\tilde{u}^i$ is at least of class $C_{\mathcal{X}}^{1,2}$ with
\begin{align*}
\partial_t \tilde{u}^i(t, x) = & \partial_t V(t, x, \tilde{m}^i(t))+\int_{\mathbb{T}_{\mathbb{G}}}\frac{\delta V}{\delta m}(t, x, \tilde{m}^i(t),y) \partial_t\tilde{m}^i(t,y)dy \\
= & \partial_t V(t, x, \tilde{m}^i(t))+\int_{\mathbb{T}_{\mathbb{G}}}\frac{\delta V}{\delta m}(t, x, \tilde{m}^i(t), y)\left(\Delta_{\mathcal{X}}\tilde{m}^i\right. \\
& \left.+\operatorname{div}_{\mathcal{X}}\left(\tilde{m}^i D_pH\left(y,D_{\mathcal{X}}V(t, y, \tilde{m}^i(t))\right)\right)\right)dy \\
= & \partial_t V(t, x, \tilde{m}^i(t))+\int_{\mathbb{T}_{\mathbb{G}}} \Delta_{\mathcal{X}}^y\frac{\delta V}{\delta m}(t, x, \tilde{m}^i(t), y) d \tilde{m}^i(t)(y) \\
& -\int_{\mathbb{T}_{\mathbb{G}}}D_{\mathcal{X}}^y \frac{\delta V}{\delta m}(t, x, \tilde{m}^i(t), y) \cdot D_p H\left(y, D_{\mathcal{X}} V(t, y, \tilde{m}^i(t))\right) d \tilde{m}^i(t)(y) \\
= & -\Delta_{\mathcal{X}}V(t, x, \tilde{m}^i(t))+H\left(x, D_{\mathcal{X}}V(t, x, \tilde{m}^i(t))\right)-F(x, \tilde{m}^i(t)),
\end{align*}
where the last equality follows from \eqref{ME}. Thus we obtain that $\tilde{u}^i$ satisfies the HJB equation
\begin{equation*}
-\partial_t \tilde{u}^i(t, x) -\Delta_{\mathcal{X}}\tilde{u}^i(t, x)+H\left(x, D_{\mathcal{X}}\tilde{u}^i(t, x)\right)=F(x, \tilde{m}^i(t))
\end{equation*}
with terminal condition $\tilde{u}^i(T,x)=V(T,x,\tilde{m}^i(T))=G(x,\tilde{m}^i(T))$. Therefore $(\tilde{u}^i,\tilde{m}^i)$ is a solution to the MFG system \eqref{MFG system} with initial condition $(t_0,m_0^i)$. According to the uniqueness of the solution to the MFG system, we have $(\tilde{u}^i,\tilde{m}^i)=(u^{i},m^{i})$. Therefore, $V(t_0,x,m_0)=\lim_{i\to+\infty}V(t_0,x,m_0^i)=\lim_{i\to+\infty}U(t_0,x,m_0^i)=U(t_0,x,m_0)$ for any $(t_0,x,m_0)\in[0,T]\times\mathbb{T}_{\mathbb{G}}\times \mathcal{P}(\mathbb{T}_{\mathbb{G}})$, namely the uniqueness holds.

The further regularity of the derivative $\frac{\delta U}{\delta m}$ is given by Proposition \ref{Prop_U C^1 w.r.t. m}.

We thereby finish the proof of Theorem \ref{Thm_ME wellposed regularity}.
\end{proof}

\section*{Acknowledgments}

The authors sincerely thank Professor Wilfrid Gangbo for his valuable suggestions, and are grateful to the referees for their careful reading and thoughtful comments.

Yiming Jiang is supported by National Natural Science Foundation of China (Grant No. 12471141). Yawei Wei is supported by National Natural Science Foundation of China (Grant No. 12271269) and Fundamental Research Funds for the Central Universities. All authors contributed equally to this work, and the author list is ordered alphabetically by surname.

\end{document}